\documentclass[12pt]{amsart}

\usepackage{amsthm,amsfonts,amsmath,amscd,amssymb,latexsym,
pb-diagram,amssymb,epic,eepic,verbatim,graphicx,graphics,epsfig,psfrag} 
\usepackage{mathrsfs}
\usepackage{mathdots}
\usepackage[all]{xy}
\usepackage{hyperref}
\usepackage{enumitem}

\setlist[itemize]{leftmargin=*}
\setlist[enumerate]{leftmargin=*}

\newtheorem{PARA}{}[subsection]
\newtheorem{theorem}[PARA]{Theorem}

\newtheorem{lemma}[PARA]{Lemma}

\newtheorem{definition}[PARA]{Definition}
\newtheorem{remark}[PARA]{Remark}
\newtheorem{example}[PARA]{Example}

\newtheorem{conjecture}[PARA]{Conjecture}
\numberwithin{equation}{subsection}

\newcommand\cA{\mathcal{A}}

\newcommand\cC{\mathcal{C}}
\newcommand\cD{\mathcal{D}}
\newcommand\cE{\mathcal{E}}
\newcommand\cF{\mathcal{F}}
\newcommand\cG{\mathcal{G}}
\newcommand\cH{\mathcal{H}}
\newcommand\cL{\mathcal{L}}
\newcommand\cP{\mathcal{P}}
\newcommand\cQ{\mathcal{Q}}
\newcommand\cS{\mathcal{S}}

\newcommand\cY{\mathcal{Y}}
\newcommand{\N}{\mathbb{N}}
\newcommand{\R}{\mathbb{R}}
\newcommand{\Z}{\mathbb{Z}}

\newcommand{\CP}{\mathbb{CP}}
\newcommand{\RP}{\mathbb{RP}}

\newcommand{\on}{\operatorname}
\newcommand{\Don}{\on{Don}}

\newcommand{\Cat}{\on{Cat}}
\newcommand{\Fun}{\on{Fun}}
\newcommand{\Hom}{\on{Hom}}
\newcommand{\Obj}{\on{Obj}}
\newcommand{\Mor}{\on{Mor}}
\newcommand{\SMor}{\on{SMor}}

\newcommand{\Bor}{{\rm Bor}}
\newcommand{\Man}{{\rm Man}}
\newcommand{\Symp}{{\rm Symp}}
\newcommand{\id}{{\rm id}}
\newcommand\pt{{\rm pt}}

\newcommand{\im}{{\rm im}}
\newcommand{\cg}{\mathfrak g}

\newcommand\eps{\epsilon}
\newcommand\io{\iota}

\newcommand\Om{\Omega}

\def\tint{{\textstyle\int}}
\newcommand{\uL}{\underline{L}}
\newcommand{\less}{{\smallsetminus}}

\newcommand{\ul}{\underline}

   \newcounter{qcounter}

\newcommand\squ{/\kern-.7ex/} 
\newcommand\lsqu{\backslash \kern-.7ex \backslash} 

\newcommand\quotient[2]{
        \mathchoice
            {
                \text{\raise1ex\hbox{$#1$}\big/\lower1ex\hbox{$#2$}}%
            }
            {
                #1\,/\,#2
            }
            {
                #1\,/\,#2
            }
            {
                #1\,/\,#2
            }
    }

\newcommand\quot[2]{
                \text{\raise1ex\hbox{$#1$}\big/\lower1ex\hbox{$\scriptstyle#2$}}
  }

\newcommand\quo[2]{
                \text{\raise1ex\hbox{$#1\!\!$}\big/\lower1ex\hbox{$\!\scriptstyle#2$}}
  }

\newcommand\qu[2]{
                \text{\raise.8ex\hbox{$\scriptstyle#1\!$}/\lower.8ex\hbox{$\!\scriptstyle#2$}}
  }

\newcommand\qq[2]{
                \text{\raise.8ex\hbox{$#1\!$}/\lower.8ex\hbox{$#2$}}
}

\title{Floer field philosophy}
\author{Katrin Wehrheim}
\date{}

\address{Department of Mathematics, UC Berkeley CA 94720; katrin@math.berkeley.edu}

\begin{document}

\maketitle

\begin{abstract}
Floer field theory is a construction principle for e.g.\ 3-manifold invariants via decomposition in a bordism category and a functor to the symplectic category, and is conjectured to have natural 4-dimensional extensions.
This survey provides an introduction to the categorical language for the construction and extension principles and provides the basic intuition for two gauge theoretic examples which conceptually frame Atiyah-Floer type conjectures in Donaldson theory as well as the relations of Heegaard Floer homology to Seiberg-Witten theory.
\end{abstract}

\tableofcontents

\section{Introduction} 

In the 1980s the areas of low dimensional topology and symplectic geometry both saw important progress arise from the study of moduli spaces of solutions of nonlinear elliptic PDEs.
In the study of smooth four-manifolds, Donaldson \cite{donaldson} introduced the use of 
ASD Yang-Mills instantons\footnote{
A smooth four manifold can be thought of as a curved $4$-dimensional space-time. ASD (anti-self-dual)
instantons in this space-time satisfy a reduction of Maxwell's equations for the electro-magnetic potential in vacuum, which has an infinite dimensional gauge symmetry.
}, which were soon followed by Seiberg-Witten equations \cite{morgan} --  another gauge theoretic\footnote{In mathematics, ``gauge theory''  refers to the study of connections on principal bundles, where ``gauge symmetries'' arise from the pullback action by bundle isomorphisms; see e.g. \cite[App.A]{W:book}.} PDE.
In the study of symplectic manifolds, Gromov \cite{G} introduced pseudoholomorphic curves\footnote{Symplectic manifolds can be thought of as the configuration spaces of classical mechanical systems, with the position-momentum pairing providing the symplectic structure as well as a class of almost complex structures $J$. 
Pseudoholomorphic curves can then be thought of as $2$-dimensional surfaces in a $2n$-dimensional symplectic ambient space, which can be locally described as the image of $2n$ real valued functions $\ul u$ of a complex variable $z=x+iy$ that satisfy a generalized Cauchy-Riemann equation $\partial_x \ul u = J(u) \, \partial_y \ul u$.} In both subjects Floer \cite{Floer:inst,Floer:Lag} then introduced a new approach to infinite dimensional Morse theory\footnote{
Morse theory captures the topological shape of a space by studying critical points of a function and flow lines of its gradient vector field. In finite dimensions it yields a complex whose homology is independent of choices (e.g.\ of function) and in fact equals the singular homology of the space.
 } 
based on the respective PDEs. 
This sparked the construction of various algebraic structures -- such as the Fukaya $A_\infty$-category of a symplectic manifold \cite{seidel}, a Chern-Simons field theory for 3-manifolds and 4-cobordisms 
\cite{Don:book}, and analogous Seiberg-Witten 3-manifold invariants \cite{KM} -- from these and related PDEs, which encode significant topological information on the underlying manifolds.
Chern-Simons field theory in particular comprises the Donaldson invariants of 4-manifolds, together with algebraic tools to calculate these by decomposing a closed 4-manifold into 4-manifolds whose common boundary is given by a 3-dimensional submanifold. 
This strategy of decomposition into simpler pieces inspired the new topic of ``topological (quantum) field theory'' \cite{atiyah:field,lurie,segal,witten:field}, in which the properties of such theories are described and studied.

In trying to extend the field-theoretic strategy to the decomposition of 3-manifolds along 2-dimensional submanifolds, Floer and Atiyah \cite{Atiyah} realized a connection to symplectic geometry: A degeneration of the ASD Yang-Mills equation on a 4-manifold with 2-dimensional fibers $\Sigma$ yields the Cauchy-Riemann equation on a (singular) symplectic manifold $M_\Sigma$ given by the flat connections on $\Sigma$ modulo gauge symmetries.
Along with this, 3-dimensional handlebodies $H$ with boundary $\partial H=\Sigma$ induce Lagrangian submanifolds $L_H\subset M_\Sigma$ given by the boundary restrictions of flat connections on $H$.
Now Lagrangians\footnote{Throughout this paper, the term ``Lagrangian'' refers to a half-dimensional isotropic submanifold of a symplectic manifold -- corresponding to fixing the integrals of motion, e.g.\ the momentums.}
are the most fundamental topological object studied in symplectic geometry. They are often studied by means of the Floer homology $HF(L_0,L_1)$ of pairs of Lagrangians, which arises from a complex that is generated by the intersection points $L_0\cap L_1$ and whose homology is invariant under Hamiltonian deformations of the Lagrangians.
For the pair $L_{H_0},L_{H_1} \subset M_\Sigma$ arising from the splitting $Y=H_0\cup_\Sigma H_1$ of a 3-manifold into two handlebodies $H_0,H_1$, these generators are naturally identified with the generators of the instanton Floer homology $HF_{\rm inst}(Y)$, given by flat connections on $Y$ modulo gauge symmetries. (Indeed, restricting the latter to $\Sigma\subset Y$ yields a flat connection on $\Sigma$ that extends to both $H_0$ and $H_1$ -- in other words, an intersection point of $L_{H_0}$ with $L_{H_1}$.)

These observations inspired the {\bf Atiyah-Floer conjecture}
$$
HF_{\rm inst}(H_0\cup_\Sigma H_1) \simeq HF(L_{H_0},L_{H_1}),
$$ 
which asserts an equivalence between the differentials on the Floer complexes -- arising from ASD instantons on $\R\times Y$ and pseudoholomorphic maps $\R\times[0,1]\to M_\Sigma$ with boundary values on $L_{H_0},L_{H_1}$, respectively.
While this conjecture is not well defined due to singularities in the symplectic manifolds $M_\Sigma$, 
and the proof of a well defined version by Dostoglou-Salamon \cite{DS} required hard adiabatic limit analysis,
the underlying ideas sparked inquiry into relationships between low dimensional topology and symplectic geometry.
At this point, the two fields are at least as tightly intertwined as algebraic and symplectic geometry (via mirror symmetry), most notably through the {\bf Heegaard-Floer invariants for 3- and 4-manifolds} (as well as knots and links), which were discovered by Ozsvath-Szabo \cite{OS1} by following the line of argument of Atiyah and Floer in the case of Seiberg-Witten theory.
In both cases the concept for the construction of an invariant of 3-manifolds $Y$ is the same:
\begin{enumerate}
\item
Split $Y=H_0\cup_\Sigma H_1$ along a surface $\Sigma$ into two handlebodies $H_i$ with $\partial H_i=\Sigma$. 
\item
Represent the dividing surface $\Sigma$ by a symplectic manifold $M_\Sigma$ and the two handlebodies by Lagrangians $L_{H_i}\subset M_\Sigma$ arising from dimensional reductions of a gauge theory which is known to yield topological invariants.
\item
Take the Lagrangian Floer homology $HF(L_{H_0},L_{H_1})$ of the pair of Lagrangians.
\item
Argue that different splittings yield isomorphic Floer homology groups -- due to an isomorphism to a gauge theoretic invariant of $Y$ or by direct symplectic isomorphisms 
$HF(L_{H_0},L_{H_1})\simeq HF(L_{\widetilde H_0},L_{\widetilde H_1})$
for different splittings ${Y=\widetilde H_0\cup_{\widetilde \Sigma} \widetilde H_1}$.
\end{enumerate}
{\bf Floer field theory} is an extension of this approach to more general decompositions of 3-manifolds, by phrasing Step 4 above as the existence of a functor between topological and symplectic categories that extends the association 
$$
\Sigma\mapsto M_\Sigma ,\qquad\qquad H \mapsto L_H, \qquad\qquad \partial H = \Sigma \;\Rightarrow\; L_H\subset M_\Sigma .
$$
It gives a conceptual explanation for Step 4 invariance proofs such as \cite{OS1} which bypass a comparison to the gauge theory by directly relating the Floer homologies of Lagrangians $L_{H_0},L_{H_1} \subset M_\Sigma$ and
$L_{\widetilde H_0},L_{\widetilde H_1}\subset M_{\widetilde \Sigma}$.
Since these can arise from surfaces $\Sigma\not\simeq\widetilde \Sigma$ of different genus, the comparison between pseudoholomorphic curves in symplectic manifolds $M_\Sigma\not\simeq M_{\widetilde \Sigma}$ of different dimension must crucially use the fact that the Lagrangian boundary conditions encode different splittings of the same 3-manifold.
Floer field theory encodes this as an isomorphism between algebraic compositions of the Lagrangians, which in turn yields isomorphic Floer homologies
(a strategy that we elaborate on in \S\ref{ss:field} and \S\ref{ss:symp2}),
\begin{align*}
H_0\cup_\Sigma H_1 \simeq \widetilde H_0\cup_{\widetilde \Sigma} \widetilde H_1
\quad\Longrightarrow\quad &
L_{H_0}\# L_{H_1} \sim L_{\widetilde H_0}\# L_{\widetilde H_1} \\
\quad\Longrightarrow\quad &
HF(L_{H_0},L_{H_1}) \simeq HF(L_{\widetilde H_0}, L_{\widetilde H_1}).
\end{align*}
Floer field theory, in particular its key isomorphism of Floer homologies \cite{ww:isom} hinted at above, was discovered by the author and Woodward \cite{ww:fielda,ww:fieldb} when attempting to formulate well defined versions of the Atiyah-Floer conjecture. While the isomorphism of Floer homologies in \cite{ww:isom} is usually formulated in terms of strip-shrinking in a new notion of quilted Floer homology \cite{ww:qhf}, it can be expressed purely in terms of Floer homologies of pairs of Lagrangians, which lie in different products of symplectic manifolds. 
In this language, strip shrinking then is a degeneration of the Cauchy-Riemann operator to a limit in which the curves in one factor of the product of symplectic manifolds become trivial. (For more details, see \S\ref{ss:symp2}.)
This relation between pseudoholomorphic curves in different symplectic manifolds then provides a purely symplectic analogue of the adiabatic limit in \cite{DS}, which relates ASD instantons to pseudoholomorphic curves.

\smallskip
{\bf The value of Floer field theory to 3-manifold topology is mostly of philosophical nature} -- giving a conceptual understanding for invariance proofs and a general construction principle for 3-manifold invariants (and similarly for knots and links), which has since been applied in a variety of contexts \cite{auroux_hf,lekili,MWcr,reza,ww:fielda,ww:fieldb}. One main purpose of this paper and the content of \S\ref{s:fft} is to explain this philosophy and cast the construction principle into rigorous mathematical terms.
For that purpose 
\S\ref{ss:cat} gives brief expositions of the notions of categories and functors, the category $\Cat$, and bordism categories $\Bor_{d+1}$. 
After introducing the symplectic category $\Symp$ in \S\ref{ss:symp}, the categorical structure in symplectic geometry that can be related to low dimensional topology, in \S\ref{ss:cerf} we cast the concept of Cerf decompositions (cutting manifolds into simple cobordisms) into abstract categorical terms that apply equally to bordism categories and our construction of the symplectic category.
We then exploit the existence of Cerf decompositions in $\Bor_{d+1}$ and $\Symp$ together with a Yoneda functor $\Symp\to\Cat$ (see Lemma~\ref{le:sympcat}) to formulate a general construction principle for Floer field theories. This notion of Floer field theory is defined in \S\ref{ss:field} as a functor $\Bor_{d+1}\to\Cat$ that factors through $\Symp$.
This construction is exemplified in \S\ref{ss:ex} by naive versions of two gauge theoretic examples related to Yang-Mills-Donaldson resp.\ Seiberg-Witten theory in dimensions 2+1. Finally, \S\ref{ss:af} explains how this yields conjectural symplectic versions of the gauge theoretic 3-manifold invariants, as predicted by Atiyah and Floer.

\medskip

The second purpose of this paper and content of \S\ref{s:ext} is to lay some foundations for an {\bf extension of Floer field theories to dimension 4}.
Our goal here is to provide a rigorous exposition of the algebraic language in which this extension principle can be formulated -- at a level of sophistication that is easily accessible to geometers while sufficient for applications.
Thus we review in detail the notions of 2-categories and bicategories in \S\ref{ss:2cat}, including the 2-category $\Cat$ of (categories, functors, natural transformations), explicitly construct a bordism bicategory $\Bor_{2+1+1}$ in \S\ref{ss:bord}, and summarize notions and Yoneda constructions of 2-functors between these higher categories in \S\ref{ss:funk}.
Moreover, \S\ref{ss:symp2} outlines the construction of symplectic 2-categories, based on abstract categorical notions of adjoints and quilt diagrams that we develop in \S\ref{ss:quilt}. The latter transfers notions of adjunction and spherical string diagrams from monoidal categories into settings without natural monoidal structure.
This provides sufficient language to at least advertise an extension principle
which we further discuss in \cite{W:ext}:

\smallskip
\noindent
{\it 
Any Floer field theory $\Bor_{2+1}\to\Symp\to\Cat$ which satisfies a quilted naturality axiom has a natural extension to a 2-functor $\Bor_{2+1+1}\to\Symp\to\Cat$.
}
\smallskip

This says in particular that any 3-manifold invariant which is constructed along the lines of the Atiyah-Floer conjecture naturally induces a 4-manifold invariant. 
While it does seem surprising, such a result could be motivated from the point of view of gauge theory, since the Atiyah-Floer conjecture and Heegaard-Floer theory were inspired by dimensional reductions of 3+1 field theories $\Bor_{3+1}\to\cC$.
It also can be viewed as a pedestrian version of the cobordism hypothesis \cite{lurie}\footnote{
Lurie's constructions involve the canonical extension of a functor $\Bor_{0+1+\ldots+\eps}\to\cC$ to $\Bor_{0+1+\ldots+1}\to\cC$. However, this requires an extension of the field theory to dimensions 1 and 0 (which we do not even have ideas for) as well as a monoidal structure on the target category $\cC$ (which is lacking at present because the gauge theoretic functors are well defined only on the connected bordism category).
On the other hand, we have other categorical structures at our disposal, which we formalize in \S\ref{ss:quilt} as the notion of a quilted 2-category (akin to a spherical 2-category as described in \cite{spherical}).
In that language, the diagram of a Morse 2-function as in \cite{GayKirby} expresses a 4-manifold as a quilt diagram in the bordism bicategory $\Bor_{2+1+1}$. Now the key idea for \cite{W:ext} is that a functor $\Bor_{2+1}\to\Symp$ translates the diagram of a 4-manifold into a quilt diagram in the symplectic 2-category in \S\ref{ss:symp2}, where it is reinterpreted in terms of pseudoholomorphic curves.}, saying that a functor $\Bor_{2+1+\eps}\to\Symp$ (where the $\epsilon$ stands for compatibility with diffeomorphisms of 3-manifolds) has a canonical extension $\Bor_{2+1+1}\to\Symp$.

Finally, the extensions of Floer field theory to dimension 4 are again expected to be isomorphic to the associated gauge theoretic 4-manifold invariants, in a way that is compatible with decomposition into 3- and 2-manifolds. We phrase these expectations in  \S\ref{ss:gauge2} as {\bf quilted Atiyah-Floer conjectures}, which identify field theories $\Bor_{2+1+1}\to\cC$.
The last section \S\ref{ss:gauge2} also demonstrates the construction principle for 2-categories via associating elliptic PDEs to quilt diagrams in several more gauge theoretic examples, which provide not only the proper context for stating all the generalized Atiyah-Floer conjectures, but also yield conceptually clear contexts for the various approaches to their proofs.

\medskip

While this 2+1 field theoretic circle of ideas has been and used in various publications, its rigorous abstract formulation in terms of a notion of ``category with Cerf decompositions'' is new to the best of the author's knowledge.
Similarly, the notions of bordism bicategories, the symplectic 2-category, generalized string diagrams, and field theoretic proofs of Floer homology isomorphisms have been known and (at least implicitly) used in similar contexts, but are here cast into a new concept of ``quilted bicategories'' which will be central to the extension principle -- both of which seem significantly beyond the known circle of ideas.
Finally, note that Floer field theory should not be confused with the symplectic field theory (SFT) introduced by \cite{egh}, in which another symplectic category -- given by contact-type manifolds and symplectic cobordisms -- is the domain, not the target of a functor.

We end this introduction by a more detailed explanation of the notion of an ``invariant'' as it applies to the study of topological or smooth compact manifolds, and a very brief introduction to the resulting classification of manifolds.

\subsection{A brief introduction to invariants of manifolds} \label{sec:invariant}
In order to classify manifolds of a fixed dimension $n$ up to diffeomorphism, one would ideally like to have a complete invariant $I: \Man_n \to \cC$. 
Here $\Man_n$ is the category of $n$-manifolds and diffeomorphisms between them (see Example~\ref{ex:man}), and  $\cC$ is a category such as $\cC=\Z$ with trivial morphisms or the category $\cC={\rm Gr}$ of groups and homomorphisms. Such $I$ is an invariant if it is a functor (see Definition~\ref{def:functor}), since this guarantees that diffeomorphic manifolds are mapped to isomorphic objects of $\cC$ (e.g.\ the same integer or isomorphic groups). 
In other words, functoriality guarantees that $I$ induces a well defined map $|I|:|\Man_n| \to |\cC|$ from diffeomorphism classes of manifolds to e.g.\ $\Z$ or isomorphism classes of groups.
Such an invariant lets us distinguish manifolds: If $I(X), I(Y)$ are not isomorphic (i.e.\ $|I|([X])\neq |I|([Y])$) then $X$ and $Y$ cannot be diffeomorphic.
Moreover, an invariant is called ``complete'' if an isomorphism $I(X)\simeq I(Y)$ implies the existence of a diffeomorphism $X\simeq Y$, i.e.\ $|I|([X])= |I|([Y]) \Rightarrow [X]=[Y]$.

Simple examples of invariants -- when restricting $\Man_n$ to compact oriented manifolds -- are the homology groups $H_k:\Man_n\to {\rm Groups}$ for fixed $k\in\N_0$ or their rank, i.e.\ the Betti numbers $\beta_k: \Man_n\to\Z$. These are in fact topological -- rather than smooth -- invariants since homeomorphic -- rather than just diffeomorphic -- manifolds have isomorphic homology groups.
The 0-th Betti number $\beta_0$ is complete for $n=0,1$ since it determines the number of connected components, and there is only one compact, connected manifold of dimension 0 (the point) or 1 (the circle). The first more nontrivial complete invariant -- now also restricting to connected manifolds -- is the first Betti number $\beta_1: \Man_2\to\Z$, since compact, connected, oriented 2-manifolds are determined by their genus $g=\frac 12 \beta_1$. 

The fundamental group $\pi_1:\Man_n\to {\rm Gr}$ is not strictly well defined since it requires the choice of a base point and thus is a functor on the category of manifolds with a marked point. However, for connected manifolds it still induces a well defined map $|\pi_1|: |\Man_n|\to|{\rm Gr}|$ from manifolds modulo diffeomorphism to groups modulo isomorphism, since change of base point induces an isomorphism of fundamental groups. Viewing this as an invariant, it is complete for $n=0,1,2$. In dimension $n=3$, completeness would mean that the isomorphism type of the fundamental group of a (compact, connected) 3-manifold determines the 3-manifold up to diffeomorphism. 
This is true in the case of the trivial fundamental group: By the Poincar\'e conjecture, any simply connected 3-manifold ``is the 3-sphere'', i.e.\ is diffeomorphic to $S^3$.
It is also true for a large class (irreducible, non-spherical) of 3-manifolds, but there are plenty of groups that can be represented by many non-diffeomorphic 3-manifolds, e.g.\ lens spaces and connected sums with them (see \cite{hatcher,AFW} for surveys).
Thus $|\pi_1|$ is a useful but incomplete invariant of closed, connected 3-manifolds. 
In dimension $n\ge 4$ however, the classification question should be posed for fixed $|\pi_1|$ since on the one hand any finitely presented group appears as the fundamental group of a closed, connected $n$-manifold, and on the other hand  the classification of finitely presented groups is a wide open problem itself.\footnote{
As a matter of curiosity: The group isomorphism problem -- determining whether different finite group presentations define isomorphic groups -- is undecidable, i.e.\ cannot be solved for all general presentations by an algorithm; see e.g.\
\cite{johnson}.}

Moreover, while in dimension $n\leq 3$, the classifications up to homeomorphism and up to diffeomorphism coincide (i.e.\ topological $n$-manifolds can be equipped with a unique smooth structure), these differ in dimensions $n\geq 4$. 
In dimension $n\ge 5$, both classifications can be undertaken with the help of surgery theory introduced by Milnor \cite{milnor}. In dimension 4, the classification of smooth 4-manifolds differs drastically from that of topological manifolds (see \cite{wild} for a survey). Here gauge theory -- starting with the work of Donaldson, and continuing with Seiberg-Witten theory -- is the main source of invariants which can differentiate between different smooth structures on the same topological manifold.
In particular, Donaldson's first results using ASD Yang-Mills instantons \cite{donaldson} showed that a large number of topological manifolds (those with non-diagonalizable definite intersection form $H_2(X;\Z)\times H_2(X;\Z)\to\Z$) in fact do not support any smooth structure.

\section{Floer field theory} \label{s:fft}

\subsection{Categories and functors} \label{ss:cat}

\begin{definition}
A {\bf category}\footnote{
Throughout, all categories are meant to be small, i.e.\ consist of sets of objects and morphisms. However, we will usually neglect to specify constructions in sufficient detail -- e.g.\ require manifolds to be submanifolds of some $\R^N$ -- in order to obtain sets.
}
 $\cC$ consists of
\begin{itemize}
\item
a set $\Obj_\cC$ of {\bf objects},
\item
for each pair $x_1,x_2\in \Obj_\cC$ a set of {\bf morphisms} $\Mor_\cC(x_1,x_2)$,
\item
for each triple $x_1,x_2,x_3\in \Obj_\cC$ a {\bf composition map} 
$$
\Mor_\cC(x_1,x_2)\times \Mor_\cC(x_2,x_3)\to \Mor_\cC(x_1,x_3),\quad 
(f_{12},f_{23}) \mapsto f_{12}\circ f_{23} ,
$$
\end{itemize}
such that
\begin{itemize}
\item
composition is {\it associative}, i.e.\ we have $\bigl(f_{12}\circ f_{23}\bigr) \circ f_{34} = f_{12}\circ  \bigl( f_{23} \circ f_{34}\bigr)$ for any triple of composable morphisms $f_{12}, f_{23}, f_{34}$,
\item
composition has {\it identities}, i.e.\ for each $x\in\Obj_\cC$ there exists a unique\footnote{
Note that uniqueness follows immediately from the defining properties: If $\id'_x$ is another identity morphism then we have $\id'_x = \id'_x\circ \id_x = \id_x$.
} 
morphism $\id_x\in\Mor_\cC(x,x)$ such that $\id_x\circ f = f$ and $g\circ \id_x= g$ hold for any $f\in\Mor_\cC(x,y)$ and $g\in\Mor_\cC(y,x)$.
\end{itemize}
\end{definition}

The very first example of a category consists of objects that are sets (possibly with extra structure such as a linear structure, metric, or smooth manifold structure), morphisms that are maps (preserving the extra structure), composition given by composition of maps, and identities given by the identity maps.
The following bordism categories contain more general morphisms, which are more rigorously constructed in Remark~\ref{rmk:1epsbor}.

\begin{figure}[!h]
\centering
\includegraphics[width=5.5in]{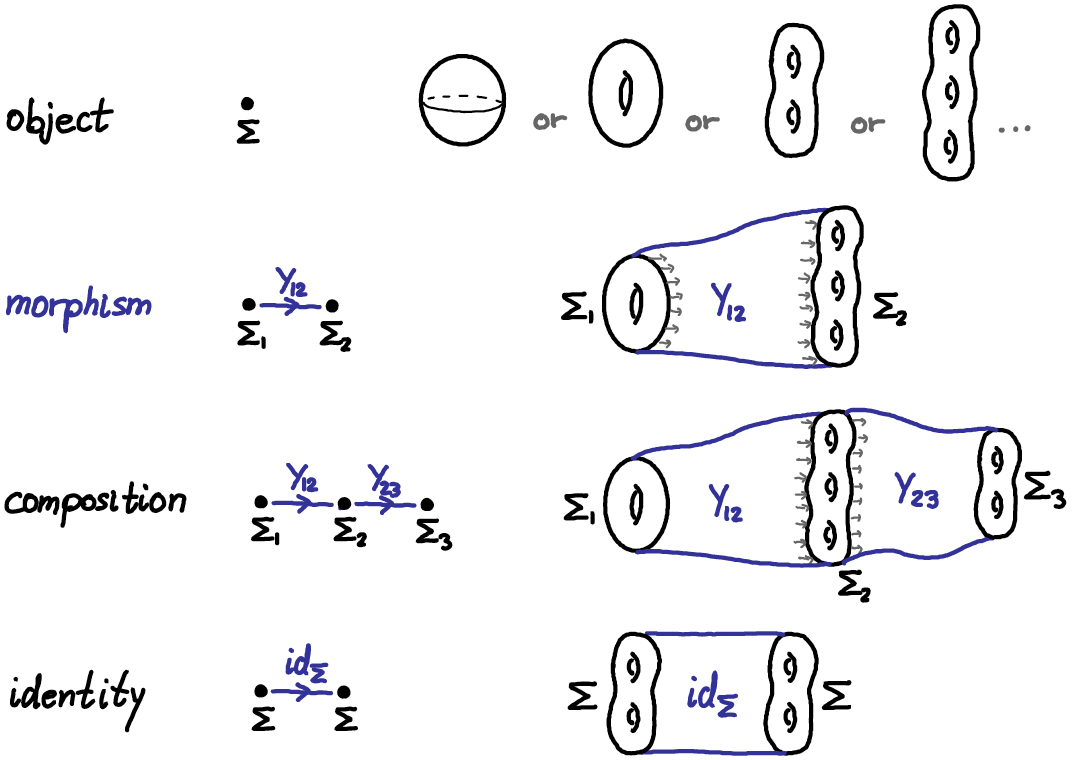}
\caption{The bordism category $\Bor_{2+1}$.}
\label{fig:bor}
\end{figure}

\begin{example}\rm \label{ex:1bor}
The {\bf bordism category ${\rm \mathbf{Bor}}_{\mathbf{d+1}}$} in dimension $d\ge 0$ is roughly defined as follows; see Figure~\ref{fig:bor} for illustration.
\begin{itemize}
\item
Objects are the closed, oriented, $d$-dimensional manifolds $\Sigma$.
\item
Morphisms in $\Mor(\Sigma_1,\Sigma_2)$ are the compact, oriented, $(d+1)$-dimensional cobordisms $Y$ with identification of the boundary $\partial Y \simeq \Sigma_1^- \sqcup \Sigma_2$, modulo diffeomorphisms relative to the boundary.
\item
Composition of morphisms $[Y_{12}]\in\Mor(\Sigma_1,\Sigma_2)$ and $[Y_{23}]\in\Mor(\Sigma_2,\Sigma_3)$ is given by gluing $[Y_{12}]\circ[Y_{23}]:= [Y_{12}\cup_{\Sigma_2} Y_{23}]\in\Mor(\Sigma_1,\Sigma_3)$ along the common boundary.
\end{itemize}
Here one needs to be careful to include the choice of boundary identifications in the notion of morphism. Thus a diffeomorphism $\phi:\Sigma_0\to\Sigma_1$ can be cast as a morphism 
\begin{equation}\label{eq:Zphi}
Z_\phi := \bigl[ \bigl( [0,1]\times \Sigma_1 , \{0\}\times\phi , \{1\}\times\id_{\Sigma_1} \bigr) \bigr]
\in\Mor_{\Bor_{d+1}}(\Sigma_0,\Sigma_1) 
\end{equation}
given by the cobordism
$[0,1]\times \Sigma_1$ with boundary identifications $\{0\}\times\phi: \Sigma_0\to\{0\}\times\Sigma_1$ and $\{1\}\times\id_{\Sigma_1}: \Sigma_1\to\{1\}\times\Sigma_1$, as illustrated in Figure~\ref{fig:cyl}. In that sense, the identity morphisms $\id_\Sigma=Z_{\id_\Sigma}$ are given by the identity maps $\id_\Sigma:\Sigma\to\Sigma$.

Equipping the composed morphism $[Y_{12}]\circ[Y_{23}]$ with a smooth structure moreover requires a choice of tubular neighbourhoods of $\Sigma_2$ in the gluing operation. The good news is that gluing with respect to different choices yields diffeomorphic results, so that composition is well defined.
The interesting news is that this ambiguity in the composition precludes the extension to a 2-category; see Example~\ref{ex:2bor}.
\end{example}

The notion of categories becomes most useful in the notion of a functor relating two categories, since preservation of various structures (composition and identities) can be expressed efficiently as ``functoriality''.

\begin{definition} \label{def:functor}
A {\bf functor} $\cF:\cC\to\cD$ between two categories $\cC,\cD$ consists of
\begin{itemize}
\item
a map $\cF:\Obj_\cC\to \Obj_\cD$ between the sets of objects,
\item
for each pair $x_1,x_2\in \Obj_\cC$ a map $\cF_{x_1,x_2}:\Mor_\cC(x_1,x_2) \to \Mor_\cD(\cF(x_1),\cF(x_2))$,
\end{itemize}
that are compatible with identities and composition, i.e.\  
$$
\id_{\cF(x)} = \cF_{x,x}(\id_x)
\qquad\qquad
\cF_{x_1,x_3}(f_{12}\circ f_{23}) = \cF_{x_1,x_2}(f_{12}) \circ \cF_{x_2,x_3}(f_{23}).
$$
\end{definition}

For example, the inclusion of diffeomorphisms into the bordism category in Example~\ref{ex:1bor} can be phrased as a functor as follows.

\begin{example}\rm  \label{ex:man}
Let $\Man_d$ be the category consisting of the same objects as $\Bor_{d+1}$, morphisms given by diffeomorphisms, and composition given by composition of maps.
Then there is a functor $\Man_d \to \Bor_{d+1}$ given by
\begin{itemize}
\item
the identity map between the sets of objects,
\item
for each pair $\Sigma_0,\Sigma_1$ of diffeomorphic $d$-manifolds
the map $\Mor_{\Man_d}(\Sigma_0,\Sigma_1) \to \Mor_{\Bor_{d+1}}(\Sigma_0,\Sigma_1)$ that associates to a diffeomorphism $\phi$ the cobordism $Z_\phi$ defined in \eqref{eq:Zphi}.
\end{itemize}
\end{example}

A more algebraic example of a category is given by categories and functors.

\begin{example}\rm \label{ex:1cat}
The {\bf category of categories} ${\rm Cat}$ consists of
\begin{itemize}
\item
objects given by categories $\cC$,
\item
morphisms in $\Mor_{\rm Cat}(\cC_1,\cC_2)$ given by functors $\cF_{12}:\cC_1\to\cC_2$,
\item
composition of morphisms given by composition of functors -- i.e.\ composition of the maps on both object and morphism level.
\end{itemize}
\end{example}

\subsection{The symplectic category} \label{ss:symp}

The vision of Alan Weinstein \cite{weinstein} was to construct a symplectic category along the following lines. (See \cite{cannas,ms} for introductions to symplectic topology.)

\begin{itemize}
\item
Objects are the symplectic manifolds $M:=(M,\omega)$.
\item
Morphisms are the Lagrangian submanifolds\footnote{
Other terms for a Lagrangian, viewed as a morphism $M_1 \to M_2$, are ``Lagrangian relation'' or ``Lagrangian correspondence'', but we will largely avoid such distinctions in this paper.
} 
$L\subset M_1^-\times M_2$, where we denote by $M_1^-:=(M_1,-\omega_1)$ the same manifold with reversed symplectic structure.
\item
Composition of morphisms $L_{12}\subset M_1^-\times M_2$ and $L_{23}\subset M_2^-\times M_3$ is defined by the {\bf geometric composition} (where $\Delta_M\subset M\times M^-$ denotes the diagonal)
$$
L_{12}\circ L_{23}:= {\rm pr}_{M_1^-\times M_3}\bigl(L_{12}\times L_{23} \, \cap  \, M_1^-\times \Delta_{M_2}\times M_3 \bigr) \subset M_1^-\times M_3.
$$
\end{itemize}
This notion includes symplectomorphisms $\phi:M_1\to M_2$, $\phi^*\omega_2=\omega_1$ as morphisms given by their graph ${\rm gr}(\phi)=\{(x,\phi(x))\,|\,x\in M_1\}\subset M_1^-\times M_2$. Also, geometric composition is defined exactly so as to generalize the composition of maps. That is, we have ${\rm gr}(\phi) \circ {\rm gr}(\psi) = {\rm gr}(\psi\circ\phi)$.
On the other hand, this more generalized notion allows one to view pretty much all constructions in symplectic topology as morphisms -- for example, symplectic reduction from $\CP^2$ to $\CP^1$ is described by a Lagrangian $3$-sphere $\Lambda \subset (\CP^2)^-\times\CP^1$; see \cite{gu:rev,weinstein,ww:qhf} for details and more examples.

Unfortunately, geometric composition generally -- even after allowing for perturbations (e.g.\ isotopy through Lagrangians) -- at best yields immersed or multiply covered Lagrangians.\footnote{
Even the question of finding a Lagrangian $L\subset\CP^2$ with embedded composition $L\circ\Lambda \subset \CP^1$ was open until the recent construction of a new Lagrangian embedding $\RP^2\hookrightarrow L\subset\CP^2$ in \cite{c-lag}.}
However, Floer homology\footnote{Floer homology is a central tool in symplectic topology introduced by Floer \cite{Floer:Lag} in the 1980s, inspired by Gromov \cite{G} and Witten \cite{witten:morse}. It has been extended to a wealth of algebraic structures such as Fukaya categories; see e.g.\ \cite{seidel}.
It can be thought of as the Morse homology of a symplectic action functional on the space of paths connecting two Lagrangians, and recasts the ill posed gradient flow ODE as a Cauchy-Riemann PDE (whose solutions are pseudoholomorphic curves).
} 
is at most expected to be invariant under {\bf embedded geometric composition}, i.e.\ when the intersection in
\begin{equation}\label{eq:embedded}
{\rm pr}_{M_1^-\times M_3}: \; L_{12}\times_{M_2} L_{23}:= L_{12}\times L_{23} \, \cap  \, M_1^-\times \Delta_{M_2}\times M_3 \; \longrightarrow\;  M_1^-\times M_3
\end{equation}
is transverse, and the projection is an embedding.
In the linear case -- for symplectic vector spaces and linear Lagrangian subspaces -- 
this issue was resolved in \cite{gu:rev} by observing that linear composition, even if not transverse, always yields another Lagrangian subspace. In higher generality, and compatible with Floer homology, a symplectic category $\Symp=\Symp^\#/\!\sim\,$ was constructed in \cite{ww:cat} by the following general algebraic completion construction for a partially defined composition.

\begin{definition}\label{def:extsymp}
The {\bf extended symplectic category} $\Symp^\#$ is defined as follows.
\begin{itemize}
\item
Objects are the symplectic manifolds $(M,\omega)$.
\item
Simple morphisms $L_{12}\in\SMor(M_1,M_2)$ are the Lagrangian submanifolds $L_{12}\subset M_1^-\times M_2$.
\item
General morphisms $\uL=(L_{01},\ldots,L_{(k-1)k})\in\Mor_{\Symp^\#}(M,N)$ are the composable chains of simple morphisms $L_{ij}\in \SMor(M_i, M_j)$ between symplectic manifolds $M=M_0, M_1, \ldots, M_k=N$.
\item
Composition of morphisms $\uL=(L_{01},\ldots,L_{(k-1)k})\in\Mor_{\Symp^\#}(M,N)$ and $\uL'=(L'_{01},\ldots,L'_{(k'-1)k'})\in\Mor_{\Symp^\#}(N,P)$ is given by algebraic concatenation $\uL \# \uL' := (L_{01},\ldots,L_{(k-1)k},L'_{01},\ldots,L'_{(k'-1)k'})$.
\end{itemize}
For this to form a strict category, we include trivial chains $(\;)\in\Mor(M,M)$ of length $k=0$ as identity morphisms.
\end{definition}

While this is a well defined category, its composition notion is not related to geometric composition yet. However, the following quotient construction ensures that composition is given by geometric composition when the result is embedded.

\begin{definition}\label{def:1symp}
The {\bf symplectic category} $\Symp$ is defined as follows.
\begin{itemize}
\item
Objects are the symplectic manifolds $(M,\omega)$.
\item
Morphisms are the equivalence classes in $\Mor_{\Symp}(M,N):=\Mor_{\Symp^\#}(M,N)/\!\sim$. 
\item
Composition $[\uL]\circ[\uL']:=[\uL \# \uL']$ is induced by the composition in $\Symp^\#$.
\end{itemize}
Here the composition-compatible equivalence relation $\sim$ on the morphism spaces of $\Symp^\#$ is obtained as follows.
\begin{itemize}
\item
The subset of {\bf geometric composition moves} ${\rm Comp}\subset \Mor_{\Symp^\#} \times \Mor_{\Symp^\#}$ consists of all pairs $\bigl( (L_{12},L_{23}), L_{12}\circ L_{23}\bigr)$ and $\bigl( L_{12}\circ L_{23}, (L_{12},L_{23}) \bigr)$ for which the geometric composition $L_{12}\circ L_{23}$ is embedded as in \eqref{eq:embedded}.
\item
The equivalence relation $\sim$ on $\Mor_{\Symp^\#}$ is defined by $\uL\sim\tilde{\uL}$ 
if there is a finite sequence of moves
$\uL \leadsto \uL' \leadsto \uL'' \ldots \leadsto \uL^{(N)}=\tilde{\uL}$
in which each move replaces one subchain of simple morphisms by another, 
\begin{align*}
 \uL^{(k)}=\bigl( \ldots , L_{ij}, L_{jl}, \ldots \bigr) & \;\leadsto\; \uL^{(k+1)}= \bigl( \ldots, L_{ij}\circ L_{jl}, \ldots\bigr) \\
\quad\mbox{resp.}\quad
\uL^{(k)}=\bigl( \ldots, L_{ij}\circ L_{jl}, \ldots\bigr) & \; \leadsto\; \uL^{(k+1)}= \bigl( \ldots , L_{ij}, L_{jl}, \ldots \bigr)
\end{align*}
according to a geometric composition move
$\bigl( (L_{ij}, L_{jl}) , L_{ij}\circ L_{jl} \bigr) \in {\rm Comp}$ resp.\ 
$\bigl( L_{ij}\circ L_{jl},  (L_{ij}, L_{jl}) \bigr) \in {\rm Comp}$.
\end{itemize}
\end{definition}

The result of this quotient construction is that the composition of morphisms is given by geometric composition  ${[L_{12}]\circ[L_{23}]= [L_{12}\circ L_{23}]}$ if the latter is embedded.
We will later recast this construction in terms of an extension of the symplectic category $\Symp^\#$  to a 2-category in which the equivalence relation $\sim$ is obtained from 2-isomorphisms; see Example~\ref{ex:ext0} and \S\ref{ss:symp2}.

\begin{remark} \rm 
The present equivalence relation does not identify a Lagrangian $L\subset M^-\times N$ with its image $\phi_H(L)\subset M^-\times N$ under a Hamiltonian symplectomorphism $\phi_H$. Indeed, any morphism $\ul L$ in $\Mor_{\Symp^\#}(M,N)$ induces a (Lagrangian where immersed) subset of $M^-\times N$ by complete geometric composition, and this subset is invariant under geometric composition moves.
However, such equivalences under Hamiltonian deformation can also be cast as 2-isomorphisms; see Example~\ref{ex:2symp}.
\end{remark}

\subsection{Categories with Cerf decompositions} \label{ss:cerf}

The basic idea of Cerf decompositions is to decompose a $(d+1)$-manifold $Y=Y_{01}\cup_{\Sigma_1} Y_{12} \ldots \cup_{\Sigma_{k-1}}Y_{(k-1)k}$ into simpler pieces $Y_{ij}=f^{-1}([b_i,b_j])$ by cutting at regular level sets $\Sigma_i=f^{-1}(b_i)$ of a Morse function $Y\to\R$ as illustrated in Figure~\ref{fig:Cerf} below.
By viewing $Y$ as a cobordism between empty sets, i.e.\ as a morphism in $\Mor_{\Bor_{d+1}}(\emptyset,\emptyset)$, this can be seen as a factorization $[Y]=[Y_{01}]\circ [Y_{12}]\circ \ldots \circ [Y_{(k-1)k}]$ in $\Bor_{d+1}$. Here the Morse function $f$ and regular levels $b_i$ can be chosen such that each piece $Y_{i(i+1)}$ contains either none or one critical point, and thus is either a cylindrical cobordism -- diffeomorphic to the product cobordism $Z_\phi=[0,1]\times \Sigma_j$ as in \eqref{eq:Zphi} -- or a handle attachment as in the following remark. These ``simple cobordisms'' are illustrated in figures~\ref{fig:cyl} and \ref{fig:handle}.

\begin{figure}[!h]
\centering
\includegraphics[width=5.5in]{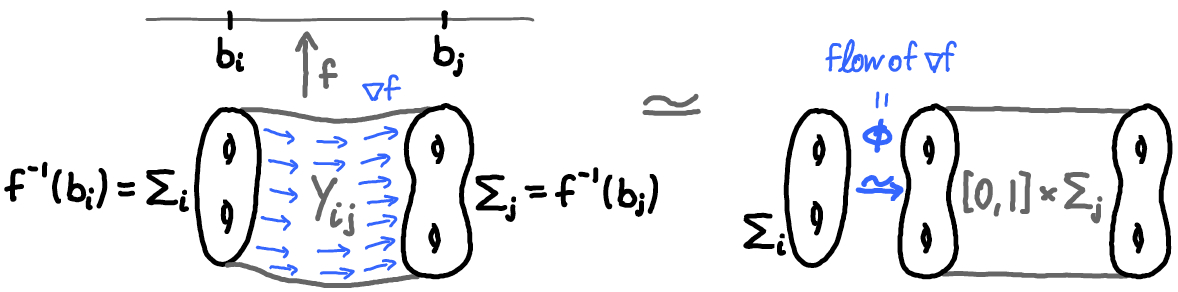}
\caption{A cylindrical cobordism supports a Morse function without critical points, whose gradient flow induces a diffeomorphism to the product cobordism $[0,1]\times \Sigma_j$ with natural identification $\Sigma_j\cong \{1\}\times\Sigma_j$ and boundary identification $\phi:\Sigma_i \overset{\sim}{\to} \{0\}\times\Sigma_j$ arising from the flow.}
\label{fig:cyl}
\end{figure}

\begin{remark}\label{rmk:handleattach} \rm 
A {\bf k-handle attachment} $Y_\alpha$ of index $0\leq k\leq d+1$ is a $(d+1)$-dimensional cobordism, which is obtained by attaching to a cylinder $[0,1]\times \Sigma$ a handle $B^{k} \times B^{d+1-k}$ along an attaching cycle $S^{k-1} \hookrightarrow \alpha\subset \{1\}\times\Sigma$, as illustrated in figure~\ref{fig:handle}. Here $B^k$ denotes a $k$-dimensional ball with boundary $\partial B^k=S^{k-1}$.

\begin{figure}[!h]
\centering
\includegraphics[width=5.5in]{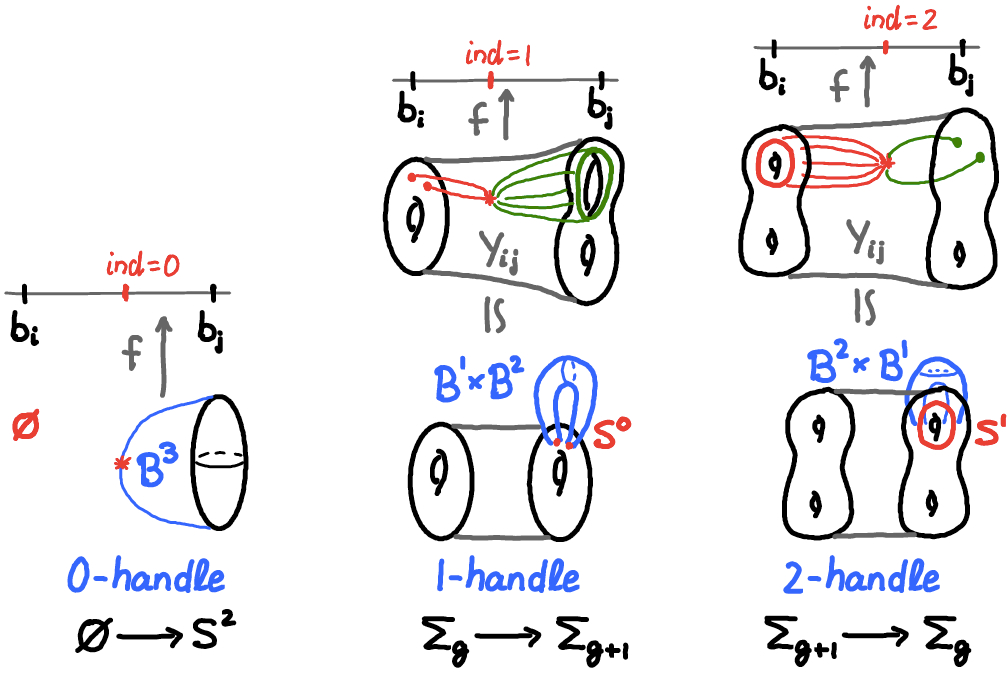}
\caption{Handle attachments in dimension $d=2$ are ``simple cobordisms'' which support a Morse function $f$ with a single critical point of index $0\leq k \leq 3$. The attaching cycles are given by intersection of the unstable manifold (in red) with the boundary.
Index $k=0$ and $k=3$ handle attachments are adjoint via orientation reversal and only appear between the empty set and sphere $S^2$. 
Index $k=1$ and $k=2$ handle attachments are adjoint via orientation reversal (which interchanges unstable and stable manifolds) and appear between surfaces $\Sigma_g, \Sigma_{g+1}$ of adjacent genus $g,g+1\in\N_0$. 
}
\label{fig:handle}
\end{figure}

By reversing the orientation and boundary identifications of any $k$-handle attachment $Y_\alpha$ from $\Sigma$ to $\Sigma'$, we obtain a cobordism  $Y_\alpha^-$ from $\Sigma'$ to $\Sigma$.
This reversed cobordism is also a $d+1-k$-handle attachment $Y_\alpha^-=Y_{\alpha^*}$ for an attaching cycle $S^{d-k} \hookrightarrow \alpha^* \subset \{1\}\times\Sigma'$. It moreover is the adjoint of $Y_\alpha$ in the sense of Remark~\ref{rmk:adj} and will become useful in the formulation of Cerf moves below.

Specifying to dimension $d=2$ and the connected bordism category, it will suffice to consider $2$-handle attachments (and their adjoints) with attaching circles that are homologically nontrivial and thus do not disconnect the surface.
More precisely, any attaching circle $S^1\simeq\alpha\subset \Sigma$ in a closed surface $\Sigma$ determines a $2$-handle attachment as follows:
Replacing an annulus neighbourhood of $\alpha$ by two disks specifies a lower genus surface $\Sigma'=\Sigma_\alpha$ together with a diffeomorphism $\pi_\alpha:\Sigma\less\alpha \to \Sigma'\less\{\text{2 points}\}$. Given this construction, the $2$-handle attaching cobordism $Y_\alpha$ from $\Sigma$ to $\Sigma'$ is unique up to diffeomorphism fixing the boundary.
\end{remark}

More detailed introductions to Cerf theory can be found in e.g.\ \cite{ce:st,gww,milnor:hcob}. 
Here we concentrate on the algebraic structure that it equips the bordism categories with.
To describe this structure, we may think of Cerf decompositions as a prime decomposition of $(d+1)$-manifolds, and more generally of $(d+1)$-cobordisms: A decomposition into simple cobordisms (cylindrical cobordisms and handle attachments) always exists and simple cobordisms have no further simplifying decomposition. And while these Cerf decompositions are not unique, any two choices of decomposition are related via just a few moves, some of which are shown in Figure~\ref{fig:Cerf}.
These moves reflect changes in the Morse function (critical point cancellations and critical point switches), cutting levels (cylinder cancellation), and the ways in which pieces are glued together (diffeomorphism equivalences which in particular encode handle slides). All of these Cerf moves are local in the sense\footnote{
While a diffeomorphism equivalence is not local, it decomposes into a sequence of local moves.
} 
that they replace only one or two consecutive cobordisms by one or two consecutive cobordisms with the same composition. That is, the moves are of one of three forms:
\begin{align*}
\ldots \cup_{\scriptscriptstyle\Sigma_i} Y_{ij}\cup_{\scriptscriptstyle\Sigma_j} Y_{jl} \cup_{\scriptscriptstyle\Sigma_l}\ldots  
&\;=\;
\ldots \cup_{\scriptscriptstyle\Sigma_1} \tilde Y_{ij}\cup_{\scriptscriptstyle\tilde \Sigma_j} \tilde Y_{jl} \cup_{\scriptscriptstyle\Sigma_l}\ldots  
\qquad\mbox{for}\; Y_{ij}\cup_{\scriptscriptstyle\Sigma_j} Y_{jl} =  \tilde Y_{ij}\cup_{\scriptscriptstyle\tilde \Sigma_j} \tilde Y_{jl} , 
\\
\ldots \cup_{\scriptscriptstyle\Sigma_1} Y_{ij}\cup_{\scriptscriptstyle\Sigma_j} Y_{jl} \cup_{\scriptscriptstyle\Sigma_l}\ldots  
&\;=\;
\ldots \cup_{\scriptscriptstyle\Sigma_1} \tilde Y_{il} \cup_{\scriptscriptstyle\Sigma_l}\ldots  
\qquad\qquad\quad\!\mbox{for}\; Y_{ij}\cup_{\scriptscriptstyle\Sigma_j} Y_{jl} =  \tilde Y_{il} , 
\\
\ldots \cup_{\scriptscriptstyle\Sigma_1} Y_{il} \cup_{\scriptscriptstyle\Sigma_l}\ldots  
&\;=\;
\ldots \cup_{\scriptscriptstyle\Sigma_1} \tilde Y_{ij}\cup_{\scriptscriptstyle\tilde \Sigma_j} \tilde Y_{jl} \cup_{\scriptscriptstyle\Sigma_l}\ldots  
\qquad\mbox{for}\; Y_{il} =  \tilde Y_{ij}\cup_{\scriptscriptstyle\tilde \Sigma_j} \tilde Y_{jl} .
\end{align*}
In the following, we will cast this notion -- decompositions into simple pieces that are unique up to a set of moves -- into more formal terms.
For that purpose we denote the union of all morphisms of a category $\cC$ by 
$\Mor_\cC := \bigcup_{x_1,x_2\in\Obj_\cC}\Mor_\cC(x_1,x_2)$, 
and we denote all relations between composable chains\footnote{
Throughout, we will use the term ``composable chain'' to denote ordered tuples of morphisms, in which each consecutive pair is composable, so that the entire tuple -- by associativity of composition -- has a well defined composition.
} 
of morphisms by 
$$
{\rm Rel}_\cC:= \bigcup_{k,\ell\in\N} \bigl\{ \bigl( (f_i), (g_j) \bigr) \in (\Mor_\cC)^k \times (\Mor_\cC)^\ell \,\big|\, 
f_1\circ\ldots\circ f_k = g_1 \circ \ldots\circ g_\ell \bigr\} .
$$

\begin{definition} \label{def:Cerf}
A {\bf category with Cerf decompositions} is a category $\cC$ together with
\begin{itemize}
\item
a subset $\SMor \subset \Mor_\cC$ of {\bf simple morphisms},
\item
a subset ${\rm Cerf}\subset {\rm Rel}_\cC$ of {\bf local Cerf moves}, which is symmetric (under exchanging the factors) and consists of pairs of composable chains of simple morphisms $f_{12}, \ldots, f_{(k-1)k}\in \SMor$, $g_{12}, \ldots, g_{(\ell-1)\ell} \in \SMor$ whose compositions are equal,
\end{itemize}
such that
\begin{itemize}
\item
the simple morphisms generate all morphisms, i.e.\ for any $m\in\Mor_\cC$ there exist $h_{12}, \ldots, h_{(n-1)n}\in\SMor$ such that $m=h_{12}\circ \ldots \circ h_{(n-1)n}$,
\item
the presentation in terms of simple morphisms is unique up to {\bf Cerf moves}, i.e.\ any two presentations of the same morphism in terms of $h_{12}, \ldots , h_{(n-1)n}\in\SMor$ and $\tilde h_{12}, \ldots ,\tilde h_{(\tilde n-1)\tilde n}\in\SMor$
are related by a finite sequence\footnote{
Throughout, we will use the term ``sequence'' to denote a finite totally ordered set.
} 
of identities
$$
h_{12}\circ \ldots \circ h_{(n-1)n}=h'_{12}\circ \ldots \circ h'_{(n'-1)n'} = \ldots = 
\tilde h_{12}\circ \ldots \circ \tilde h_{(\tilde n-1)\tilde n}
$$
in which each equality replaces one subchain of simple morphisms by another, 
$$
\ldots \circ f_{12}\circ \ldots\circ f_{(k-1)k} \circ \ldots = \ldots \circ g_{12}\circ \ldots\circ g_{(\ell-1)\ell}  \circ \ldots
$$
according to a local Cerf move $\bigl( (f_{12}, \ldots, f_{(k-1)k}), (g_{12}, \ldots, g_{(\ell-1)\ell}) \bigr)\in{\rm Cerf}$.
\end{itemize}
\end{definition}

The bordism categories $\Bor_{d+1}$ are the motivating example of categories with Cerf decompositions, with $\SMor$ and ${\rm Cerf}$ given by the simple cobordisms and Cerf moves as discussed above (for a more detailed exposition see \cite{gww}).
However, in the examples arising from gauge theory, we consider the $2+1$-dimensional connected bordism category, the $d=2$ case of the following general notion for $d\ge 2$.\footnote{
We restrict to dimension $d\ge 2$ when discussing connected bordisms since the handle attachments in dimension $d=1$ are morphisms between generally disconnected $1$-manifolds, so that $\Bor^{\rm conn}_{1+1}$ does not have useful connected Cerf decompositions.}

\begin{example}\rm \label{ex:connbor}
The {\bf connected bordism category} ${\rm \mathbf {Bor}}^{\rm \mathbf{conn}}_{\mathbf{d+1}}$ is defined as follows.
\begin{itemize}
\item
Objects are the closed, connected, oriented $d$-dimensional manifolds.
\item
Morphisms are the compact, connected, oriented $d+1$-dimensional cobordisms with identification of the boundary, and modulo diffeomorphisms as in $\Bor_{d+1}$.
\item
Composition is by gluing via boundary identifications as in $\Bor_{d+1}$.
\end{itemize}
If we allow $\Sigma=\emptyset$ as object, then closed, connected, oriented $d+1$-manifolds are contained in this category as morphisms from $\emptyset$ to $\emptyset$.
\end{example}

In this language, the Cerf decomposition theorem for $3$-manifolds -- in the connected case proven in \cite{GayKirby} and reviewed in \cite{gww} -- can be stated as in the following theorem, and is illustrated in Figure~\ref{fig:Cerf} and further explained in Remark~\ref{rmk:borCerf}.
Here, in strict categorical language, a $3$-cobordism from $\Sigma_-$ to $\Sigma_+$ is an equivalence class $[(Y,\iota^-,\iota^+)]$ of $3$-cobordisms and embeddings $\iota^\pm:\Sigma_\pm \to \partial Y$ modulo diffeomorphisms relative to the boundary identifications $\iota^\pm$. 
However, the decomposition and boundary identifications are actually induced by a decomposition of representatives, thus we drop the brackets and embeddings -- see \cite{gww} and \S\ref{ss:bord} for more deliberations on this.
Moreover, we may again generalize to dimension $d\ge 2$.

\begin{figure}[!h]
\centering
\includegraphics[width=5.5in]{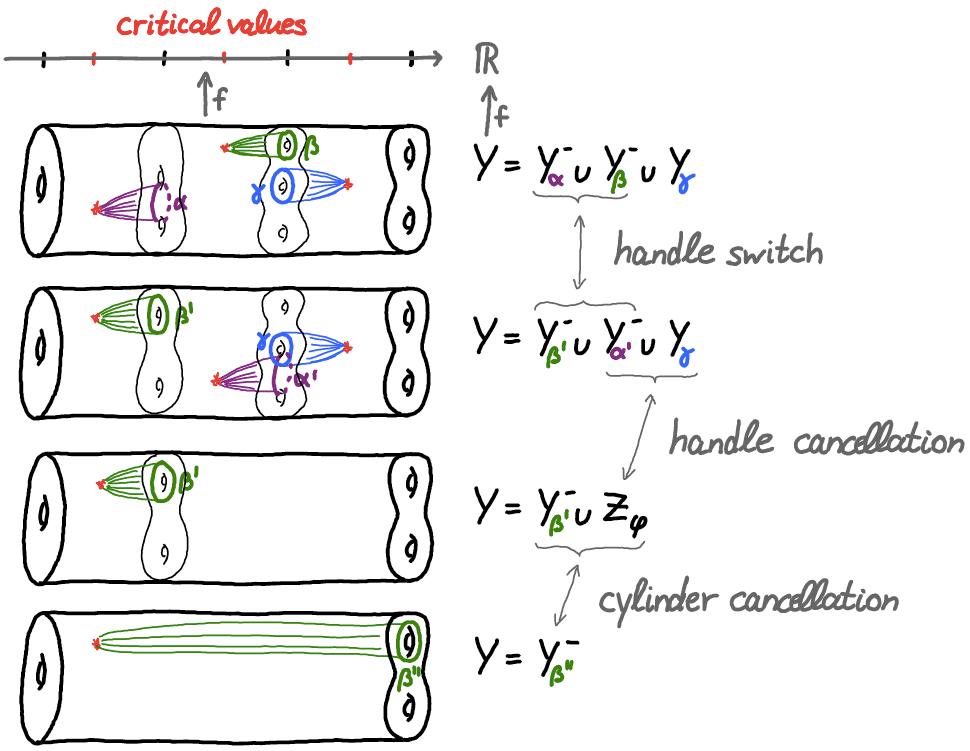}
\caption{Cerf decompositions of a 3-cobordism $Y$ and Cerf moves between them.}
\label{fig:Cerf}
\end{figure}

\begin{theorem}\label{thm:borCerf}
$\Bor^{\rm conn}_{d+1}$ is a category with Cerf decompositions as follows.
\begin{itemize}
\item
The set of simple morphisms $\SMor\subset\Mor_{\Bor^{\rm conn}_{d+1}}$ consists of 
\begin{itemize}
\item
cylindrical cobordisms $Z_\phi$ for diffeomorphisms $\phi:\Sigma\to\Sigma'$ as in \eqref{eq:Zphi},
\item
k-handle attachments $Y_\alpha\in\Mor(\Sigma,\Sigma')$ for $1\leq k \leq d$ as in Remark~\ref{rmk:handleattach}.

\end{itemize}
\item
The Cerf moves ${\rm Cerf}\subset {\rm Rel}_{\Bor^{\rm conn}_{d+1}}$ are the following and their transpositions:
\begin{itemize}
\item
Cylinder cancellations $\bigl( (Z_\phi , Z_\psi) , Z_{\psi\circ\phi} \bigr)$ for all composable pairs of diffeomorphisms $\phi,\psi$.
\item
Cylinder cancellations $\bigl( (Z_\phi , Y) , Y' \bigr)$ resp.\ $\bigl( ( Y, Z_\phi) , Y' \bigr)$ in which $Y'$ is the same cobordism as $Y$ (up to diffeomorphism), but with incoming resp.\ outgoing boundary inclusion pre- resp.\ post-composed with a diffeomorphism $\phi$.
\item
Critical point cancellations $\bigl( (Y_\alpha^-, Y_\beta) , Z_{\phi} \bigr)$ occur for attaching cycles $\alpha,\beta\subset \Sigma$ with transverse intersection in a single point;
these give rise to a pair of cobordisms 
$Y_\alpha^-\in\Mor(\Sigma',\Sigma)$, $Y_\beta\in\Mor(\Sigma,\Sigma'')$
whose composition is a cylindrical cobordism representing a diffeomorphism $\phi:\Sigma'\to\Sigma''$.
\item
Critical point switches $\bigl( (Y_\alpha, Y'_\beta) , (Y_\beta, Y'_\alpha) \bigr)$ 
and $\bigl( (Y_\alpha^-, Y_\beta) , (Y'_\beta, {Y'_\alpha}^-) \bigr)$
occur for disjoint attaching cycles $\alpha,\beta\subset \Sigma$; these give rise to a pair of cobordisms\footnote{
See Remark~\ref{rmk:field2} for more details on the notation used here.
} 
$Y_\alpha\in\Mor(\Sigma,\Sigma_\alpha')$, $Y'_\beta\in\Mor(\Sigma_\alpha',\Sigma'')$
whose composition is the same
as that of the pair
$Y_\beta\in\Mor(\Sigma,\Sigma_\beta')$, $Y'_\alpha\in\Mor(\Sigma_\beta',\Sigma'')$.
\end{itemize} 
\end{itemize}
\end{theorem}

\begin{remark}\label{rmk:borCerf} \rm 
For $d=2$ the objects of $\Mor_{\Bor^{\rm conn}_{2+1}}$ -- closed, connected, oriented surfaces -- can be classified up to diffeomorphism by their genus.
Moreover, the simple morphisms $\SMor\subset\Mor_{\Bor^{\rm conn}_{2+1}}$ can be further specified:
\begin{itemize}
\item
Cylindrical cobordisms $Z_\phi$ represent diffeomorphisms $\phi:\Sigma\to\Sigma'$ between surfaces of the same genus as in \eqref{eq:Zphi}.
\item
2-Handle attachments $Y_\alpha\in\Mor(\Sigma,\Sigma')$ specified by a homologically nontrivial circle $S^1\simeq \alpha\subset \Sigma$ are simple morphisms\footnote{
More precisely, $Y_\alpha$ is obtained by attaching to the cylindrical cobordism $[0,1]\times \Sigma$ a 2-handle $B^2 \times [-\eps,\eps]$ along a thickening $[-\eps,\eps]\times S^1 \subset \{1\}\times\Sigma$ of the attaching circle. 
} 
from a surface $\Sigma$ of genus $g$ to a surface $\Sigma'$ of genus $g-1$.
\item
1-Handle attachments are 2-handle attachments with reversed orientation, i.e.\ the simple morphisms $Y_\alpha^-\in\Mor(\Sigma',\Sigma)$ from a surface $\Sigma'$ of genus $g-1$ to a surface $\Sigma$ of genus $g$.
\end{itemize}
\end{remark}

The structural similarities between the symplectic and bordism categories
can now be phrased in terms of abstract Cerf decompositions.

\begin{lemma}
The symplectic category $\Symp$ from Definition~\ref{def:1symp} is a category with Cerf decompositions as follows:
\begin{itemize}
\item
The set of simple morphisms $\SMor \subset \Mor_{\Symp}$ consists of the equivalence classes $[L_{12}]$ of Lagrangian submanifolds $L_{12}\subset M_1^-\times M_2$.
\item
The set of local Cerf moves ${\rm Cerf}\subset {\rm Rel}_{\Symp}$ consists of the relations
$$
\bigl( \, [ (L_{01},L_{12}) ] \,,\, [L_{01}\circ L_{12}] \, \bigr)
\qquad\mbox{and}\qquad
\bigl( \, [L_{01}\circ L_{12}]\,,\,  [ (L_{01},L_{12})] \, \bigr)
$$
for embedded geometric compositions $L_{01}\circ L_{12}$ as in \eqref{eq:embedded}.
\end{itemize}
\end{lemma}
\begin{proof}
To check that the simple morphisms generate all morphisms, consider a general morphism $\uL\in\Mor_{\Symp}(M,N)$ and pick a representative $(L_{01},\ldots,L_{(k-1)k})$, given by a composable chain of Lagrangian submanifolds $L_{ij}\subset M_i^-\times M_j$ from $M_0=M$ to $M_k=N$.
The definition of composition in $\Symp$ yields the identity
$$
\uL = \bigl[ (L_{01},\ldots,L_{(k-1)k}) \bigr] 
= \bigl[ L_{01}\# \ldots \# L_{(k-1)k} \bigr] 
= [L_{01}] \circ \ldots \circ [L_{(k-1)k}].
$$
Since each $[L_{ij}]$ is a simple morphism, this is the required decomposition of $\uL$ into simple morphisms.
To show that these decompositions are unique up to the given Cerf moves, note that an equality
$$
[L_{01}] \circ \ldots \circ [L_{(k-1)k}] = [L'_{01}] \circ \ldots \circ [L'_{(k'-1)k'}]
$$
in $\Mor_{\Symp}$ means by definition that the corresponding morphisms in $\Symp^\#$ are equivalent
$$
( L_{01} , \ldots , L_{(k-1)k} ) \sim ( L'_{01} , \ldots , L'_{(k'-1)k'} ) 
$$
under the equivalence relation $\sim$ given in Definition~\ref{def:1symp}.
Recall that this relation is generated by the geometric composition moves ${\rm Comp}\subset \Mor_{\Symp^\#}\times\Mor_{\Symp^\#}$, so that there is a sequence of moves from $( L_{01} , \ldots , L_{(k-1)k} )$ to $( L'_{01} , \ldots , L'_{(k'-1)k'} )$ in which adjacent pairs are replaced by their embedded geometric composition. 
Our definition of ${\rm Cerf}\subset {\rm Rel}_{\Symp}$ by moves on equivalence classes encoded by ${\rm Comp}$ translates this into a sequence of Cerf moves from $[L_{01}] \circ \ldots [L_{(k-1)k}]$ to $[L'_{01}] \circ \ldots  [L'_{(k'-1)k'}]$.
\end{proof}

\subsection{Construction principle for Floer field theories} \label{ss:field}

The algebraic background of Floer field theory is the following construction principle for functors between categories with Cerf decompositions.

\begin{lemma} \label{le:field0}
Let $\cC,\cD$ be two categories with Cerf decompositions and a {\bf{\rm\bf Cerf}-compatible partial functor} $\cF: (\Obj_\cC,\SMor_\cC) \to (\Obj_\cD,\SMor_\cD)$ consisting of
\begin{itemize}
\item
a map $\Obj_\cC \to \Obj_\cD$, 
\item
a map $\SMor_\cC \to \SMor_\cD$ which induces 
a map ${\rm Cerf}_\cC \to {\rm Cerf}_\cD$ given by 
$$
\bigl( (f_{(i-1)i})_{i=1,\ldots,k} , (g_{(j-1)j})_{j=1,\ldots,\ell} \bigr)  
\mapsto \bigl( (\cF(f_{(i-1)i}))_{i=1,\ldots,k} , (\cF(g_{(j-1)j}))_{j=1,\ldots,\ell}  \bigr) .
$$
\end{itemize}
Then $\cF$ has a unique extension to a functor $\overline\cF: \cC\to\cD$ which restricts to $\cF$ on $\Obj_\cC$ and $\SMor_\cC\subset\Mor_\cC$.
\end{lemma}
\begin{proof}
Compatibility of $\overline\cF$ with composition requires its value on a general morphism $f\in\Mor_\cC$ to be $\overline\cF(f) = \cF(f_{01}) \circ \ldots \circ \cF(f_{(k-1)k})$ for any Cerf decomposition $f= f_{01}\circ \ldots \circ f_{(k-1)k}$ into simple morphisms $f_{ij}\in\SMor_\cC$. 
The induced map ${\rm Cerf}_\cC \to {\rm Cerf}_\cD$ guarantees that this definition of $\overline\cF(f)$ is independent of the choice of decomposition, thus yields a well defined map $\Mor_\cC\to\Mor_\cD$. Moreover, this map is compatible with composition by construction. Thus a well defined functor $\overline\cF$ is uniquely determined by $\cF$. 
\end{proof}

The next Lemma specializes this abstract construction principles to $\cC=\Bor^{\rm conn}_{d+1}$ and $\cD=\Symp$ and is illustrated in Figure~\ref{fig:Ffield}. It can be read in two ways: In the strictly categorical sense, a partial functor should assign to a class $[Y]$ of simple cobordisms modulo diffeomorphisms relative to the boundary identifications a class $[L_Y]$ of Lagrangian submanifolds modulo embedded geometric composition. In practice, this will be achieved by assigning to each simple cobordism $Y$ a Lagrangian submanifold $L_Y$ in a way that is compatible with diffeomorphisms.

\begin{figure}[!h]
\centering
\includegraphics[width=5.5in]{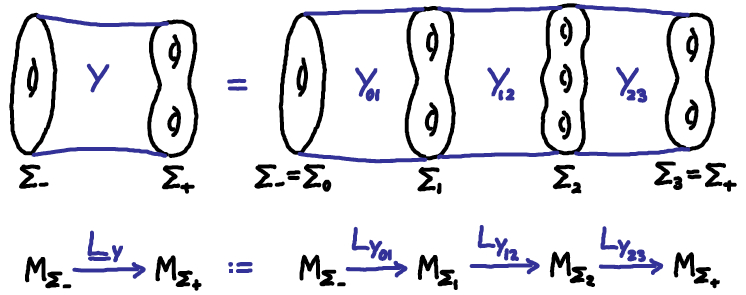}
\caption{Construction principle for Floer field theory: A functor $\Bor^{\rm conn}_{2+1}\to\Symp$ can be specified by associating symplectic manifolds $M_\Sigma$ to surfaces $\Sigma$ and simple Lagrangians $L_Y$ to simple 3-cobordisms $Y$ in a way that is compatible with Cerf moves.}
\label{fig:Ffield}
\end{figure}

Strictly speaking, the following is a mild generalization of Lemma~\ref{le:field0} because critical point switches really correspond to two Cerf moves in $\Symp$, for example
$$
Y_{01} \cup_{\Sigma_1} Y_{12} \simeq Z_{01}' \cup_{\Sigma_1'} Y'_{12}
\quad\Longrightarrow\quad
\bigl( L_{Y_{01}} , L_{Y_{12}} \bigr) \sim  
L_{Y_{01}} \circ L_{Y_{12}} = L_{Y'_{01}} \circ L_{Y'_{12}}
\sim \bigl( L_{Y'_{01}} , L_{Y'_{12}} \bigr).
$$
Examples of Floer field theories constructed in this way will be discussed in \S\ref{ss:ex}.

\begin{lemma}\label{le:field1}
Let $\cF: (\Obj_{\Bor^{\rm conn}_{d+1}},\SMor_{\Bor^{\rm conn}_{d+1}}) \to (\Obj_{\Symp},\SMor_{\Symp})$ be a {\rm Cerf}-compatible partial functor consisting of the following:
\begin{itemize}
\item
symplectic manifolds $M_\Sigma$ for each $d$-manifold $\Sigma\in\Obj_{\Bor^{\rm conn}_{d+1}}$;
\item
Lagrangian submanifolds $L_{[Y]}\in \SMor_{\Symp}(M_\Sigma, M_{\Sigma'})$ for each simple $d+1$-cobordism $[Y]\in\SMor_{\Bor^{\rm conn}_{d+1}}(\Sigma,\Sigma')$; \\
$\!\!\big(\!\!$ More precisely, this requires the following:

\begin{itemize}
\item
Lagrangian submanifolds $L_Y\subset M_{\partial^- Y}^-\times M_{\partial^+ Y}$ for each 
handle attachment $Y$ with partitioned boundary $\partial Y= \partial^- Y \sqcup \partial^+ Y$,
\item
symplectomorphisms $M_\Sigma \to M_{\Sigma'}$ denoted by their graphs $L_\phi \subset M_\Sigma^- \times M_{\Sigma'}$ for the cylindrical cobordisms $Z_{\phi}$ representing each diffeomorphism $\phi:\Sigma\to\Sigma'$,
\item
identities\footnote{
In these identities $L_\phi$ is the graph of a map so that $\circ$ is geometric composition of Lagrangians.
Viewing $L_\phi$ as map, they could be rewritten as
$L_{\psi\circ\phi} = L_\psi \circ L_\phi$ and 
 $L_{\Psi(Y)} = \bigl( L_{\Psi|_{\partial^-Y}}\times L_{\Psi|_{\partial^+Y}}\bigr)(L_Y)$.
} 
$L_{\psi\circ\phi} = L_\phi \circ L_\psi$ for diffeomorphisms $\phi:\Sigma\to\Sigma'$, $\psi:\Sigma'\to\Sigma''$
and
$L_{\Psi(Y)} = L_{\Psi|_{\partial^-Y}^{-1}} \circ L_Y \circ L_{\Psi|_{\partial^+Y}}$
for $\Psi:Y\to Z$.
\end{itemize}

Then any choice of a representative cobordism $Y$ with orientation preserving diffeomorphisms $\iota^-:\Sigma^- \to \partial^- Y$, $\iota^+:\Sigma' \to \partial^+ Y$ induces well defined morphism
$$
L_{[Y]} \,:=\; \bigl[\bigl(L_{\iota^-}, L_Y, L_{(\iota^+)^{-1}} \bigr)\bigr]
\;=\;
\bigl[ \, \bigl(L_{\iota^-}^{-1} \times L_{\iota^+}^{-1} \bigr)( L_Y) \subset M_\Sigma^- \times M_{\Sigma'} \,\bigr] ,
$$
where in the last equality we view $L_{\iota^\pm}: M_{\Sigma^\pm}\to M_{\partial^\pm Y}$ as maps.$\big)$
\item
 identities of Lagrangians for each local Cerf move
\begin{align*}
\bigl( (X,Y), Z \bigr) \in {\rm Cerf}_{\Bor^{\rm conn}_{d+1}} & \quad\Rightarrow\quad L_X\circ L_Y = L_Z , \\
\bigl( X, (Y,Z) \bigr) \in {\rm Cerf}_{\Bor^{\rm conn}_{d+1}} & \quad\Rightarrow\quad L_X = L_Y\circ L_Z  ,\\
\bigl( (V,W), (X,Y) \bigr) \in {\rm Cerf}_{\Bor^{\rm conn}_{d+1}} & \quad\Rightarrow\quad L_V\circ L_W = L_X \circ L_Y ,
\end{align*}
where all geometric compositions on the right hand side are embedded as in \eqref{eq:embedded}.
\end{itemize}
Then $\cF$ has a unique extension to a functor $\overline\cF:\Bor^{\rm conn}_{d+1}\to\Symp$.

Moreover, if $\cF$ takes values in an exact or monotone symplectic category $\Symp^\tau$ (see Remark~\ref{rmk:monotone}), then $\cF$ induces a functor $\Bor^{\rm conn}_{d+1}\to\Cat$.
\end{lemma}
\begin{proof}
To check that the construction of simple morphisms in the second bullet point is well defined we need to consider a diffeormorphism $\Psi:Y\to Z$ which preserves the partition of boundary components, i.e.\ $\Psi_{\partial^\pm Y}$ maps $\partial^\pm Y$ to $\partial^\pm Z$. 
Then $Y$ and $Z=\Psi(Y)$ with the corresponding boundary identifications yields the same Lagrangian submanifold $L_{[\Psi(Y)]}= L_{[Y]}$ since we have
\begin{align*}
& \bigl(\, L_{\Psi|_{\partial^-Y}\circ\iota^-} \,,\, L_{\Psi(Y)} \,,\,  L_{(\Psi|_{\partial^+Y}\circ\iota^+)^{-1}} \,\bigr)\\
&\quad=\;
\bigl(\, L_{\iota^-}\circ L_{\Psi|_{\partial^-Y}}\,,\, L_{\Psi|_{\partial^-Y}^{-1}} \circ L_Y \circ  L_{\Psi|_{\partial^+Y}} \,,\,  L_{(\Psi|_{\partial^+Y})^{-1}} \circ  L_{(\iota^+)^{-1}} \,\bigr) \\
&\quad\sim\;
\bigl(\, L_{\iota^-}\circ L_{\Psi|_{\partial^-Y}}\circ L_{(\Psi|_{\partial^-Y})^{-1}} \,,\,  L_Y \,,\,  L_{\Psi|_{\partial^+Y}} \circ L_{(\Psi|_{\partial^+Y})^{-1}}\circ L_{(\iota^+)^{-1}} \,\bigr) \\
&\quad=\;
\bigl(\, L_{(\Psi|_{\partial^-Y})^{-1} \circ \Psi|_{\partial^-Y} \circ \iota^- } \,,\,  L_Y \,,\,  L_{ (\iota^+)^{-1} \circ (\Psi|_{\partial^+Y})^{-1} \circ \Psi|_{\partial^+Y} }  \,\bigr)
\;=\;
\bigl(\, L_{\iota^-} \,,\,  L_Y \,,\,  L_{(\iota^+)^{-1}} \,\bigr) .
\end{align*}
Now on objects $\Sigma$, the functor $\overline\cF$ is determined by the symplectic manifolds $M_\Sigma$. For a morphism $[Y]\in\Bor^{\rm conn}_{d+1}(\Sigma,\Sigma')$, pick a representative cobordism $Y$ with orientation preserving embeddings $\iota^-_Y:\Sigma^- \to \partial Y$, $\iota^+_Y:\Sigma' \to \partial Y$ to the respective boundary components.
By the Cerf decomposition Theorem~\ref{thm:borCerf}, there exists a decomposition
$Y=Y_{01}\cup_{\Sigma_1}Y_{12}\cup\ldots\cup_{\Sigma_{n-1}}Y_{(n-1)n}$ into 
simple morphisms which are either handle attachments $Y_{i(i+1)}$ with boundary identifications $\iota_i^-:\Sigma^-_i\to\partial Y_{i(i+1)}$, $\iota_{i+1}^+:\Sigma_{i+1}\to\partial Y_{i(i+1)}$
or cylindrical cobordisms $Y_{i(i+1)}=Z_{\phi_i}$ representing a diffeomorphism $\phi_i:\Sigma_i\to\Sigma_{i+1}$.
As in Lemma~\ref{le:field0}, functoriality then requires 
$$
\overline\cF([Y])=\bigl[\bigl(L_{\iota^-_Y}, L_{Y_{01}},L_{Y_{12}}, \ldots, L_{Y_{(n-1)n}} , L_{(\iota^+_Y)^{-1}} \bigr)\bigr]
$$ 
to be given by the algebraic composition in $\Symp$ of the corresponding Lagrangian submanifolds.
This fully determines $\overline\cF$, but to see that it is well defined we need to consider not just another Cerf decomposition of $Y$ -- for which the proof is exactly as in Lemma~\ref{le:field0} --
but also allow for a diffeomorphism $\Psi:Y\to Z$ that intertwines boundary identifications, $\Psi\circ\iota_Y^\pm = \iota^\pm_Z$.
The latter induces a Cerf decomposition $Z=\Psi(Y_{01})\cup_{\Psi(\Sigma_1)}\Psi(Y_{12})\cup\ldots\cup_{\Psi(\Sigma_{n-1})}\Psi(Y_{(n-1)n})$ with $\Sigma_i:=Y_{(i-1)i}\cap Y_{i(i+1)}\subset Y$, whose value under $\overline\cF$ is
\begin{align*}
\overline\cF([Z])
&=
\bigl[\bigl(L_{\Psi|_{\partial^-Y}\circ \iota^-_Y} , L_{\Psi(Y_{01})},L_{\Psi(Y_{12})}, \ldots, L_{\Psi(Y_{(n-1)n})} ,L_{(\Psi|_{\partial^+Y}\circ \iota^+_Y)^{-1}} \bigr)\bigr] \\
&=
\bigl[\bigl(L_{\iota^-_Y} \circ L_{\Psi|_{\partial^-Y}} ,
 L_{(\Psi|_{\partial^-Y})^{-1}}\circ L_{Y_{01}} \circ L_{\Psi|_{\Sigma_1}}, L_{(\Psi|_{\Sigma_1})^{-1}}\circ L_{Y_{12}}\circ L_{\Psi|_{\Sigma_2}}, \ldots\\
&\qquad\qquad\qquad\qquad\ldots
 L_{(\Psi|_{\Sigma_{n-1}})^{-1}}\circ L_{Y_{(n-1)n}} 
\circ L_{\Psi|_{\partial^+Y}} , L_{(\Psi|_{\partial^+Y})^{-1}} \circ L_{(\iota^+_Y)^{-1}} 
  \bigr)\bigr] \\
  &=
\bigl[\bigl(L_{\iota^-_Y} , L_{Y_{01}} , L_{Y_{12}} , \ldots , L_{Y_{(n-1)n}} 
, L_{(\iota^+_Y)^{-1}}  \bigr)\bigr] 
\;=\; \overline\cF([Y]) .
\end{align*}
This finishes the proof that the unique extension $\overline\cF$ is a well defined functor.

Finally, if $\overline\cF: \Bor^{\rm conn}_{d+1}\to\Symp^\tau$ takes values in a monotone symplectic category (for a monotonicity constant $\tau\ge0$; see Remark~\ref{rmk:monotone}), then it can be composed with the Yoneda functor $\Symp^\tau\to\Cat$ constructed in \cite{ww:cat} and Lemma~\ref{le:sympcat} below to induce a functor 
$\Bor^{\rm conn}_{d+1}\to\Cat$, as claimed.
Here the existence of the Yoneda functor follows from the fact that $\Symp^\tau$ extends to a 2-category. 
\end{proof}

A formal notion of $d+1$ Floer field theory should also include a notion of duality. However, the abstract categorical notion of duality requires a monoidal structure -- roughly speaking, an associative multiplication of objects that extends to a bifunctor.
While in the bordism category $\Bor_{d+1}$ a monoidal structure is naturally given by disjoint unions of objects and morphisms, an extension of the gauge theoretic examples in \S\ref{ss:ex} to disconnected bordisms remains elusive; see Remark~\ref{rmk:conn}. 
Instead, we work with the following practical notion of adjunctions, which will be part of an abstract notion of quilted 2-categories in Definition~\ref{def:adj}.

\begin{remark} \label{rmk:adj} \rm 
The {\bf adjoint of a cobordism} $[Y]\in\Mor_{\Bor_{d+1}}(\Sigma_0,\Sigma_1)$ with boundary embeddings $\iota^\pm_Y:\Sigma_i \to \partial Y$ is the cobordism $[Y^-]\in\Mor_{\Bor_{d+1}}(\Sigma_1,\Sigma_0)$ obtained by reversing the orientation and boundary embeddings $\iota^+_Y: \Sigma_1^- \to \partial Y^-$, $\iota^-_Y: \Sigma_0 \to \partial Y^-$. In particular, the adjoint of a $k$-handle attachment is a $d+1-k$-handle attachment.

The {\bf adjoint of a Lagrangian} $L\subset M_0^-\times M_1$ is $L^T:=\tau(L)\subset M_1^-\times M_0$ obtained by transposition $\tau(p_0,p_1):=(p_1,p_0)$.
For very simple morphisms -- cylindrical cobordisms and graphs of symplectomorphisms -- these adjoints are also inverse morphisms, but not in general.

In the category of categories, not every functor may have an adjoint, but there also is a notion of two functors $f:\cC\to\cD$ and $f^T:\cD\to\cC$ being adjoint; see Definition~\ref{def:adj}.
\end{remark}

With this we can somewhat formalize our notion of connected Floer field theories. We will keep the definition flexible to allow for current progress towards constructing more general symplectic 2-categories as discussed in Example~\ref{ex:infdim} and Remark~\ref{rmk:ainfty2symp}.

\begin{definition} \label{def:fft}
A {\bf d+1 connected Floer field theory} is an adjunction preserving functor $\Bor^{\rm conn}_{d+1}\to\cC$ to an algebraic category (such as $\cC=\Cat$) that arises as composition of a functor $\cF: \Bor^{\rm conn}_{d+1}\to\cS$ to a symplectic category (i.e.\ a category such as $\cS=\Symp^\tau$ whose objects are symplectic manifolds) 
with a Yoneda-type functor arising from a 2-categorical structure on $\cS$ that encodes Floer theory (such as the functor $\Symp^\tau\to\Cat$ constructed in Lemma~\ref{le:sympcat}).
\end{definition}

Here the Yoneda functor $\Symp^\tau\to\Cat$ arises from a quilted generalization of Floer homology which was developed in \cite{ww:qhf,ww:quilts,ww:cat} within a mononote symplectic category (see Remark~\ref{rmk:monotone}) that guarantees well behaved moduli spaces of pseudoholomorphic quilts; see \S\ref{ss:symp2}.
Since the composition with this functor is automatic (if it exists), we will sometimes also refer to a functor $\Bor^{\rm conn}_{d+1}\to\Symp$ (even if it does not take values in a monotone subcategory) as a Floer field theory -- because it reduces the question of constructing a functor $\Bor^{\rm conn}_{d+1}\to\Cat$ to ensuring that quilted Floer homology is well defined on its image.
One might be tempted to call a functor $\Bor_{d+1}\to\Symp$ a ``d+1 symplectic field theory'', but the label of SFT = symplectic field theory was given by \cite{egh} to a theory in which another symplectic category -- given by contact-type manifolds and symplectic cobordisms -- is the domain, not the target of a functor.

\subsection{2+1 Floer field theories arising from gauge theory} \label{ss:ex}

Working more specifically in dimensions 2+1, and making use of the adjunctions in Remark~\ref{rmk:adj}, we can specialize Lemma~\ref{le:field1} even further to observe that a 2+1 connected Floer field theory $\Bor^{\rm conn}_{2+1}\to\Cat$ in the sense of Definition~\ref{def:fft}
can be obtained by essentially just fixing symplectic data for one surface of each genus and attaching circles in these.
Here we will be somewhat cavalier about diffeomorphisms that are isotopic to the identity. These do not affect the representation spaces in Example~\ref{ex:rep}, but in general, e.g.\ in Example~\ref{ex:sym}, more vigilance such as in \cite{GayKirby, gww, perutz1} is required.

\begin{remark} \label{rmk:field2} \rm 
In order to construct a 2+1 connected Floer field theory $\Bor^{\rm conn}_{2+1}\to\Cat$, it suffices to 
construct a functor $\cF:\Bor^{\rm conn}_{2+1}\to\Symp^\tau$ that preserves adjunctions. The latter can be obtained as in Lemma~\ref{le:field1} by the following constructions.
\begin{enumerate}
\item
To a closed, connected, oriented surface $\Sigma$, associate a symplectic manifold $M_\Sigma$ (that is compact and $\tau$-monotone for a fixed $\tau\ge 0$; see Remark~\ref{rmk:monotone}).
\item
To a diffeomorphism $\phi:\Sigma_0\to\Sigma_1$ associate a symplectomorphism ${L_\phi: M_{\Sigma_0}\to M_{\Sigma_1}}$ such that $L_\phi \circ L_\psi=L_{\phi\circ\psi}$ (as maps) when $\phi,\psi$ are composable.
\item
To a 2-handle attaching cobordism $Y_\alpha\in\Mor_{\Bor^{\rm conn}_{2+1}}(\Sigma,\Sigma')$ between connected surfaces as in Remark~\ref{rmk:handleattach} associate a Lagrangian submanifold $L_\alpha\subset M_\Sigma^-\times M_{\Sigma'}$ (that is compact and $\tau$-monotone).
\item[(3')]
To the reversed 1-handle attachment $Y_\alpha^-\in\Mor_{\Bor^{\rm conn}_{2+1}}(\Sigma',\Sigma)$ associate the transposed Lagrangian $L_\alpha^T\subset M_{\Sigma'}^-\times M_\Sigma$.
\item
For attaching circles $\alpha, \phi(\alpha) \subset\Sigma$ related by a diffeomorphism $\phi:\Sigma\to\Sigma$, there is a diffeomorphism $\phi':\Sigma_\alpha \to \Sigma_{\phi(\alpha)}$ determined by $\phi'\circ\pi_\alpha = \pi_{\phi(\alpha)}\circ\phi$ 
such that the 3-cobordisms $Y_\alpha \simeq Y_{\phi(\alpha)}$ are diffeomorphic relative to $\phi,\phi'$ on the boundary.
Ensure that this is reflected by an identity of Lagrangians
$( L_\phi \times L_{\phi'}) (L_\alpha) = L_{\phi(\alpha)}$ 
via the symplectomorphisms given in 2.
\item
For disjoint attaching circles $\alpha,\beta\subset \Sigma$, denote by $\beta':=\pi_\alpha(\beta)\subset\Sigma_\alpha$ and $\alpha':=\pi_\beta(\alpha)\subset\Sigma_\beta$ the attaching circles in the outgoing boundary of $Y_\alpha$ resp.\ $Y_\beta$ that are obtained from $\beta$ resp.\ $\alpha$. 
Then there is a diffeomorphism $\phi'':(\Sigma_\alpha)_{\beta'} \to (\Sigma_\beta)_{\alpha'}$ between the outgoing boundaries of $Y_{\beta'}, Y_{\alpha'}$, determined by $\phi''\circ \pi_{\beta'}\circ\pi_\alpha = \pi_{\alpha'}\circ \pi_\beta$, such that the 3-cobordisms\footnote{
Here $\cup_{\phi''}$ denotes a gluing of the boundaries of $Y_{\beta'}, Y_{\alpha'}^-$ via the diffeomorphism $\phi''$.
} 
$Y^-_\alpha \cup_\Sigma Y_\beta \simeq Y_{\beta'} \cup_{\phi''} Y_{\alpha'}^-$ are diffeomorphic with fixed boundary, and the 3-cobordisms 
$Y_\alpha \cup_{\Sigma_\alpha} Y_{\beta'} \simeq Y_\beta \cup_{\Sigma_\beta} Y_{\alpha'}$ are diffeomorphic relative to $\id_\Sigma,\phi''$ on the boundary.
Ensure that this is reflected by embedded geometric compositions 
$L_\alpha^T \circ L_\beta$, $(\id\times \phi'')(L_{\beta'}) \circ L_{\alpha'}^T$, 
$L_\alpha \circ L_{\beta'}$, $L_\beta \circ L_{\alpha'}$ and identities 
\begin{equation} \label{eq:referee}
(\id \times L_{\phi''})(L_\alpha \circ L_{\beta'})=L_\beta \circ L_{\alpha'}, 
\qquad
L_\alpha^T \circ L_\beta = (\id\times \phi'')(L_{\beta'}) \circ L_{\alpha'}^T .
\end{equation}
\item
For attaching circles $\alpha,\beta\subset \Sigma$ with transverse intersection in a single point, the composition $Y_\alpha^-\cup_\Sigma Y_\beta \simeq Z_\phi$ is diffeomorphic with fixed boundary to the cylindrical cobordism of a diffeomorphism $\phi:\Sigma_\alpha\to\Sigma_\beta$ determined by $\phi\circ\pi_\alpha=\pi_\beta$ on $\Sigma\less(\alpha\cup\beta)$ and $\phi(\pi_\alpha(\beta))=\pi_\beta(\alpha)$.
Ensure that this is reflected by an embedded geometric composition
$L_\alpha^T \circ L_{\beta}={\rm gr}(L_{\phi})$.
\end{enumerate} 
While step 1 fixes the functor $\cF$ on all objects, steps 2 and 3 fix explicit Lagrangians $\cF([Y])=\underline L_Y$ only for simple morphisms $Y$ as $L_{Z_\phi}=L_\phi$ for cylindrical cobordisms, $L_{Y_\alpha}=L_\alpha$ for 2-handle attachments, and $L_{Y_\alpha^-}=L_\alpha^T$ for their adjoint 1-handle attachments.
To determine the value of the functor $\cF([Y])=[\underline L_Y]$ on a general cobordism $Y\in\Mor_{\Bor_{2+1}}(\Sigma,\Sigma')$, we choose a Cerf decomposition $Y=Y_{01}\cup_{\Sigma_1} Y_{12} \ldots \cup_{\Sigma_{k-1}}Y_{(k-1)k}$ into a composable chain of simple morphisms $Y_{ij}\in\Mor_{\Bor_{2+1}}(\Sigma_i,\Sigma_j)$ from $\Sigma_0=\Sigma$ to $\Sigma_k=\Sigma'$. Then functoriality requires 
$$
[\underline L_Y] = \cF([Y])= \cF([Y_{01}]) \circ  \cF([Y_{12}]) \ldots  \circ\cF([Y_{(k-1)k}]) = 
[L_{Y_{01}}] \circ  [L_{Y_{12}}]  \ldots \circ [L_{Y_{(k-1)k}}] , 
$$
and this is well defined since different Cerf decompositions of $[Y]$ are related by Cerf moves, which steps 4-6 guarantee to correspond to embedded geometric compositions, i.e.\ yield the same morphisms in the symplectic category. 
More precisely, steps 2,3 associate to a cobordism with Cerf decomposition (a factorization in $\Bor^{\rm conn}_{2+1}$) a morphism in the extended symplectic category of Definition~\ref{def:extsymp},
$$
Y=Y_{01}\cup_{\Sigma_1} Y_{12} \ldots \cup_{\Sigma_{k-1}}Y_{(k-1)k}
\quad\mapsto\quad
\underline L_Y =  L_{Y_{01}} \# L_{Y_{12}}  \ldots \# L_{Y_{(k-1)k}} .
$$
Then Cerf moves can be viewed as isomorphisms between different factorizations in $\Bor^{\rm conn}_{2+1}$, 
and steps 4-6 relate these to isomorphisms in $\Symp^\#$ given by the relation used in Definition~\ref{def:1symp} of the symplectic category as the quotient of $\Symp^\#$.
This could more precisely be phrased as a 2-functor between extensions of $\Bor^{\rm conn}_{2+1}$ to a bicategory as in Example~\ref{ex:1epsbor} and of $\Symp^\#$ to a 2-category as in Example~\ref{ex:ext0}.
\end{remark}

Since its first announcement in \cite{ww:fielda}, this Floer field philosophy has 
been applied to obtain various proposals for $2+1$ field theories, which are inspired from various gauge theories. Unfortunately, these are still preprints \cite{ww:fielda,ww:fieldb}, work in progress \cite{lp}, or published \cite{reza,MWcr,lekili,auroux_hf} but hinging on generalizations of the crucial isomorphism in Floer homology under geometric compositions beyond the (compact monotone) setting in which it was proven in \cite{ww:isom}; see Remarks~\ref{rmk:monotone}--\ref{rmk:LLBS}.
Instead of discussing the technicalities and possible obstructions, this section focusses on the motivations and thus presents both intuitive and naive reasonings why theories along these lines are to be expected.

The intuitive reason for an intimate connection between symplectic geometry and gauge theory in dimensions $2+1$ is the following example of a partial functor from $\Bor^{\rm conn}_{2+1}$ to a category of infinite dimensional symplectic Banach spaces and Lagrangian Banach-submanifolds. It provides the basic data from which one expects a 2+1+1 field theory which comprises Donaldson invariants and instanton Floer homology\footnote{
Donaldson invariants and instanton Floer homology are invariants for smooth 4- and 3-manifolds that were developed in the 1980s \cite{Don:inv,Floer:inst}; see \cite{DK,Don:book} for introductions. Similar to the symplectic versions of Floer homology, the 3-manifold invariant can be viewed as the Morse homology of the Chern-Simons functional on a space of connections (modulo gauge) on the 3-manifold, with the gradient flow recast as the ASD Yang-Mills PDE (whose stationary solutions are the flat connections).
}
for certain 4- and 3-manifolds, as discussed in \S\ref{ss:af}.

\begin{example}[Infinite dimensional Floer field theory from spaces of connections]   \rm\label{ex:infdim}
Fix a compact, connected, simply connected Lie group $G$, and let $\langle \cdot  , \cdot \rangle$ be a $G$-invariant inner product on the Lie algebra ${\mathfrak g}$. (The main and first nontrivial examples are $G=SU(r)$ for $r\ge 2$.)
The following constructions will use some basic notations from gauge theory, which can be found in e.g.\ \cite{W:book}.
These constructions also have natural extensions to nontrivial bundles -- such as the unique nontrivial $SO(3)$-bundles over surfaces and handle attachments used in \cite{ww:fielda}, which also serve to avoid issues of reducible connections.
\begin{enumerate}
\item
To each closed, connected, oriented surface $\Sigma$, we associate the  space of connections $\cA(\Sigma):=\Om^1(\Sigma,{\mathfrak g})$ on the trivial $G$-bundle over $\Sigma$. It has a natural symplectic structure given by $\omega( a_1 , a_2 ) = \int_\Sigma \langle a_1 \wedge a_2 \rangle$ for $a_i\in \cA(\Sigma)$; see \cite{AtiyahBott, Sa:AF, W:survey}.
Indeed, $\omega$ is bilinear and alternating (recall that $\alpha_1 \wedge \alpha_2 = - \alpha_2 \wedge \alpha_1$ for real-valued 1-forms), and it is nondegenerate since the Hodge star operator for any choice of metric on $\Sigma$ induces an $L^2$-metric $g(a_1,a_2)=\omega(a_1,* a_2)$ on $\cA(\Sigma)$. 

Note here that reversing the orientation of $\Sigma$ corresponds to reversing the sign of the symplectic form, i.e.\ $\cA(\Sigma^-)=\cA(\Sigma)^-$. 
Moreover, $*|_{\cA(\Sigma)}$ is in fact an $\omega$-compatible complex structure since $*^2=-\id$.
\item
To each diffeomorphism $\phi:\Sigma_0\to\Sigma_1$, we associate the push forward $L_\phi:=\phi_*: \cA(\Sigma_0)\to \cA(\Sigma_1)$ given by $(\phi_*a )(v) := a({\rm d}\phi^{-1}(v))$.
This is a symplectomorphism since for $a_1,a_2\in \cA(\Sigma_0)$ we have
\begin{align*}
\bigl(L_\phi^*\omega_{\cA(\Sigma_1)} \bigr)(a_1,a_2)
&= \tint_{\Sigma_1} \langle \phi_*(a_1) \wedge \phi_*(a_2) \rangle
= \tint_{\Sigma_1} \phi_* \langle a_1 \wedge a_2 \rangle \\
&=\tint_{\phi^{-1}(\Sigma_1)}  \langle a_1 \wedge a_2 \rangle 
=\omega_{\cA(\Sigma_0)}(a_1,a_2) .
\end{align*}
Moreover we have $L_\phi \circ L_\psi =\phi_* \circ \psi_* = (\phi\circ\psi)_* = L_{\phi\circ\psi}$ as required
when $\phi,\psi$ are composable.
\item
To each 2-handle attachment $Y_\alpha\in\Mor_{\Bor^{\rm conn}_{2+1}}(\Sigma,\Sigma')$, we associate the space of restrictions of flat connections on $Y_\alpha$ to the boundary components $\partial Y_\alpha = \Sigma^- \sqcup \Sigma'$,
$$
{\mathcal L}(Y_\alpha):=
\bigl\{\bigl(\tilde A |_{\Sigma}, \tilde A |_{\Sigma'}\bigr) \,|\, \tilde A \in \cA(Y_\alpha), F_{\tilde A}=0\} \; \subset\; \cA(\Sigma)^-\times \cA(\Sigma').
$$ 
This yields an isotropic of $\cA(\Sigma)^-\times \cA(\Sigma') \cong \cA(\Sigma^-\sqcup \Sigma')=\cA(Y_\alpha)$ since the linearization of curvature $\frac{\rm d}{{\rm d}t}\big|_{t=0} F_{\tilde A + t \tilde a}={\rm d}_{\tilde A} \tilde a$ at a connection $\tilde A$ is the associated differential, so that 
$\omega(\tilde a_1|_{\partial Y}, \tilde a_2|_{\partial Y}) = \int_Y \langle {\rm d}_{\tilde A} \tilde a_1 \wedge \tilde a_2 \rangle - \langle \tilde a_1 \wedge {\rm d}_{\tilde A} \tilde a_2 \rangle = 0$ by Stokes' theorem.
In appropriate Banach space completions, one can also show that ${\mathcal L}(Y_\alpha)$ is a Banach submanifold and coisotropic, hence a Lagrangian submanifold of $\cA(\Sigma)^-\times \cA(\Sigma')$.
(This is a direct generalization of \cite[Lemma~4.6]{W:Banach} which proves these claims for $Y_\alpha$ replaced by a handlebody.)

\item[(3')]
The analogous construction for the 1-handle attachment $Y_\alpha^-\in\Mor_{\Bor^{\rm conn}_{2+1}}(\Sigma',\Sigma)$ yields the transposed Lagrangian 
$$
{\mathcal L}(Y_\alpha^-):=
\bigl\{\bigl(\tilde A |_{\Sigma'}, \tilde A |_{\Sigma}\bigr) \,|\, \tilde A \in \cA(Y^-_\alpha), F_{\tilde A}=0\} 
\;=\; {\mathcal L}(Y_\alpha)^T .
$$ 

\item
To check $( \phi_* \times\phi'_*) (\cL(Y_\alpha)) = \cL(Y_{\phi(\alpha)})$ for a diffeomorphism $\phi:\Sigma\to\Sigma$, recall that $\phi,\phi'$ are the boundary restrictions of a diffeomorphism $\widetilde\phi:Y_\alpha \to Y_{\phi(\alpha)}$. Then the relation between the Lagrangians follows from the fact that the spaces of flat connections on $Y_\alpha$ and $Y_{\phi(\alpha)}$ are identified by pullback with $\widetilde\phi$.

\item[(5,6)]
For any composable pair of cobordisms $Y_{ij}\in\Mor_{\Bor^{\rm conn}}(\Sigma_i,\Sigma_j)$, we have
the Lagrangian for the composition of cobordisms given by the geometric composition of the Lagrangians for the separate cobordisms, 
\begin{align*}
& \cL(Y_{01}\cup_{\Sigma_1} Y_{12})  \\
&\; =
\bigl\{(A_0,A_2) \,|\, \exists\, \tilde A_{\rm flat} \in \cA(Y_{01}\cup_{\Sigma_1} Y_{12}), 
\tilde A |_{\Sigma_i}=A_i
\bigr\}  \\
&\;=
\bigl\{(A_0,A_2) \,|\, \exists\, \tilde A_{ij} \in \cA_{\rm flat}(Y_{ij}), 
\tilde A_{01}|_{\Sigma_1}=\tilde A_{12}|_{\Sigma_1}, 
\tilde A_{01} |_{\Sigma_0}=A_0, \tilde A_{12} |_{\Sigma_2}=A_2
\bigr\} \\
&\;=
\pi_{\cA(\Sigma_0)\times\cA(\Sigma_2)}\bigl( \bigl(\cL(Y_{01}) \times \cL(Y_{12})\bigr)  \cap\bigl( \cA(\Sigma_0) \times \Delta_{\cA(\Sigma_1)} \times \cA(\Sigma_2)\bigr) \bigr) \\
&\;=
\cL(Y_{01}) \circ \cL(Y_{12}), 
\end{align*}
where we denote the sets of flat connections by ${\cA_{\rm flat}(Y):=\{ \tilde A \in \cA(Y) \,|\, F_{\tilde A}=0\}}$.
This proves all required identities of geometric compositions. However, these geometric compositions are never embedded since all restrictions of the connections to $\Sigma_1$ are flat, thus cannot span the complement of the diagonal.
\end{enumerate} 
While these constructions do not yield a functor $\Bor^{\rm conn}_{2+1}\to\Symp$ via the principle of Remark~\ref{rmk:field2}, we will explain in Example~\ref{ex:inst2cat} how one might use quilts (see \S\ref{ss:symp2}) made up of ASD instantons in place of pseudoholomorphic curves to extend this partial functor to a Floer field theory $\Bor^{\rm conn}_{2+1}\to\Cat$ that factors through a symplectic instanton 2-category whose objects are symplectic Banach spaces of connections.
\end{example}

The beginning of an instanton Floer field theory given above is the natural intermediate step in an expected relation between Chern-Simons theory on 3-manifolds and symplectic invariants arising from a choice of decomposition of the 3-manifold as formulated by Atiyah \cite{Atiyah} in terms of Floer homologies \cite{Floer:inst,Floer:Lag}; see also \cite{Sa:AF,W:survey} and \S\ref{ss:af}.
This symplectic invariant uses Heegaard splittings as explained before Example~\ref{ex:sym} and finite dimensional symplectic quotients of the above spaces of connections, as explained in the following remark.
Moreover, the Chern-Simons theory on 3-manifolds is naturally coupled with Donaldson-Yang-Mills theory on 4-manifolds; see \cite{Don:inv,DK,Don:book}. 
Thus the subsequent sketch of Floer field theories arising from representation spaces should be viewed as the beginning of a symplectic categorification of Donaldson-Yang-Mills theory (in various versions, depending on choice of group and twisting). It also serves as a purely symplectic explanation of the conjecture that the Floer homology arising from a decomposition of the 3-manifold is in fact a 3-manifold invariant, i.e.\ independent of the choice of decomposition; see \S\ref{ss:af} for details.

\begin{remark}[Finite dimensional reduction of instanton Floer field theory]  \rm \label{rmk:quotient}
While the spaces of connections in Example~\ref{ex:infdim} are infinite dimensional and tend to have a smooth structure, a symplectic reduction by the Hamiltonian action of the gauge group yields finite dimensional but generally singular spaces. 
Here the gauge group $\cG(\Sigma)=\cC^\infty(\Sigma,G)$ acts on $\cA(\Sigma)$ by pulling back connections with bundle isomorphisms, and its moment map is the curvature; see \cite{AtiyahBott, Sa:AF, W:survey}.
The symplectic quotient $M_\Sigma:=\cA(\Sigma)\squ{\mathcal G}(\Sigma)$ can thus be understood topologically as the space of representations of the fundamental group $\pi_1(\Sigma)$ in the Lie group $G$ -- given by the holonomies of flat connections -- modulo gauge symmetries represented by simultaneous conjugation of the holonomies.
The quotient\footnote{
The fact that we can take the quotient by the product of gauge groups is due to the identification
${\mathcal G}(\Sigma)\times {\mathcal G}(\Sigma') = \cC^\infty(\partial Y_\alpha, G) = \cC^\infty(Y_\alpha,G)|_{\partial Y_\alpha}$ with the boundary values of the gauge group $\cG(Y_\alpha)$, which uses the assumption of $G$ being connected and simply connected.} 
of the Lagrangian $L_{Y_\alpha}:={\mathcal L}(Y_\alpha)/{\mathcal G}(\Sigma)\times {\mathcal G}(\Sigma') \subset M_\Sigma^-\times M_{\Sigma'}$ is given by those representations that arise as the restriction of a representation of $\pi_1(Y_\alpha)$, i.e.\ yield the identity when evaluated on loops in $\partial Y_\alpha =\Sigma^-\sqcup\Sigma'$ that are contractible in~$Y_\alpha$.

Singularities in these spaces are due to reducible connections, corresponding to representations $\rho:\pi_1(\Sigma)\to G$ on which conjugation by $G$ acts with nondiscrete stabilizer $G_\rho=\{g\in G \,|\, g^{-1}\rho g=\rho\}$ (e.g.\ the stabilizer of the trivial representation is the whole group $G$).
These can be avoided by working on appropriately twisted bundles or making holonomy requirements around punctures in $\Sigma$ resp.\ tangles\footnote{A tangle in a cobordism is an embedded submanifold whose boundary coincides with given punctures on the boundary of the cobordism.} in $Y_\alpha$. (The latter usually yields field theories for cobordisms with tangles, but there are specific -- central in G -- holonomy requirements for which the position of puncture resp.\ tangle is irrelevant.)
Then the symplectic quotient by the gauge group $\cG(\Sigma)$ yields a finite dimensional Lagrangian submanifold $L_{Y_\alpha}\subset M_\Sigma^- \times M_{\Sigma'}$.
\end{remark}

Instead of discussing possible twisting constructions to avoid the reducibles noted above, the following example gives an idea of a finite dimensional Floer field theory in terms of sets rather than manifolds.
For abelian groups $G$, this will actually yield smooth symplectic and Lagrangian manifolds, but a field theory based on these would only capture homological information of the bordism category.

\begin{example}[Naive Floer field theory from representation spaces] \label{ex:rep} \rm
We will go through the Floer field theory construction outlined in Remarks~\ref{rmk:field2} in the example of representations of a compact, connected, simply connected Lie group $G$ such as $G=SU(2)$, which arise from trivial $G$-bundles in Example~\ref{ex:infdim} and Remark~\ref{rmk:quotient}.

\begin{enumerate}
\item
To each closed, connected, oriented surface $\Sigma$, associate the representation space
$$
M_\Sigma := \quotient{ \bigl\{ \rho \in \Hom(\pi_1(\Sigma), G ) \bigr\} }{\sim} 
\quad\text{with}\quad \rho\sim\rho' \, :\Leftrightarrow \; \exists g\in G : \rho' = g^{-1}\rho g .
$$
Any standard basis $(\alpha_1,\beta_1,\ldots,\alpha_g,\beta_g)$ for $\pi_1(\Sigma)$, i.e.\ loops that 
are disjoint except for single transverse intersection points $\alpha_i\pitchfork\beta_i$ and whose concatenation $\prod_{i=1}^g \alpha_i\beta_i\alpha_i^{-1}\beta_i^{-1}$ is homotopic to the constant loop, yields an identification
$$
M_\Sigma \simeq \quotient{\bigl\{ (a_1,b_1,\ldots, a_g,b_g) \in G^{2g} \,\big|\, {\textstyle\prod_{i=1}^g} a_i b_i a_i^{-1} b_i^{-1} = \id\bigr\}}{ \sim } 
$$
(with $\id\in G$ denoting the identity),
modulo simultaneous conjugation
$$
(a_i,b_i)_{i=1,\ldots,g} \sim (g^{-1} a_i g, g^{-1} b_i g)_{i=1,\ldots,g} \qquad \forall g\in G .
$$
\item
To each diffeomorphism $\phi:\Sigma_0\to\Sigma_1$ associate the map $L_\phi: M_{\Sigma_0}\to M_{\Sigma_1}$ which maps $\rho\in M_{\Sigma_0}$ to the representation $L_\phi(\rho)\in M_{\Sigma_1}$
given by $[\gamma\,] \mapsto \rho([\phi^{-1}\circ\gamma\,])$ for any circle $\gamma:S^1\to\Sigma_1$. 
Observe that $L_\phi \circ L_\psi=L_{\psi\circ\phi}$ when $\phi,\psi$ are composable.
\item
For each attaching circle $\alpha\subset\Sigma$ we use the bijection $\pi_\alpha:\Sigma\less\alpha\to\Sigma'\less\{\text{2 points}\}$ and a deformation of any loop $\gamma:S^1\to\Sigma'$ to avoid the special points to construct
$$
L_\alpha := \bigl\{ \bigl([\rho],[\rho']\bigr) \in M_\Sigma^-\times M_{\Sigma'} \,\big|\, 
\rho([\alpha])=\id , \forall \gamma:  \rho'([\gamma\,])=\rho([\pi_\alpha^{-1}\circ\gamma\,])\bigr\}.
$$ 
Note that this construction is independent of the choice of a parametrization $\alpha:S^1\to \Sigma$ of the attaching circle (and deformation to ${\alpha(1)=z}$).
In the identification obtained from a standard basis $(\alpha_i,\beta_i)_{i=1,\ldots,g}$ for $\Sigma$ with $[\alpha_1]=[\alpha]$ and the induced basis $(\pi_\alpha\circ\alpha_i,\pi_\alpha\circ\beta_i)_{i=2,\ldots,g}$ for $\Sigma'$ we have
$$
L_\alpha = \bigl\{ \bigl( [(a_i,b_i)_{i=1,\ldots,g}],  [(a'_i,b'_i)_{i=2,\ldots,g}] \bigr) \,\big|\, a_1=\id, \forall i\ge 2: a'_i=a_i, b'_i=b_i\bigr\}.
$$
\item[(3')]
The analogous construction for  the adjoint cobordism
$Y_\alpha^-\in\Mor_{\Bor^{\rm conn}_{2+1}}(\Sigma',\Sigma)$ 
yields the transposed Lagrangian 
$L_\alpha^T\subset M_{\Sigma'}^-\times M_\Sigma$.
\item
For any attaching circle $\alpha: S^1 \to \Sigma$ and diffeomorphism $\phi:\Sigma\to\Sigma$ we can rewrite $\rho'([\gamma\,])=\rho([\pi_{\phi(\alpha)}^{-1}\circ\gamma\,])$ in the construction of $L_{\phi(\alpha)}$ equivalently as $\rho'([\phi'\circ\tilde\gamma\,])=\rho([\phi\circ \pi_\alpha^{-1}\circ\tilde\gamma\,])$ for all loops $\tilde\gamma$ since $\pi_{\phi(\alpha)}\circ\phi = \phi' \circ\pi_\alpha$, and thus
\begin{align*}
L_{\phi(\alpha)} 
&= \bigl\{ \bigl([\rho],[\rho'] \bigr) \in M_\Sigma^-\times M_{\Sigma'} \,\big|\, 
\rho([\phi\circ\alpha])=\id , \rho'([\phi'\circ\gamma\,])=\rho([\phi\circ \pi_\alpha^{-1}\circ\gamma\,])\bigr\} \\
&= \bigl\{ \bigl(L_\phi([\tilde\rho]),L_{\phi'}([\tilde\rho']) \bigr) \in M_\Sigma^-\times M_{\Sigma'} \,\big|\, 
\tilde\rho([\alpha])=\id , \tilde\rho'([\gamma\,])=\tilde\rho([\pi_\alpha^{-1}\circ\gamma\,])\bigr\} \\
&= ( L_\phi \times L_{\phi'}) (L_\alpha) .
\end{align*}
\item
For disjoint attaching circles $\alpha\cap\beta=\emptyset$ we calculate the geometric composition
\begin{align*}
L_\alpha \circ L_{\beta'} 
&= \bigl\{ \bigl([\rho ] , [\rho'' ]\bigr) \,\big|\, \exists [\rho' ]\in M_{\Sigma_\alpha} : 
\bigl([\rho] , [\rho']\bigr) \in L_\alpha, \bigl([\rho' ] , [\rho'' ]\bigr) \in L_{\beta'} \bigr\} \\
&=\bigl\{ \bigl([\rho ] , [ \rho'' ]\bigr) \,\big|\, 
\rho([\alpha])=\rho([\beta])=\id, 
\rho''([\gamma\,])=\rho([(\pi_{\beta'}\pi_\alpha)^{-1}\circ\gamma\,])
\bigr\}
\end{align*}
by noting that $[\rho'']$ is determined from $[\rho]$ by
$$
\rho''([\gamma\,])=\rho'([\pi_{\beta'}^{-1}\circ\gamma\,])
=\rho([\pi_\alpha^{-1}\circ\pi_{\beta'}^{-1}\circ\gamma\,])
\qquad
\forall\; \gamma:S^1\to \Sigma'':=(\Sigma_\alpha)_{\beta'}
$$ 
and the additional requirement
$\id=\rho'([\beta'])=\rho([\pi_\alpha^{-1}\circ\beta'\,])$, 
where we have $\pi_\alpha^{-1}(\beta')=\beta$ because the attaching circles are disjoint.
Analogously, in the composition $L_\beta \circ L_{\alpha'}=\bigl\{ \bigl([\rho ] , [ \rho'']\bigr) \,\big|\, \ldots\bigr\}$
we have $\rho''([\gamma\,])=\rho([(\pi_{\alpha'}\pi_\beta)^{-1}\circ\gamma\,])$
for all $\gamma:S^1\to (\Sigma_\beta)_{\alpha'}$.
Using the diffeomorphism $\phi''$ given by 
$\phi''\circ \pi_{\beta'}\pi_\alpha = \pi_{\alpha'} \pi_\beta$, 
we rewrite this as 
$\rho''([\phi''\circ\tilde\gamma\,]) =\rho([(\pi_{\beta'}\pi_\alpha)^{-1}\circ \tilde\gamma\,])$
for all $\tilde\gamma=(\phi'')^{-1}\circ\gamma$ so that we obtain the first identity in \eqref{eq:referee},
\begin{align*}
L_\beta \circ L_{\alpha'} 
&=\bigl\{ \bigl([\rho ] , [ \rho'']\bigr) \,\big|\, 
\rho([\beta])=\rho([\alpha])=\id, 
\rho''([\phi''\circ\gamma\,]) =\rho([(\pi_{\beta'}\pi_\alpha)^{-1}\circ \gamma\,])
\bigr\} \\
&= (\id\times L_{\phi''})\bigl( L_\alpha \circ L_{\beta'} \bigr).
\end{align*}
The second identity between geometric compositions of Lagrangians is similar:
\begin{align*}
&(\id\times L_{\phi''})(L_{\beta'}) \circ L_{\alpha'}^T \\
&=\bigl\{ \bigl([\rho' ] , [\sigma' ]\bigr) \,\big|\, \exists [\rho'']\in M_{(\Sigma_\alpha)_{\beta'}} : 
\bigl([\rho'] , [\rho'']\bigr) \in L_{\beta'}, \bigl([\sigma'] , L_{\phi''}([\rho''])\bigr) \in L_{\alpha'} \bigr\} \\
&=\bigl\{ \bigl([\rho' ] , [\sigma' ]\bigr) \,\big|\,
\rho'([\beta'])=\sigma'([\alpha'])=\id, 
\rho'([\pi_{\beta'}^{-1}\circ\gamma\,]) =\sigma'([\pi_{\alpha'}^{-1}\circ\phi''\circ\gamma\,])
\;\forall \gamma \bigr\} \\
&=\bigl\{ \bigl([\rho'] , [\sigma']\bigr) \,\big|\, 
\rho([\alpha])=\rho([\beta]) =\id, 
 \rho'=\rho([\pi_\alpha^{-1}\circ\ldots\,]) , 
 \sigma'=\rho([\pi_\beta^{-1}\circ\ldots\,])  \bigr\} \\
&=\bigl\{ \bigl([\rho'] , [\sigma']\bigr) \,\big|\, \exists [\rho]\in M_\Sigma : 
\bigl([\rho] , [\rho']\bigr) \in L_\alpha, \bigl([\rho] , [\sigma'])\bigr) \in L_\beta \bigr\}
\; = \; L_\alpha^T \circ L_\beta ,
\end{align*}
where the first composition requires $\rho'([\beta'])=\id=\sigma'([\alpha'])$ in addition to 
$$
\rho'([\pi_{\beta'}^{-1}\circ\gamma\,]) = \rho''([\gamma\,])=\sigma'([\pi_{\alpha'}^{-1}\circ\phi''\circ\gamma\,])
\qquad
\forall\; \gamma:S^1\to (\Sigma_\alpha)_{\beta'}.
$$
Using $\phi''\circ \pi_{\beta'}\circ \pi_\alpha = \pi_{\alpha'} \circ\pi_\beta$ we can rewrite this as
$$
\rho'([\tilde\gamma]) =\sigma'([\pi_{\alpha'}^{-1}\circ\phi''\circ{\pi_\beta'}\circ\tilde\gamma\,])
=\sigma'([\pi_\beta\circ\pi_\alpha^{-1}\circ\tilde\gamma\,])
\qquad\forall\; \tilde\gamma:S^1\to \Sigma_\alpha \less \beta' ,
$$
i.e.\ the conditions in $L_\alpha^T \circ L_\beta$ for these loops, which also correspond to the loops in $\Sigma_\beta\less \alpha'$.
In addition, this second geometric composition requires 
$\rho'([\beta'])=\rho([\beta])=\id$, $\sigma'([\alpha'])=\rho([\alpha])=\id$, which identifies it with the first 
composition.

Note here that either one of the representations $\bigl([\rho'] , [\sigma']\bigr)\in L_{\beta}'^T \circ L_{\alpha'}$ of $\pi_1(\Sigma_\beta)$ or $\pi_1(\Sigma_\alpha)$ fully determines the intermediate representation $[\rho'']$ of $(\Sigma_\alpha)_{\beta'}$. This can also be seen from the fact that $\pi_{\beta'}$ (as well as $\pi_{\alpha'}$) acts surjectively on fundamental groups, in fact any loop in $Y_{\beta'}$ (not just in $(\Sigma_\alpha)_{\beta'}\subset\partial Y_{\beta'}$) can be homotoped into the boundary component $\Sigma_\alpha$ of higher genus.
This uniqueness of the intermediate representations proves injectivity of the projection in the geometric compositions, and -- if there was a smooth structure -- the corresponding infinitesimal fact would also prove transversality of the intersection, thus embeddedness of the geometric composition $L_{\beta'}^T \circ L_{\alpha'}$.

Embeddedness of $L_\beta \circ L_{\alpha'}$ resp.\ $L_\alpha \circ L_{\beta'}$ analogously follows from $\pi_1$-surjectivity of $\pi_\beta$ resp.\ $\pi_\alpha$. 
For the last geometric composition corresponding to the gluing of cobordisms $Y_\alpha^-\cup_\Sigma Y_\beta$ at the highest genus surface $\Sigma$, the fact that the intermediate representation $[\rho]$ on $\Sigma$ is determined by the representations $\bigl([\rho'] , [\sigma']\bigr)\in L_\alpha^T \circ L_\beta$ on the two lower genus surfaces $\Sigma_\alpha$, $\Sigma_\beta$, is not evident from the formulas. In fact, it is false if we allow $\alpha,\beta$ to be homologous. However, this is excluded by the assumption of all surfaces, in particular $(\Sigma_\alpha)_{\beta'}\simeq (\Sigma_\beta)_{\alpha'}$ being connected.
Thus we can choose a standard basis $(\alpha_1,\beta_1,\ldots,\alpha_g,\beta_g)$ for $\pi_1(\Sigma)$ with $\alpha_1=\alpha$ and $\beta_g=\beta$ to see that points in $L_\alpha^T  \circ L_\beta$ have the form $\bigl( [(a_i,b_i)_{i=2,\ldots,g}],  [(a_i,b_i)_{i=1,\ldots,g-1}]$, which determines the indermediate $[(a_i,b_i)_{i=1,\ldots,g}]\in M_\Sigma$ uniquely.

\item
For attaching circles $\alpha,\beta\subset \Sigma$ with unique transverse intersection point we can choose a standard basis $(\alpha_1,\beta_1,\ldots,\alpha_g,\beta_g)$ for $\pi_1(\Sigma)$ with $\alpha_1=\alpha$ and $\beta_1=\beta$. 
Then $L_\alpha^T \circ L_\beta$ is given by pairs $\bigl(\bigl[(a_i,b_i)_{i=2,\ldots g}\bigr] , \bigl[(a'_i,b'_i)_{i=2,\ldots g} \big] \bigr)\in M_{\Sigma_\alpha}^-\times M_{\Sigma_\beta}$
for which -- after conjugation of the representative $(a'_i,b'_i)_{i=2,\ldots g}$ -- there exists $[(a_i,b_i)]_{i=1,\ldots g} \in M_\Sigma$ such that $a_1 = b_1 = \id$ and $a_i=a'_i, b_i=b'_i$ for $i\ge 2$.
That is, in this basis $L_\alpha^T \circ L_\beta$ is identified with the diagonal over the identified representation spaces $M_{\Sigma_\alpha} \simeq M_{\Sigma_\beta}$.
Since this identification is by the map $L_\phi$, it shows the identity $L_\alpha^T \circ L_\beta = {\rm gr}(L_\phi)$.
Moreover, in the presence of a smooth structure, the geometric composition $L_\alpha^T \circ L_{\beta}$ would be embedded since the intermediate point $[(a_i,b_i)]_{i=1,\ldots g} \in M_\Sigma$ is uniquely determined.
\end{enumerate} 
\end{example}

\begin{remark}[Rigorous Floer field theories from representation spaces] \rm \label{rmk:rigrep}
Even for the simplest nonabelian group $G=SU(2)$, the representation space for the torus $\Sigma=T^2$ in Example~\ref{ex:rep} is the pillowcase $M_{T^2}\simeq S^1\times S^1/\Z_2$ (here $\Z_2$ acts on each factor $S^1$ by reflection with two fixed points), and more complicated representation spaces may not even be orbifolds. In some simple cases, e.g.\ in \cite{hedden} for knots represented by Lagrangians in the pillowcase, one can deal explicitly with these singularities. 
To obtain a full Floer field theory, \cite{ww:fielda} replaces moduli spaces of flat $G$-connections with moduli spaces of central-curvature connections on unitary bundles with fixed determinant and coprime rank $r$ and degree $d$. For $r=2$, $d=1$ this corresponds to flat connections on nontrivial $SO(3)$-bundles, which can also be viewed as taking the above representation spaces for $G=SU(2)$ on a punctured surface $\Sigma\less\pt$, 
and instead of holonomy $\id$ requiring $-\id$ around the puncture. This yields monotone symplectic manifolds
$$
\widehat M_\Sigma \simeq \quotient{\bigl\{ (a_i,b_i)_{i=1,\ldots,g} \in SU(2)^{2g} \,\big|\, 
{\textstyle \prod_{i=1}^g} a_i b_i a_i^{-1} b_i^{-1} = -\id\bigr \} }{\sim} .
$$ 
If instead of $-\id$ we replace $\id$ with a non-central element $k\in G$, then the representation spaces for the cobordisms are no longer independent of the choice of paths connecting the punctures on the surface (around which the holonomy is required to be conjugate to $k$). The corresponding Floer field theory in \cite{ww:fieldb} thus yields invariants for pairs of cobordisms with embedded tangles (though invariance under isotopies of the embedding is not yet discussed, so the field theory falls short of yielding knot or link invariants). 
\end{remark}

Just as dimensional reductions of Donaldson-Yang-Mills theory give rise to the Atiyah-Floer conjecture, the Seiberg-Witten theory for 4-manifolds motivated the development of Heegaard-Floer homology by Ozsv\'ath-Szab\'o \cite{OS1}. 
Since a 2-dimensional reduction of the Seiberg-Witten equations gives rise to vortex equations, whose moduli spaces of solutions can be identified with symmetric products of the ambient space \cite{Gar}, they arrived at a 3-manifold invariant that on a given 3-manifold $Y$ is constructed by choosing a so-called Heegaard splitting $Y=H_0^-\cup_\Sigma H_1$ into two handlebodies,\footnote{A handlebody is a 3-manifold $H$ with boundary $\partial H=\Sigma$  (i.e.\ a cobordism from $\Sigma$ to the empty set), which is obtained from handle attachments along a maximal number of disjoint attaching circles $\alpha_1,\ldots,\alpha_g\subset \Sigma$ that are homologically independent.} 
representing the handlebodies by Lagrangians $L_{H_i}\subset M_\Sigma={\rm Sym}^g(\Sigma)$ in the symmetric product of the dividing surface $\Sigma$, and taking Floer theoretic invariants of the pair $L_{H_0}, L_{H_1}\subset M_\Sigma$.
Here and throughout, $g$ will denote the genus of the present surface $\Sigma$.
Since Heegaard splittings are not unique by any means, Ozsv\'ath-Szab\'o had to explicitly compare holomorphic curves in symmetric products of different surfaces to prove that the Heegaard-Floer homology groups $HF(L_{H_0},L_{H_1})$ (with ``plus/minus/hat'' decorations arising from keeping track of intersections with a marked point in $\Sigma$) are in fact 3-manifold invariants, i.e.\ independent of the choice of splitting. 

There are several more conceptual explanations of this independence. Firstly, \cite{KLT} recently proved an Atiyah-Floer type identification of $HF(L_{H_0},L_{H_1})$ with monopole Floer homology -- the 3-manifold invariant arising directly from Seiberg-Witten gauge theory \cite{KM} . Secondly, as explained in \S\ref{ss:af}, an extension of Heegaard-Floer homology to a 2+1 Floer field theory would also reproduce the Heegaard-Floer 3-manifold invariant.
In addition, this would provide a symplectic categorification of Seiberg-Witten theory.
Perutz established the basics of such a theory by constructing Lagrangian matching invariants \cite{perutz2} for 4-manifolds equipped with broken Lefshetz fibrations, which are expected to be equal to the Seiberg-Witten invariants, in particular independent of the choice of broken fibration. 
The core of this approach is a construction in \cite{perutz1} of Lagrangians in symmetric products associated to simple 3-cobordisms, whose basic structure we explain in the following.

\begin{example}[Naive Floer field theory from symmetric products] \label{ex:sym} \rm
We will use the steps in Remark~\ref{rmk:field2} to outline the extension of Heegaard-Floer homology to a Floer field theory as proposed in \cite{perutz1,lekili,lp} for any fixed $n\ge 0$ (or $n<0$ with surfaces restricted to genus $g\ge -n$).
To avoid dealing with complex geometry, we will work with a naive version of symmetric products in which they are constructed as sets rather than smooth algebraic varieties.
The smooth, symplectic, and Lagrangian structures are discussed in \cite{perutz1}.

\begin{enumerate}
\item
To each closed, connected, oriented surface $\Sigma$, associate the symmetric product 
$$
M_\Sigma:={\rm Sym}^{g + n}\Sigma \;=\; \quotient{\Sigma^{g+n}}{S_{g+n}}
\;=\; \quotient{\Sigma\times\ldots\times\Sigma}{\small (z_1,\ldots,z_{g+n}) \sim (z_{\sigma(1)}, \ldots, z_{\sigma(g+n)})},
$$ 
where $g$ is the genus of $\Sigma$, and $S_{g+n}$ is the symmetric group acting by permutations $\sigma:\{1,\ldots,g+n\} \to\{1,\ldots,g+n\}$.
On the complement of the diagonal $\Delta\subset\Sigma^{g+n}$ (where two or more points coincide) this is a smooth quotient, but to obtain a global smooth structure it has to be viewed as the symmetric product of an algebraic curve. This requires the choice of a complex structure on $\Sigma$, and the symplectic structure is an additional choice -- induced by the broken fibration in \cite{perutz1} -- all of which we suppress here.
\item
To each diffeomorphism $\phi:\Sigma_0\to\Sigma_1$ associate the map
$$
L_\phi: M_{\Sigma_0}\to M_{\Sigma_1}, \quad \bigl[\bigl(z_1,\ldots,z_{g+n}\bigr)\bigr] \mapsto  \bigl[\bigl(\phi(z_1),\ldots,\phi(z_{g+n})\bigr)\bigr]
$$ 
and observe that $L_\phi \circ L_\psi=L_{\phi\circ\psi}$ when $\phi,\psi$ are composable.
This yields a smooth map when $\phi$ is holomorphic in the chosen complex structures on $\Sigma_i$, but in general this naive construction only yields the correct map outside of the diagonal.
\item
To each attaching circle $\alpha\subset\Sigma$ associate $L_\alpha \subset M_\Sigma^-\times M_{\Sigma'}$  given by
$$
L_\alpha := \bigl\{  \bigl( \bigl[(z_1,\ldots,z_{g+n})\bigr] , \bigl[(\pi_\alpha(z_2),\ldots,\pi_\alpha(z_{g+n}) )\bigr]\bigr) \,\big|\, z_1\in\alpha, z_2,\ldots,z_{g+n}\in \Sigma\less\alpha
\bigr\} .
$$ 
Note that this naively constructed subset is not even closed, let alone a smooth submanifold. However, \cite{perutz1} rigorously constructs Lagrangian submanifolds $\widehat V_\alpha$ that are smoothly isotopic to $L_\alpha$ on the subset $U_0\cup U_1\subset {\rm Sym}^{g + n}\Sigma$ given by tuples with up to one point in a given tubular neighbourhood $\widetilde\alpha\subset\Sigma$ of $\alpha$.
Thus it makes some sense to discuss the field theory construction in this model. 
\item[(3')]
To the adjoint 1-handle attachment 
$Y_\alpha^-\in\Mor_{\Bor^{\rm conn}_{2+1}}(\Sigma',\Sigma)$ 
we associate the transposed Lagrangian 
$L_\alpha^T\subset M_{\Sigma'}^-\times M_\Sigma$.
\item
For an attaching circle $\alpha \subset\Sigma$ and diffeomorphism $\phi:\Sigma\to\Sigma$ note that we have $( L_\phi \times L_{\phi'}) (L_\alpha) = L_{\phi(\alpha)}$ because $z'_i:=\phi(z_i)$ yields an identification
\begin{align*}
&
\bigl\{  \bigl( \bigl[\bigl(\phi(z_i)\bigr)_{i=1,\ldots g+n}\bigr] , \bigl[\bigl(\phi'(\pi_\alpha(z_i))\bigr)_{i=2,\ldots g+n} \bigr] \bigr) \,\big|\, z_1\in\alpha, z_{i\ge 2}\in \Sigma\less\alpha
\bigr\} \\
&= \bigl\{  \bigl( \bigl[ (z'_i )_{i=1,\ldots g+n}\bigr] , \bigl[\bigl(\pi_{\phi(\alpha)}(z'_i)\bigr)_{i=2,\ldots g+n} \bigr]\bigr) \,\big|\, z'_1\in\phi(\alpha), z'_{i\ge 2}\in \Sigma\less\phi(\alpha)
\bigr\}.
\end{align*}
In \cite[2.3.1]{perutz1}, actual symplectomorphisms are associated to diffeomorphisms $\phi$ that arise from parallel transport in a broken fibration.
\item
For disjoint attaching circles $\alpha\cap\beta=\emptyset$ the bijectivity of $\pi_\beta:\Sigma\less\beta \to \Sigma_\beta\less\{\text{2 points}\}$ implies 
$x_1\in\alpha \Leftrightarrow \pi_\beta(x_1)\in\alpha'$ and analogously 
${x_2\in\beta \Leftrightarrow \pi_\alpha(x_2)\in\beta'}$. Thus
\begin{align*}
L_\alpha \circ L_{\beta'} 
&= \bigl\{ \bigl([\, \ul x\, ] , [\, \ul z\, ]\bigr) \,\big|\, \exists [\, \ul y\, ]\in M_{\Sigma_\alpha} : 
\bigl([\, \ul x\,] , [\,\ul y\,]\bigr) \in L_\alpha, \bigl([\,\ul y\,] , [\,\ul z\,]\bigr) \in L_{\beta'} \bigr\}
\\
&=\bigl\{
 \bigl( \bigl[(x_i)_{i=1,\ldots g+n}\bigr] , \bigl[(z_i)_{i=3,\ldots g+n} \big]\bigr) \,\big|\, x_1\in\alpha, \pi_\alpha(x_2)\in\beta', 
z_i=\pi_{\beta'}(\pi_\alpha(x_i))
\bigr\}, \\
L_\beta \circ L_{\alpha'} 
&=\bigl\{
 \bigl( \bigl[(x_i)_{i=1,\ldots g+n}\bigr] , \bigl[(z_i)_{i=3,\ldots g+n} \big]\bigr) \,\big|\, x_2\in\beta, \pi_\beta(x_1)\in\alpha', z_i=\pi_{\alpha'}(\pi_\beta(x_i))
\bigr\}
\end{align*}
are related via $\id \times L_{\phi''}$ by the defining property 
$\phi''\circ \pi_{\beta'}\circ\pi_\alpha = \pi_{\alpha'}\circ \pi_\beta$ of $\phi''$.
The second identity between geometric compositions of Lagrangians is similar:
\begin{align*}
&(\id\times L_{\phi''})(L_{\beta'}) \circ L_{\alpha'}^T \\
&=\bigl\{ \bigl([\, \ul x\, ] , [\, \ul z\, ]\bigr) \,\big|\, \exists [\, \ul v\, ]\in M_{(\Sigma_\alpha)_{\beta'}} : 
\bigl([\, \ul x\,] , [\,\ul v\,]\bigr) \in L_{\beta'}, \bigl([\,\ul z\,] , L_{\phi''}([\,\ul v\,])\bigr) \in L_{\alpha'} \bigr\} \\
&=\bigl\{
 \bigl( \bigl[(x_i)_{i=2,\ldots g+n}\bigr] , \bigl[(z_i)_{i=2,\ldots g+n} \big] \bigr)\,\big|\, x_2\in\beta', 
 z_2 \in \alpha', \phi''(\pi_{\beta'}(x_i))=\pi_{\alpha'}(z_i) \;\forall i\ge 3 \bigr\} \\
&=\bigl\{ \bigl([\, \ul x\, ] , [\, \ul z\, ]\bigr) \,\big|\, \exists [\, \ul y\, ]\in M_\Sigma : 
\bigl([\, \ul y\,] , [\,\ul x\,]\bigr) \in L_\alpha, \bigl([\,\ul y\,] , [\,\ul z\,]\bigr) \in L_\beta \bigr\}
\; = \; L_\alpha^T \circ L_\beta .
\end{align*}
Indeed, we have 
$[\, \ul y\,] = [(y_1,\tilde x_2,\ldots,\tilde x_{g+n})] = [(y'_1,\tilde z_2,\ldots,\tilde z_{g+n})]$ for 
$\tilde x_i = \pi_\alpha^{-1}(x_i)$, $\tilde z_i = \pi_\beta^{-1}(z_i)$ and
some $y_1\in\alpha$, $y'_1\in\beta$. Since $\alpha,\beta$ are disjoint, this implies $y_1=\tilde z_i$ and $y'_1=\tilde x_j$ for some $i,j\ge 2$ which we can permute to $i=j=2$ to obtain $z_2=\pi_\beta(y_1)\in\alpha'$ and $x_2=\pi_\alpha(y'_1)\in\beta'$. Permutation also achieves $\tilde x_i=\tilde z_i$ for $i\ge 3$ and hence $x_i=\pi_\alpha(y_i)$, $z_i=\pi_\beta(y_i)$ for some $y_i\in\Sigma\less (\alpha\cup\beta)$, which can be rewritten as $\phi''(\pi_{\beta'}(x_i))=\pi_{\alpha'}(z_i)$ by the defining property of $\phi''$ applied to $y_i$.

While transversality cannot be discussed at the level of sets, note that the intermediate points $[\,\ul y\,]$ resp.\ $[\,\ul v\,]$ in the four geometric compositions above are uniquely determined by $\bigl([\, \ul x\,] , [\,\ul z\,]\bigr)$. This proves injectivity of the projection in the geometric composition, and the same infinitesimal fact in the presence of a smooth structure also proves transversality of the intersection, thus embeddedness of the geometric compositions.
For the true Lagrangian submanifolds, the corresponding identities -- up to Hamiltonian isotopy -- are conjectured in \cite[3.6.1]{perutz1}.
 
\item
For attaching circles $\alpha,\beta\subset \Sigma$ with unique transverse intersection point  we have
\begin{align*}
L_\alpha^T \circ L_\beta
&=\bigl\{ \bigl([\, \ul x\, ] , [\, \ul z\, ]\bigr) \,\big|\, \exists [\, \ul y\, ]\in M_\Sigma : 
\bigl([\, \ul y\,] , [\,\ul x\,]\bigr) \in L_\alpha, \bigl([\,\ul y\,] , [\,\ul z\,]\bigr) \in L_\beta \bigr\} \\
&=\bigl\{
 \bigl( \bigl[(x_i)_{i=2,\ldots g+n}\bigr] , \bigl[(z_i)_{i=2,\ldots g+n} \big] \bigr)\,\big|\, \text{\eqref{i2}},  \phi(x_i)=z_i \;\forall i\ge 3 \bigr\} 
 \;\simeq\;{\rm gr}(L_\phi),
\end{align*}
where the intermediate point
$[\, \ul y\,] = [(y_1,\tilde x_2,\ldots,\tilde x_{g+n})] = [(y'_1,\tilde z_2,\ldots,\tilde z_{g+n})]$ after permutation satisfies either $\tilde z_2=y_1\in\alpha$, $\tilde x_2=y'_1\in\beta$ or $y_1=y'_1\in \alpha\pitchfork\beta$, $\pi_\alpha^{-1}(x_2)=\pi_\beta^{-1}(z_2)$. In both cases $\pi_\alpha^{-1}(x_i)=\pi_\beta^{-1}(z_i)$  for $i\ge 3$ can be rewritten as $\phi(x_i)=z_i$ by $\phi\circ\pi_\alpha=\pi_\beta$. For $i=2$ we have
\begin{equation}
\label{i2}
x_2 \in \pi_\alpha(\beta), z_2 \in \pi_\beta(\alpha) \qquad\text{or}\qquad \pi_\alpha^{-1}(x_2)=\pi_\beta^{-1}(z_2) \in\Sigma\less(\alpha\cup\beta) .
\end{equation}
In view of the additional property $\phi(\pi_\alpha(\beta))=\pi_\beta(\alpha)$ of the diffeomorphism $\phi:\Sigma_\alpha\to\Sigma_\beta$ the expectation is that \eqref{i2} is equivalent (up to Hamiltonian isotopy of the Lagrangian) to $\phi(x_2)=z_2$.

Note moreover that the intermediate point $[\,\ul y\,]$ is uniquely determined by $\bigl([\, \ul x\,] , [\,\ul z\,]\bigr)$, which as before would proves embeddedness of the geometric composition
$L_\alpha^T \circ L_{\beta}$ if the same fact holds after adjustment to achieve a smooth structure.

In the true Lagrangian setting of \cite{perutz1}, this move has not been addressed yet.
\end{enumerate} 
\end{example}

\begin{remark}[Monoidal structures and gauge theory for disconnected surfaces] \rm \label{rmk:conn}
Note that the functor arising from infinite dimensional gauge theory in Example~\ref{ex:infdim} can equally be applied to disconnected surfaces and cobordisms and intertwines the disjoint union $\sqcup$ on $\Bor_{2+1}$ with a natural monoidal structure on the symplectic category -- the Cartesian product:
$$
\cA(\Sigma \sqcup \Sigma') = \cA(\Sigma) \times \cA(\Sigma') , \qquad
\cL(Y \sqcup Y') = \cL(Y) \times \cL(Y') .
$$
The same can be said for the representation spaces in Example~\ref{ex:rep}, but it no longer holds in the gauge theoretic settings in which we actually obtain smooth, finite dimensional symplectic manifolds and Lagrangians.
While the symmetric product of a disconnected surface at least is given by a union of Cartesian products, e.g.\
$$
{\rm Sym}^2(\Sigma \sqcup \Sigma') = {\rm Sym}^2(\Sigma) \sqcup {\rm Sym}^1(\Sigma) \times {\rm Sym}^1(\Sigma') \sqcup {\rm Sym}^2(\Sigma') ,
$$
the representation spaces of Remark~\ref{rmk:rigrep} become singular on disconnected surfaces.
Indeed, a puncture $\pt\in \Sigma \sqcup \Sigma'$ lies on only one of the connected components, w.l.o.g.\ $\pt\in\Sigma$, so that the holonomy of a flat connection yields an element of
$$
\Hom\bigl(\pi_1( (\Sigma \sqcup \Sigma') \less \pt ) , SU(2)\bigr)= \Hom\bigl(\pi_1( \Sigma \less \pt ), SU(2) \bigr) \times \Hom\bigl(\pi_1( \Sigma' ), SU(2) \bigr) ,
$$
and thus the moduli space of flat connections is
$\widehat M_{\Sigma \sqcup \Sigma'} = \widehat M_{\Sigma} \times M_{\Sigma'}$, 
where the second factor is the singular representation space from Example~\ref{ex:rep}.

Moreover, adding a requirement of compatibility with monoidal structures to our notion of Floer field theory, such as $\Sigma\sqcup\Sigma' \mapsto M_{\Sigma\sqcup \Sigma'} = M_\Sigma\times M_{\Sigma'}$ for a functor $\Bor_{2+1}\to\Symp$, only makes sense if we also have compatibility such as $M\times M' \mapsto \cC_{M\times M'} = \cC_{M} \otimes \cC_{M'}$ for the functor $\Symp\to\Cat$, i.e.\ a natural factorization of the category $\cC_{M\times M'}$ that is associated to a Cartesian product of symplectic manifolds. 
However, our construction of the symplectic 2-category and the induced functor in \S\ref{ss:symp2} is such that the objects of $\cC_{M\times M'}$ are general Lagrangian submanifolds of $M\times M'$, not just split Lagrangians $L\times L' \subset M\times M'$ arising from objects $L\in\Obj_{\cC_M}$ and $L'\in\Obj_{\cC_{M'}}$ in the categories associated to the factors of the Cartesian product.

Homological algebra allows one to formulate a sense in which refined versions of these categories may be equivalent,
$\widetilde\cC_{M\times M'} \sim \widetilde\cC_{M} \otimes \widetilde\cC_{M'}$, but it would likely require significant restrictions on the geometry of the symplectic manifolds $M,M'$.
\end{remark}

\subsection{Atiyah-Floer type conjectures for 3-manifold invariants} \label{ss:af}

This section discusses the invariants of 3-manifolds in the sense of \S\ref{sec:invariant} which arise abstractly from 2+1 connected Floer field theories as in Definition~\ref{def:fft}, and in the more specific examples surveyed in \S\ref{ss:ex}.
The notion of field theories originated with the idea of obtaining invariants for manifolds by decomposing them into simpler pieces. 
This also motivated the Atiyah-Floer conjecture in the context of Example~\ref{ex:rep} and Heegaard-Floer homology, in which a (conjectural) invariant $|I|:|Man_3|\to |{\rm Gr}|$ takes values in isomorphism classes of groups and is constructed roughly as follows (c.f.\ the outline before Example~\ref{ex:sym}).
\begin{enumerate}
\item
Choose a representative of $[Y]$ and a Heegaard splitting $Y=H^-_0\cup_\Sigma H_1$ along a surface $\Sigma$ into two handlebodies $H_i$ with $\partial H_i=\Sigma^-$. 
This is a special case of a decomposition of the morphism $[Y]\in\Mor_{2+1}(\emptyset,\emptyset)$ given by $[Y]=[H_0^-]\circ [H_1]$.
\item
Represent the dividing surface $\Sigma$ by a symplectic manifold $M_\Sigma$ and the two handlebodies by Lagrangians $L_{H_i}\subset M_\Sigma$, e.g.\ as follows in the Examples:

\smallskip
\noindent
Ex.\ref{ex:rep}:
Using the map $\pi_1(\Sigma)\to\pi_1(H)$ induced by inclusion $\Sigma\hookrightarrow H$ we set

\smallskip
$L_H:= \bigl\{\rho \,\big|\, \rho(\gamma)=\id \;\forall [\gamma]=0\in \pi_1(H) \bigr\} 
\;\subset\;
M_\Sigma :=  \Hom(\pi_1(\Sigma), G ) \big/\!\sim .$

\smallskip
\noindent
Ex.\ref{ex:sym}:
For $g$ the genus of $\Sigma$ and disjoint generators $\alpha_1,\ldots,\alpha_g\subset\Sigma\subset H$ of $\pi_1(H)$ (an additional choice that the invariant may depend on) we set 

\smallskip
$L_H:= T_{\underline \alpha}:= \bigl\{ \bigl[(z_1,\ldots,z_{g})\bigr]  \,\big|\, z_i\in \alpha_i \bigr\}
\;\subset\; 
M_\Sigma:={\rm Sym}^g\Sigma.$ 
\smallskip 

\item 
Take
 $|I|([Y])$ to be the isomorphism class of the Floer homology $HF(L_{H_0},L_{H_1})$.
\item
Check that different choices of representatives and Heegaard splittings yield isomorphic Floer homology groups.
\end{enumerate}
The last step of this program is a major challenge. In the context of Example~\ref{ex:rep}, this step would follow from the Atiyah-Floer conjecture below, since instanton Floer homology arises from the ASD Yang-Mills equation on $\R\times Y$, thus does not depend on the choice of a Heegaard splitting, and in fact yields a 3-manifold invariant, i.e.\ is independent -- up to isomorphism -- from other choices involved in the construction.
In the context of Example~\ref{ex:sym}, the invariance in Step~4 was proven as part of the construction \cite{OS1}, but also follows from the analogue of the Atiyah-Floer conjecture established in \cite{KLT}, which identifies the three flavours of Heegaard Floer homology with three flavours of monopole Floer homology. The latter arise from the Seiberg-Witten equation on $\R\times Y$ and were proven to be a 3-manifold invariant in~\cite{KM}.

\begin{conjecture}[Atiyah-Floer type conjectures for Heegaard splittings] \label{con:af}
For any Heegaard splitting $Y=H^-_0\cup_\Sigma H_1$ of a closed 3-manifold $Y$
there are isomorphisms

\smallskip
\noindent
Ex.\ref{ex:rep}:
$HF(L_{H_0},L_{H_1})\simeq HF_{\rm inst}(Y)$, if $Y$ is a homology 3-sphere\footnote{A closed 3-manifold is called (integral) homology 3-sphere if its homology groups with $\Z$-coefficients $H_*(Y;\Z)\simeq H_*(S^3;\Z)$ coincide with those of the 3-sphere. This assumption guarantees the absence of nontrivial reducible connections on $Y$.},

\smallskip
\noindent
Ex.\ref{ex:sym}: 
$HF^{\cdots}(L_{H_0},L_{H_1})\simeq HF^{\cdots}_{\rm mon}(Y)$ for the three versions $HF^+, HF^-,\widehat{HF}$.

\end{conjecture}

In the context of Example~\ref{ex:rep}, a well defined part of this conjecture (equality of Euler characteristics) was proven by Taubes \cite{taubes:casson}. Defining the full Lagrangian Floer homology would require a notion of pseudoholomorphic curves in the singular representation space $M_\Sigma$, which has not yet been approached. Aside from this, a proof approach was outlined in \cite{Sa:AF,W:survey}, extending the proof in \cite{DS} of the well posed Atiyah-Floer type conjecture described in Remark~\ref{rmk:SO3}.
Another well defined version of the original Atiyah-Floer conjecture for trivial $SU(2)$-bundles was formulated by Salamon \cite{Sa:AF} in the context of the infinite dimensional Floer field theory outlined in Example~\ref{ex:infdim}. For Heegaard splittings of homology 3-spheres $Y=H^-_0\cup_\Sigma H_1$ it asserts the existence of an isomorphism that involves an instanton Floer homology for the pair of infinite dimensional Lagrangians $\cL_{H_0},\cL_{H_1}$ and was recently proven in \cite{SW:openclosed} with field theoretic methods. Roughly speaking, the existence of a 2-category which comprises both handlebodies $H$ and their associated Lagrangians $\cL_H$ as 1-morphisms allows us to express the notion of a ``local'' isomorphism $H\sim \cL_H$, which -- once proven -- implies more ``global'' isomorphisms (see Remark~\ref{rmk:localtoglobal}) such as
\begin{equation}\label{eq:AFinfdim}
HF_{\rm inst}([0,1]\times\Sigma, \cL_{H_0}\times\cL_{H_1})\simeq HF_{\rm inst}(Y) .
\end{equation}
In particular, this proves that the above Steps 1--3 applied to Example~\ref{ex:infdim} yield a well defined invariant for homology 3-spheres, i.e.\ the left hand side of \eqref{eq:AFinfdim} is independent of the choice of Heegaard splitting, as required in Step 4.

A more conceptual reason\footnote{
This reasoning is based on noting that Heegaard splittings $Y=H^-_0\cup_\Sigma H_1$ arise from special Cerf decompositions $Y=Y^-_{\alpha_1}\cup_{\Sigma_1} \ldots Y^-_{\alpha_n} \cup_\Sigma Y_{\beta_n} \ldots \cup_{\Sigma'_1}Y_{\beta_1}$ in which all handles of the same index are grouped together. Composing the handles of equal index yields the corresponding handlebodies $H_1=Y_{\beta_n} \ldots \cup_{\Sigma'_1}Y_{\beta_1}$ and
$H_0=\bigl(Y^-_{\alpha_1}\cup_{\Sigma_1} \ldots Y^-_{\alpha_n}\bigr)^-=Y_{\alpha_n} \ldots \cup_{\Sigma_1}Y_{\alpha_1}$, 
and the moves between Heegaard splittings can be expressed in terms of Cerf moves.
}
for the invariance in Step~4 would be given by an extension of the
constructions in Steps~1--3 to a 3-manifold invariant resulting from a (connected) 2+1 Floer field theory as outlined below, together with an extension of the symplectic category to a 2-category as in Example~\ref{ex:2symp}.

\begin{enumerate}
\item
The Heegaard splitting $Y=H^-_0\cup_\Sigma H_1$ is a decomposition of the morphism $[Y]\in\Mor_{\Bor^{\rm conn}_{2+1}}(\emptyset,\emptyset)$ given by $[Y]=[H_0^-]\circ [H_1]$.
\item
The representation by symplectic data can be viewed as determining parts of a functor $\cF:\Bor^{\rm conn}_{2+1}\to\Symp$ by associating to the empty set $\emptyset\in\Obj_{\Bor^{\rm conn}_{2+1}}$ the trivial symplectic manifold given by a point $\cF(\emptyset):=\pt\in\Obj_{\Symp}$, to nonempty surfaces $\Sigma\in\Obj_{\Bor^{\rm conn}_{2+1}}$ the given symplectic manifolds $\cF(\Sigma):=M_\Sigma$, and to handlebodies $H\in \Mor_{\Bor^{\rm conn}_{2+1}}(\emptyset,\Sigma)$ the Lagrangian $\cF(H_i):=L_{H_i}\subset\pt^- \times M_\Sigma$.
\item
The Floer homology $HF(L_{H_0},L_{H_1})=\Mor^2_{\Symp}(L_{H_0},L_{H_1})$ is the 2-morphism space for $L_{H_0},L_{H_1}\in\Mor^1_{\Symp}(\pt,M_\Sigma)$ in the symplectic 2-category.
\item
Check that the construction in 2. extends to a functor $\cF:\Bor^{\rm conn}_{2+1}\to\Symp$.
\end{enumerate}

Here the functoriality in Step~4 guarantees in particular that different Heegaard decompositions $\widetilde H^-_0\cup_{\widetilde\Sigma} \widetilde H_1 = Y=H^-_0\cup_\Sigma H_1$ of the same 3-manifold, i.e.\ different factorizations
$[\widetilde H_0^-]\circ [\widetilde H_1]=[Y]=[H_0^-]\circ [H_1]\in \Mor_{\Bor^{\rm conn}_{2+1}}(\emptyset,\emptyset)$ 
are mapped to equivalent composable chains of Lagrangians
$$
[{L}_{\widetilde H_0^-}] \circ [{L}_{\widetilde H_1}] = [\underline{L}_{Y}] =
[{L}_{H_0^-}] \circ [{L}_{H_1}]  \; \in\;  \Mor_{\Symp}(\pt,\pt) .
$$
Within the symplectic 2-category, this corresponds to isomorphic 1-morphisms
$$
{L}_{\widetilde H_0^-} \circ {L}_{\widetilde H_1} \; \sim \;  \underline{L}_{Y} \; \sim \; {L}_{H_0^-} \circ{L}_{H_1}
\qquad\text{in}\;\Mor^1_{\Symp}(\pt,\pt).
$$
Now the symplectic 2-morphism spaces extend to tuples using quilted Floer homology as explained in Remark~\ref{rmk:cyclic} and \S\ref{ss:symp2}.
These cyclic morphism spaces have a cyclic symmetry that in particular induces identifications 
$\Mor^2_{\Symp}(L_{H_0},L_{H_1}) = \Mor^2_{\Symp}(\id_{\pt} , L_{H_0^-}\circ L_{H_1})$
where $\id_{\pt}\in\Mor^1_{\Symp}(\pt,\pt)$ is the identity element given by the diagonal. With that, the isomorphism between 1-morphisms induces an isomorphism between the Floer homologies viewed as 2-morphism spaces, 
\begin{align}\label{eq:strategy}
HF(L_{H_0},L_{H_1}) & = \Mor^2_{\Symp}(\id_{\pt} , L_{H_0^-}\circ L_{H_1})  \\  \notag
&\simeq
\Mor^2_{\Symp}(\id_{\pt} , L_{\widetilde H_0^-}\circ L_{\widetilde H_1}) 
= HF(L_{\widetilde H_0},L_{\widetilde H_1}).
\end{align}
While this is a more conceptual explanation of the invariance of Heegaard Floer homology than the direct proof by Ozsvath-Szabo in \cite{OS1}, it is yet to be completed in the setting of Example~\ref{ex:sym}. In the above language, \cite[Lemma 3.17]{perutz1} shows that the geometric composition
$L_{Y_{\alpha_n}} \circ \ldots \circ L_{Y_{\alpha_1}}$ is smoothly isotopic to the Heegaard torus $L_{H_0}=T_{\underline \alpha}$ used in \cite{OS1}, and \cite{lekili} announces this to be a Hamiltonian isotopy, but a result along these lines is so far only proven for handle slides (changing the order between handle attachments and diffeomorphisms) in \cite{perutz3}.

These 2-categorical considerations do however lead to natural extensions of the Atiyah-Floer type conjectures to Cerf decompositions.

\begin{conjecture}[Atiyah-Floer type conjectures for (cyclic) Cerf decompositions] \label{con:afc} \hbox{}
For any Cerf decomposition $[Y]=[Y_{01}]\circ\ldots [Y_{(k-1)k}]$ of a closed 3-manifold $Y$
into simple cobordisms $Y_{i(i+1)}$ there are isomorphisms

\smallskip
\noindent
Ex.\ref{ex:rep}: 
$HF(L_{Y_{01}},\ldots, L_{Y_{(k-1)k}})\simeq HF_{\rm inst}(Y)$,

\smallskip
\noindent
Ex.\ref{ex:sym}:
$HF^{\cdots}(L_{Y_{01}},\ldots, L_{Y_{(k-1)k}})\simeq HF^{\cdots}_{\rm mon}(Y)$.
\end{conjecture}

Another version of this conjecture is an expected isomorphism between link invariants that arise from a Floer field theory for tangle categories in \cite{ww:fielda} (in which invariance under isotopy of the link embedding remains to be proven) and Floer homology invariants defined from singular instantons in \cite{km:kh}.  

Here the Cerf decomposition $Y=Y_{01}\cup_{\Sigma_1}\ldots \cup_{\Sigma_{k-1}} Y_{(k-1)k}$ arises from a Morse function $f:Y\to\R$ and choices of regular level sets $\Sigma_j=f^{-1}(b_j)$ which separate the critical points of $f$. Analogous conjectures are obtained by working with cyclic Cerf decompositions
$Y=\bigl(Y_{01}\cup_{\Sigma_1}\ldots \cup_{\Sigma_{k-1}} Y_{(k-1)k}\bigr)/{\Sigma_0\simeq \Sigma_k}$ which arise from $S^1$-valued Morse functions $f:Y\to S^1$ and regular level sets $\Sigma_j=f^{-1}(b_j)$ with a cyclic identification $\partial Y_{01} \supset\Sigma_0=f^{-1}(b_0=b_k)=\Sigma_k \subset \partial Y_{(k-1)k}$.
In this setting we can work as in Example~\ref{ex:sym} with symmetric products ${\rm Sym}^{g+n}(\Sigma)$, where $g$ is the genus of the surface $\Sigma$ and $n\in\Z$ is any fixed integer, by restricting consideration to Morse functions whose regular fibers have genus $\ge -n$ as laid out in \cite{perutz1,usher,lekili}.
Similarly, cyclic Cerf decompositions in the context of Example~\ref{ex:rep} allow us to work with nontrivial bundles as in Remark~\ref{rmk:rigrep}, and thus obtain well defined Atiyah-Floer type conjectures, most notably for nontrivial $SO(3)$ bundles as discussed below.

\begin{example}[Cyclic Atiyah-Floer type conjecture for nontrivial $SO(3)$ bundles]   \rm\label{ex:SO3}
The Floer field theory outlined in Example~\ref{ex:rep} is made rigorous in \cite{ww:fielda} by replacing trivial $SU(2)$-bundles with nontrivial $SO(3)$-bundles. However, this excludes the empty set (over which any bundle is trivial) from the objects and thus does not allow for handlebodies (cobordisms from the empty set to a surface) as morphisms. So in this context we cannot rigorously formulate an Atiyah-Floer type Conjecture~\ref{con:af} for Heegaard splittings, but Conjecture~\ref{con:afc} for cyclic Cerf decompositions does have a well defined meaning:
An identification between the instanton Floer homology of a 3-manifold $Y$ with cyclic Cerf decomposition and the Floer homology of a cyclic sequence of Lagrangians $\underline L_Y$ arising from this Cerf decomposition, 
$$
HF\bigl( \underline L_Y = (L_{Y_{(j-1)j}})_{j\in\Z_k} \bigr) \simeq HF_{\rm inst}\bigl(  Y= (Y_{01}\cup_{\Sigma_1}\ldots \cup_{\Sigma_{k-1}} Y_{(k-1)k})/{\Sigma_0\simeq \Sigma_k}  \bigr) .
$$
Here $HF(\underline L_Y)$ is defined in terms of quilted pseudoholomorphic cylinders in \cite{ww:qhf}, but is also directly identical to the standard Floer homology 
$$
HF(\underline L_Y) = HF\bigl( (L_{Y_{01}}\times \ldots L_{Y_{(k-1)k}})^T , \Delta_{M_{\Sigma_0}}\times \ldots \Delta_{M_{\Sigma_{k-1}}} \bigr)
$$ 
for a pair of Lagrangians in $M_{\Sigma_0}\times M_{\Sigma_0}^-\times \ldots M_{\Sigma_{k-1}}\times M_{\Sigma_{k-1}}^-$, where the first requires a permutation of factors
$(\ldots)^T: M_{\Sigma_0}^-\times \ldots M_{\Sigma_{k-1}}^- \times M_{\Sigma_0}  \to M_{\Sigma_0}\times M_{\Sigma_0}^-\times \ldots M_{\Sigma_{k-1}}^-$.
\end{example}

\begin{remark}[Approaches to Atiyah-Floer conjecture for nontrivial $SO(3)$ bundles]  \rm \label{rmk:SO3}
The cyclic version of Conjecture~\ref{con:afc} in Example~\ref{ex:SO3} was proven by 
Dostoglou-Salamon \cite{DS} for 3-manifolds equipped with a Morse function $f:Y\to S^1$ without critical points, so that all fibers $\Sigma_i\simeq\Sigma$ are diffeomorphic and thus $Y = ([0,1]\times \Sigma)/\phi$ is the mapping cylinder of a diffeomorphism $\phi:\Sigma\to\Sigma$ on a regular level set $\Sigma=f^{-1}(b_0)$ arising from the flow of $\nabla f$ on $Y\less f^{-1}(b_0)$. This gives rise to a symplectomorphism $L_\phi:M_\Sigma\to M_\Sigma$, for which the Floer homology $HF(L_\phi)=HF({\rm graph}\, L_\phi, \Delta_{M_\Sigma})$ can be constructed without boundary conditions.
Then the proof of the Atiyah-Floer type isomorphism $HF(L_\phi)\simeq HF_{\rm inst}(([0,1]\times \Sigma)/\phi)$ directly identifies the Floer complexes by an adiabatic limit in which the metric on $\Sigma$ is scaled by $\epsilon^2\to 0$.

In the presence of critical points this argument is expected to generalize by partitioning
$Y= (Y_{01}\cup_{\Sigma_1}\ldots \cup_{\Sigma_{k-1}} Y_{(k-1)k})/{\Sigma_0\simeq \Sigma_k}$
into thickenings of the surfaces $[0,1]\times\Sigma_j$ with metric ${\rm d}s^2 + \epsilon^2 g_{\Sigma_j}$
and handle attachments $Y_{ij}$ with metric $\epsilon^2 g_{Y_{ij}}$.
As a result, the volumes ${\rm Vol}([0,1]\times\Sigma_j)\sim \eps^2$ and ${\rm Vol}(Y_{ij})\sim \eps^3$ scale differently, which indicates different degenerations on these types of pieces. 
Here the absence of reducibles guarantees linear bounds $d\bigl(A, \cA_{\rm flat}\bigr) \leq C \|F_A\|$ of the distance of a connection (on $\Sigma_j$ or $Y_{ij}$) to the flat connections in terms of its curvature, so that the adiabatic limit analysis \cite{DS} can be combined with the analytic setup for boundary conditions in gauge theory \cite{W:ell1, W:ell2} to obtain a compactness result: Solutions of the ASD equation defining the instanton Floer differential converge for $\eps\to 0$ to pseudoholomorphic curves in the $M_{\Sigma_j}$, whose boundaries match up via the $L_{Y_{ij}}$, thus giving rise to a contribution to the (quilted) Lagrangian Floer differential.
This initial motivation for the conjecture was fleshed out with analytic details in \cite{W:talks} and recently explicitly put into quilted settings in \cite{duncan}, but the proof of a 1-1 correspondence between the differentials remains to be completed.

An alternative approach closer to completion is based on a strategy outlined in \cite{Sa:AF, W:survey} via an intermediate instanton Floer homology 
${HF_{\rm inst}\bigl(\bigsqcup_j [0,1]\times\Sigma_j , (\cL_{Y_{ij}}) \bigr)}$
associated to the manifold with boundary $\bigsqcup_j [0,1]\times\Sigma_j$ and boundary conditions given by infinite dimensional Lagrangians $\cL_{Y_{ij}}$ as in Example~\ref{ex:infdim}. 
Its identification with $HF\bigl( \underline L_Y = (L_{Y_{ij}}) \bigr)$ should follow from a direct generalization of the adiabatic limit analysis \cite{DS} to Lagrangian boundary conditions.
On the other hand, an isomorphism as in \eqref{eq:AFinfdim} between the ``closed'' and ``open'' instanton Floer homologies
$HF_{\rm inst}\bigl(\bigsqcup_j [0,1]\times\Sigma_j , (\cL_{Y_{ij}}) \bigr)\simeq HF_{\rm inst}(Y)$
could also be approached by degenerating only the metric on the handle attachments $Y_{ij}$.
Instead of using a metric degeneration, this can also be approached by a ``local to global'' field theoretic argument as explained above and in Remark~\ref{rmk:localtoglobal}, where the strategy is to prove an isomorphism $Y_{ij} \sim \cL_{Y_{ij}}$ in a 2-categorical setting, which then implies the desired isomorphism of Floer homologies similar to the argument for equation \eqref{eq:strategy} above.
In the case of trivial SU(2)-bundles, this approach is implemented in \cite{SW:openclosed} and directly transfers to any setting in which the fundamental classes $[L_{Y_{ij}}]$ of the associated finite dimensional Lagrangians induce well defined classes in the ``2-morphism spaces'' resp.\ ``localized Floer theories'' $HF_{\rm inst}(Y_{ij}, \cL_{Y_{ij}})$. 

Finally, a third approach pioneered by Fukaya \cite{fukaya} is to avoid adiabatic limit analysis and construct a direct chain map between the instanton and Lagrangian Floer chain complexes. 
This should again be understood as a ``local to global'' field theoretic argument, based on an implicit isomorphism $Y_{ij} \sim L_{Y_{ij}}$. 
In this case, an appropriate 2-categorical setting needs to combine the ASD and Cauchy-Riemann equation.
Aside from metric degeneration proposals in \cite{fukaya}, this can be achieved by Lagrangian seam conditions as discussed in Example~\ref{ex:superenhanced} and \cite{lipyanski}.
\end{remark}

\section{Extensions of Floer field theories} \label{s:ext}

\subsection{2-categories and bicategories} \label{ss:2cat}

In the construction of both the bordism and symplectic categories we introduced an equivalence relation between the morphisms in order to obtain a geometrically meaningful notion of composition. An algebraically cleaner way of phrasing the requirements on this relation -- in particular compatibility with the desired notion of composition -- is in terms of 2-morphisms between the morphisms, forming either a 2-category or the slightly weaker notion of bicategory. 
Once these notions and some examples are established, we will cast the construction of bordism and symplectic categories in these 2-categorical terms.

\begin{definition}\label{def:2cat}
A {\bf 2-category} $\cC$ is a category enriched in categories, i.e.\ consists of
\begin{itemize}
\item
a set $\Obj_\cC$ of {\bf objects},
\item
for each pair $x_1,x_2\in \Obj_\cC$ a category of {\bf morphisms} $\Mor_\cC(x_1,x_2)$, i.e.\ 
\begin{itemize}
\item
a set $\Mor_\cC^1(x_1,x_2)$ of {\bf 1-morphisms},
\item
for each pair $f,g\in \Mor_\cC^1(x_1,x_2)$ a set of {\bf 2-morphisms} $\Mor_\cC^2(f,g)$,
\item
for each triple $f,g,h\in\Mor_\cC^1(x_1,x_2)$ an associative {\bf vertical composition} 
$$
\Mor_\cC^2(f,g)\times \Mor_\cC^2(g,h)\to \Mor^2_\cC(g,h), \quad
(\alpha_{12},\beta_{12}) \mapsto \alpha_{12}\circ_{\rm v} \beta_{12},
$$
\item 
for each $f\in \Mor_\cC^1(x_1,x_2)$ an identity $\id_f\in\Mor_\cC^2(f,f)$ for $\circ_{\rm v}$,
\end{itemize}
\item
a {\bf composition functor} $\Mor_\cC(x_1,x_2)\times \Mor_\cC(x_2,x_3)\to \Mor_\cC(x_1,x_3)$ 
for each triple $x_1,x_2,x_3\in \Obj_\cC$, i.e.\ 
\begin{itemize}
\item
an associative {\bf horizontal composition on 1-morphisms} 
$$
\Mor_\cC^1(x_1,x_2)\times  \Mor_\cC^1(x_2,x_3)\to  \Mor_\cC^1(x_1,x_3), 
\quad
(f_{12},f_{23}) \mapsto f_{12}\circ_{\rm h} f_{23} ,
$$
\item
 for each $x\in \Obj_\cC$ an identity $1_x\in\Mor_\cC^1(x,x)$ for $\circ_{\rm h}$, that is 
 $1_x \circ_{\rm h} f= f$ for any $f\in\Mor_\cC^1(x,y)$ and  $g \circ_{\rm h} 1_x = g$ for any $g\in\Mor_\cC^1(w,x)$.
\item
for any $(f_{12},f_{23}), (g_{12},g_{23})\in \Mor_\cC^1(x_1,x_2)\times  \Mor_\cC^1(x_2,x_3)$
an associative {\bf horizontal composition on 2-morphisms},
\begin{align*}
\Mor_\cC^2(f_{12},g_{12})\times  \Mor_\cC^2(f_{23},g_{23}) &\;\longrightarrow\;  \Mor_\cC^2(f_{12}\circ_{\rm h} f_{23} ,g_{12}\circ_{\rm h} g_{23} ) \\
(\alpha_{12},\alpha_{23}) &\;\longmapsto\; \alpha_{12}  \circ_{\rm h} \alpha_{23} ,
\end{align*}
that is compatible with identities, i.e.\ for $f_{12}= g_{12} \in \Mor_\cC^1(x_1,x_2)$ and $f_{23}= g_{23} \in \Mor_\cC^1(x_2,x_3)$ we have
$$
\id_{f_{12}}  \circ_{\rm h} \id_{f_{23}} = \id_{f_{12}\circ_{\rm h}f_{23}} ,
$$
and is compatible with vertical composition, i.e.\ for $\alpha_{12}, \beta_{12}\in  \Mor_\cC^2(f_{12},g_{12})$ and 
$\alpha_{23} , \beta_{23} \in \Mor_\cC^2(f_{23},g_{23})$ we have
$$
( \alpha_{12} \circ_{\rm v} \beta_{12}   ) \circ_{\rm h} ( \alpha_{23}  \circ_{\rm v} \beta_{23}  )
=
( \alpha_{12} \circ_{\rm h} \alpha_{23}   ) \circ_{\rm v} ( \beta_{12}  \circ_{\rm h} \beta_{23}  ).
$$
\end{itemize}
\end{itemize}
\end{definition}

A graphical representation of the structure and axioms of 2-categories is by string diagrams, as discussed in \S\ref{ss:quilt}. In \cite{W:slides} these were motivated as a natural visualization of 4-dimensional manifolds with boundary and corners, as they appear in the extension of 2+1 bordism categories.
However, we will see in \S\ref{ss:bord} that instead of a 2-category this yields the following notion of bicategory, in which the horizontal unital and associativity requirements on 1-morphisms are relaxed.

\begin{definition} \label{def:bicategory}
A {\bf bicategory} $\cC$ consists of a set $\Obj_\cC$ of objects and categories of morphisms $\Mor_\cC(x_1,x_2)$ as in Definition~\ref{def:2cat} (i.e.\ 1-morphisms, 2-morphisms, vertical composition $\circ_{\rm v}$, and units $\id_f$), and for each triple $x_1,x_2,x_3\in \Obj_\cC$ a {\bf horizontal bifunctor} 
$\circ_{\rm h} : \bigl( \Mor_\cC(x_1,x_2) ,  \Mor_\cC(x_2,x_3) \bigr) \to \Mor_\cC(x_1,x_3)$ consisting of
\begin{itemize}
\item
a {\bf horizontal map on 1-morphisms} 
$$
\circ_{\rm h}^1 :
\Mor_\cC^1(x_1,x_2)\times  \Mor_\cC^1(x_2,x_3)\to  \Mor_\cC^1(x_1,x_3),
$$
\item
for any composable pairs $(f_{12},f_{23}), (g_{12},g_{23})$ a {\bf horizontal map on 2-morphisms}
$$
\circ_{\rm h}^2  : \Mor_\cC^2(f_{12},g_{12})\times  \Mor_\cC^2(f_{23},g_{23}) \to \Mor_\cC^2\bigl( \circ_{\rm h} (f_{12}, f_{23}) , \circ_{\rm h}(g_{12}, g_{23}) \bigr),
$$
\end{itemize}
that make $\circ_{\rm h}$ into a bifunctor and {\bf horizontal composition up to 2-isomorphism}, as follows:
\begin{itemize}
\item
$\circ_{\rm h}=(\circ_{\rm h}^1,\circ_{\rm h}^2)$ is compatible with vertical identities, i.e.\ 
$\circ_{\rm h}^2( \id_{f_{12}}  ,\id_{f_{23}} ) = \id_{\circ_{\rm h}^1(f_{12},f_{23})}$,
\item
$\circ_{\rm h}^2$ is associative and compatible with vertical composition, i.e.\ 
\begin{align*}
 \circ_{\rm h}^2 (  \circ_{\rm h}^2( \alpha_{12} , \alpha_{23} ) ,   \alpha_{34}  )
 &=
  \circ_{\rm h}^2 (   \alpha_{12} , \circ_{\rm h}^2( \alpha_{23}  ,   \alpha_{34} )  ) , 
\\
 \circ_{\rm h}^2 (  \alpha_{12} \circ_{\rm v} \beta_{12}  ,   \alpha_{23} \circ_{\rm v}\beta_{23}  )
&=
\circ_{\rm h}^2 ( \alpha_{12},\alpha_{23}   ) \; \circ_{\rm v} \; \circ_{\rm h}^2 ( \beta_{12}  ,\beta_{23}  ) ,
\end{align*}
\item
$\circ_{\rm h}^1$ is associative up to 2-isomorphism, i.e.\ 
$\circ_{\rm h}^1 \bigl( f, \circ_{\rm h}^1 ( g, h )  \bigr) \sim \circ_{\rm h}^1 \bigl( \circ_{\rm h}^1 ( f, g ) , h  \bigr)$,
where the relation $\sim$ on $\Mor_\cC^1$ is defined by
$$
k\sim k' 
\quad\Longleftrightarrow\quad
\exists \alpha, \beta\in\Mor_\cC^2 : \; 
\id_k=\alpha\circ_{\rm v}\beta, \; \id_{k'}=\beta\circ_{\rm v}\alpha ,
$$
\item
$\circ_{\rm h}^1$ is unital up to 2-isomorphism, i.e.\ for each $y\in \Obj_\cC$ there exists a (not necessarily unique) {\bf weak identity 1-morphism} $1_y\in\Mor_\cC^1(y,y)$ such that
$$
\circ_{\rm h}^1 ( f, 1_y  ) \sim  f \quad \forall f\in\Mor_\cC^1(x,y), 
\qquad
\circ_{\rm h}^1 ( 1_y , g ) \sim  g \quad \forall g\in\Mor_\cC^1(y,z).
$$
\end{itemize}
\end{definition}

An instructive non-example of a bicategory (or 2-category) is the attempt to extend the category of sets and maps by a notion of conjugacy as 2-morphisms.

\begin{example}[Categorical structure of sets, maps, and conjugacy] \label{nonex} \rm
The category of sets consists of sets as objects and maps $f_{12}:S_1\to S_2$ as 1-morphisms from $S_1$ to $S_2$, with horizontal composition $f_{12}\circ_{\rm h} f_{23}$ given by composition of maps.
One might want to add further structure to the sets and ask it to be preserved by the maps -- e.g.\ forming the linear category of vector spaces and homomorphisms -- but the underlying form of many categories is given by sets and maps.
In the linear category, a natural relation between homomorphisms arises from a change of basis, which is formalized as the conjugation with an isomorphism. Generally, for each pair of maps $f_{12},g_{12}\in \Mor_\cC^1(S_1,S_2)$ between the same two sets $S_1,S_2$, one would like to let 2-morphisms be given by conjugation with bijections, 
$$
\Mor_\cC^2(f_{12},g_{12}) := \bigl\{ \alpha_{12}=(\alpha_1,\alpha_2) \,\big|\, \alpha_i:S_i\to S_i \;\mbox{bijections}, \alpha_2^{-1}\circ f_{12}\circ\alpha_1 = g_{12} \bigr\} .
$$
This defines a category $\Mor_\cC(S_1,S_2)$ since conjugations have a well defined (and associative, unital) vertical composition: 
$$
\alpha_2^{-1}\circ f_{12}\circ\alpha_1 = g_{12}, \quad 
\beta_2^{-1}\circ g_{12}\circ\beta_1 = h_{12} 
\quad\Rightarrow\quad
(\alpha_2\circ\beta_2)^{-1}\circ f_{12}\circ (\alpha_1\circ\beta_1) = h_{12} .
$$
That is, setting $(\alpha_1,\alpha_2) \circ_{\rm v} (\beta_1,\beta_2) := (\alpha_1\circ\beta_1, \alpha_2\circ\beta_2)$ composes two conjugacies from $f_{12}$ to $g_{12}$ and from $g_{12}$ to $h_{12}$ to a conjugacy from $f_{12}$ to $h_{12}$.
In other words, conjugacy is an equivalence relation -- its transitivity corresponds to a well defined vertical composition.
Next, a well defined horizontal bifunctor would require
conjugacy to be compatible with composition of maps, that is
$$
\alpha_2^{-1}\circ f_{12}\circ\alpha_1 = g_{12}, \quad 
\tilde\alpha_3^{-1}\circ f_{23}\circ \tilde\alpha_2 = g_{23} 
\quad\Rightarrow\quad
\gamma_3^{-1}\circ ( f_{23}\circ f_{12} ) \circ \gamma_1 = g_{23}\circ g_{12} 
$$
for some bijections $(\gamma_1,\gamma_3):= \circ_{\rm h}\bigl((\alpha_1,\alpha_2), (\tilde\alpha_2,\tilde \alpha_3)\bigr)$. 
However, this implication generally only holds if we have $\alpha_2=\tilde\alpha_2$ in 
$$
\tilde\alpha_3^{-1}\circ f_{23}\circ \tilde\alpha_2 \circ \alpha_2^{-1}\circ f_{12}\circ\alpha_1 = g_{23}\circ g_{12} .
$$
Thus our constructions do not yield a 2-category or bicategory.
This corresponds to the fact that composition of maps does not descend to a well defined composition of conjugacy classes. In those terms, the above discussion shows that $[f_{12}]=[g_{12}]$ and $[f_{23}]=[g_{23}]$ does generally not imply $[f_{12}\circ f_{23}]=[g_{12}\circ g_{23}]$.
In fact, we will see in Remark~\ref{rmk:2cat-quotient} that a bicategorical structure is exactly what is needed to obtain a well defined composition on the level of equivalence classes of 1-morphisms.
\end{example}

A better behaved notion of conjugacy-type equivalence between map-type objects is the following notion of natural transformations between functors, which are an equivalence relation (define a category) and are compatible with composition (fit into a horizontal composition functor) as we show in the following Lemma. 
Here we moreover review the category of functors and a composition functor on it.

\begin{lemma} \label{lem:func}
The {\bf category of functors} $\Fun(\cC,\cD)$ is well defined as follows.
\begin{itemize}
\item
Objects are functors $\cF:\cC\to\cD$.
\item
Morphisms $\eta\in\Mor_{\Fun(\cC,\cD)}(\cF,\cG)$ are natural transformations $\eta: \cF \Rightarrow \cG$ given by a map $\eta: \Obj_\cC \to \Mor_\cD$ which takes each $x\in\Obj_\cC$ to a morphism $\eta(x)\in\Mor_\cD(\cF(x),\cG(x))$ such that we have
$$
k\in\Mor_\cC(x,y) \qquad\Longrightarrow \qquad \cF(k) \circ \eta(y) = \eta(x) \circ \cG(k).
$$
\item 
Composition of natural transformations $\eta: \cF \Rightarrow \cG$ and $\zeta: \cG \Rightarrow \cH$ is given by $(\eta\circ_{\rm v}\zeta): x \mapsto \eta(x) \circ \zeta(x)$ as in Figure~\ref{fig:verticalnat}, giving rise to a map
$$
\circ_{\rm v}:\Mor_{\Fun(\cC,\cD)}(\cF,\cG)\times \Mor_{\Fun(\cC,\cD)}(\cG,\cH)\to
\Mor_{\Fun(\cC,\cD)}(\cF,\cH) .
$$
\item 
The identity natural transformations $\id_\cF: \cF \Rightarrow \cF$ are given by 
$\id_\cF(x)= \id_{\cF(x)}$ for all $x\in\Obj_\cC$.
\end{itemize}

\begin{figure}[!h]
\centering
\includegraphics[width=5.5in]{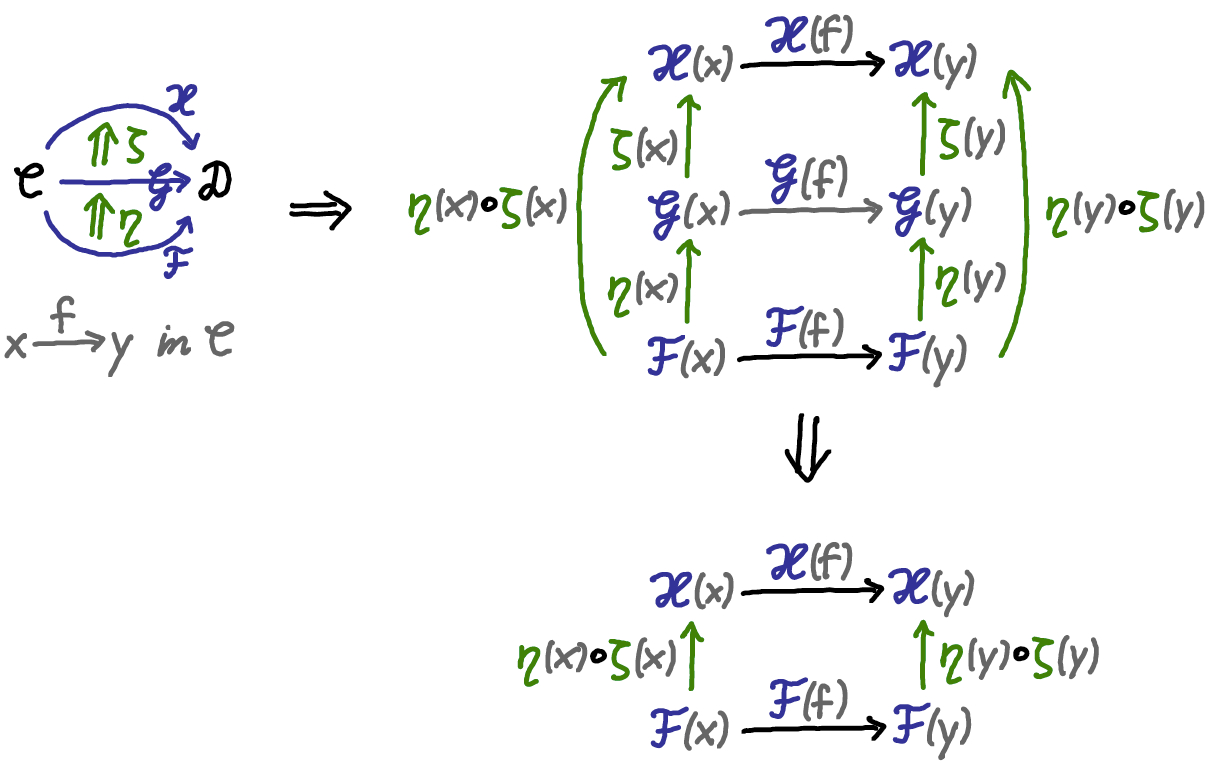}
\caption{Vertical composition of natural transformations.}
\label{fig:verticalnat}
\end{figure}

Moreover, for any triple of categories $\cC_0,\cC_1,\cC_2$, the {\bf horizontal composition functor} $\circ_{\rm h} : \Fun(\cC_0,\cC_1)\times\Fun(\cC_1,\cC_2)\to\Fun(\cC_0,\cC_2)$ is well defined as follows.
\begin{itemize}
\item
Composition of functors $(\cF_{01},\cF_{12}) \mapsto \cF_{01}\circ_{\rm h}\cF_{12}$ is given by composition of the maps $\Obj_{\cC_i}\to\Obj_{\cC_{i+1}}$ and $\Mor_{\cC_i}\to\Mor_{\cC_{i+1}}$ which make up $\cF_{i(i+1)}$ for $i=0,1$.
\item
The identities for this horizontal composition are given by the identity functors $1_\cC\in \Fun(\cC,\cC)$.
\item
For each pair of objects $(\cF_{01},\cF_{12}), (\cG_{01},\cG_{12}) \in \Fun(\cC_0,\cC_1)\times\Fun(\cC_1,\cC_2)$ in the product category, the horizontal composition of natural transformations, as illustrated in Figure~\ref{fig:nathor} is
\begin{align*}
&\circ_{\rm h}: \; \Mor_{\Fun(\cC_0,\cC_1)}(\cF_{01},\cG_{01}) \times \Mor_{\Fun(\cC_1,\cC_2)}(\cF_{12},\cG_{12}) \\
&\qquad\qquad \qquad\qquad\qquad\qquad\qquad\quad\;\to\; \Mor_{\Fun(\cC_0,\cC_2)}(\cF_{01}\circ\cF_{12},\cG_{01}\circ\cG_{12}) \\
&\bigl( \eta_{01}\circ_{\rm h}\eta_{12}\bigr) (x) := \eta_{12}\bigl(\cF_{01}(x)\bigr) \circ \cG_{12}\bigl(\eta_{01}(x)\bigr)
=
\cF_{12}\bigl(\eta_{01}(x)\bigr) \circ \eta_{12}\bigl(\cG_{01}(x)\bigr)  .
\end{align*}
\end{itemize}
\end{lemma}

\begin{figure}[!h]
\centering
\includegraphics[width=5.5in]{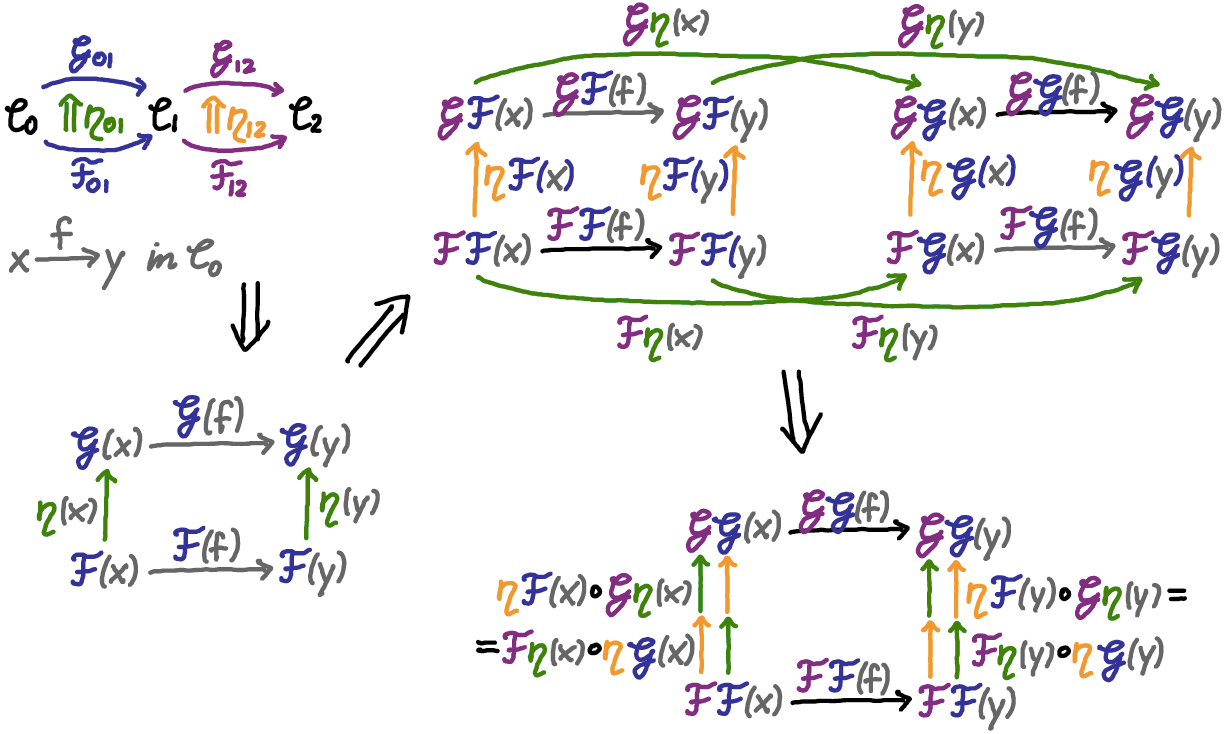}
\caption{Horizontal composition of natural transformations $\eta_{01}$ (in green) and $\eta_{12}$ (in orange).}
\label{fig:nathor}
\end{figure}

\begin{proof}
The (vertical) composition of natural transformations in $\Fun(\cC,\cD)$ is well defined since for all $k\in\Mor_\cC(x,y)$ we have 
\begin{align*}
\cF(k) \circ (\eta\circ\zeta)(y) 
 = \cF(k)  \circ \eta(y) \circ \zeta(y)
& = \eta(x) \circ \cG(k) \circ \zeta(y) \\
&= \eta(x) \circ \zeta(x) \circ \cH(k)
=  (\eta\circ\zeta)(x) \circ \cH(k) .
\end{align*} 
It is associative by associativity of the composition $\circ$ in $\cD$, and it is unital with $1_\cF: x\mapsto \id_{\cF(x)}$.
The (horizontal) composition of functors is well defined in the same way in which composition of maps is well defined.
The horizontal composition of natural transformations
$\eta_{01}:\cF_{01}\Rightarrow\cG_{01}$ and $\eta_{12}:\cF_{12}\Rightarrow\cG_{12}$
is well defined since for all $k\in\Mor_\cC(x,y)$ we have 
\begin{align*}
(\cF_{01}\circ \cF_{12})(k) \circ (\eta_{01}\circ_{\rm h} \eta_{12})(y) 
&= \cF_{12}\bigl(\cF_{01}(k) \bigr) \circ \eta_{12}\bigl(\cF_{01}(y)\bigr)  \circ \cG_{12}\bigl(\eta_{01}(y)\bigr)   \\
& = \eta_{12}\bigl(\cF_{01}(x)\bigr) \circ \cG_{12}\bigl(\cF_{01}(k) \bigr) \circ \cG_{12}\bigl(\eta_{01}(y)\bigr)  \\
& = \eta_{12}\bigl(\cF_{01}(x)\bigr) \circ \cG_{12}\bigl(\cF_{01}(k)  \circ \eta_{01}(y)\bigr)  \\
& = \eta_{12}\bigl(\cF_{01}(x)\bigr) \circ \cG_{12}\bigl(\eta_{01}(x) \circ \cG_{01}(k) \bigr)  \\
&=
\eta_{12}\bigl(\cF_{01}(x)\bigr) \circ \cG_{12}\bigl(\eta_{01}(x)\bigr) \circ \cG_{12}\bigl( \cG_{01}(k) \bigr) \\
&=   (\eta_{01}\circ_{\rm h} \eta_{12})(x) \circ (\cG_{01}\circ \cG_{12})(k) .
\end{align*}
Moreover, the identity 
$\eta_{12}\bigl(\cF_{01}(x)\bigr) \circ \cG_{12}\bigl(\eta_{01}(x)\bigr)= \cF_{12}\bigl(\eta_{01}(x)\bigr) \circ \eta_{12}\bigl(\cG_{01}(x)\bigr) $
follows from applying $\eta_{12}$ to the morphism $\eta_{01}(x): \cF_{01}(x)\to \cG_{01}$ in $\cC_1$. 

This horizontal composition is compatible with identities since for $\cF_{ij}:\cC_i\to\cC_j$ and $x\in\Obj_{\cC_0}$ we have
$$
\bigl( 1_{\cF_{01}}\circ_{\rm h}1_{\cF_{12}} \bigr) (x) = \id_{\cF_{12}(\cF_{01}(x))} \circ \cF_{12}\bigl(\id_{\cF_{01}(x)}\bigr)  = \id_{\cF_{12}(\cF_{01}(x))} = 1_{\cF_{12}\circ_{\rm h}\cF_{01}}(x), 
$$
and it is compatible with vertical composition since for $\eta_{ij}:\cF_{ij}\Rightarrow\cG_{ij}$ and  $\zeta_{ij}:\cG_{ij}\Rightarrow\cH_{ij}$ and $x\in\Obj_{\cC_0}$ we have
\begin{align*}
&\bigl( \eta_{01}\circ_{\rm v}\zeta_{01} \bigr) \circ_{\rm h} \bigl( \eta_{12}\circ_{\rm v}\zeta_{12}     \bigr) (x) 
= ( \eta_{12}\circ_{\rm v}\zeta_{12} )\bigl(\cF_{01}(x)\bigr) \circ \cH_{12}\bigl((\eta_{01}\circ_{\rm v}\zeta_{01})(x)\bigr) \\
&\qquad\qquad\qquad\qquad\quad
= \eta_{12}\bigl(\cF_{01}(x)\bigr) \circ \zeta_{12}\bigl(\cF_{01}(x)\bigr) \circ \cH_{12}\bigl(\eta_{01}(x)\bigr) \circ \cH_{12}\bigl(\zeta_{01}(x)\bigr) \\
&\qquad\qquad\qquad\qquad\quad
= \eta_{12}\bigl(\cF_{01}(x)\bigr) \circ \cG_{12}\bigl(\eta_{01}(x)\bigr)  \circ  \zeta_{12}\bigl(\cG_{01}(x)\bigr)\circ \cH_{12}\bigl(\zeta_{01}(x)\bigr) \\
&\qquad\qquad\qquad\qquad\quad
=
\bigl( \eta_{01}\circ_{\rm h}\eta_{12} \bigr) \circ_{\rm v} \bigl( \zeta_{01}\circ_{\rm h}\zeta_{12}     \bigr) (x) .
\end{align*}
\vskip-7mm
\end{proof}

The well defined category of functors and horizontal composition functor now yield an extension of the category of categories in Example~\ref{ex:1cat} to a 2-category.

\begin{example}\label{ex:2cat} \rm
The {\bf 2-category of categories} ${\rm Cat}$ consists of
\begin{itemize}
\item
objects given by categories $\cC$,
\item
the morphism category for any pair of categories $\cC_1,\cC_2$ given by the category of functors $\Fun(\cC_1,\cC_2)$, 
\item
the horizontal composition functor ${\circ_{\rm h} : \Fun(\cC_0,\cC_1)\times\Fun(\cC_1,\cC_2)\to\Fun(\cC_0,\cC_2)}$
for each triple of categories $\cC_0,\cC_1,\cC_2$.
\end{itemize}
\end{example}

Leading towards the next example of 2-categories, the construction of the symplectic category in Definition~\ref{def:1symp} can be understood as reconstructing a category from its simple morphisms and Cerf moves. For that purpose it is useful to express this Cerf data as the following ``resolution'' of the original category. 

\begin{example}[Resolution of a category with Cerf decompositions]  \label{ex:ext0} \rm
Given a category $\cC$ with Cerf decompositions into simple morphisms unique up to Cerf moves as in Definition~\ref{def:Cerf}, the 2-category $\cC^\#$ is defined by
\begin{itemize}
\item 
the set of objects $\Obj_{\cC^\#}:=\Obj_\cC$,
\item
the set of 1-morphisms $\Mor^1_{\cC^\#}$ given by finite composable chains of simple morphisms, with horizontal composition given by concatenation of chains,
\item
the set of 2-morphisms $\Mor^2_{\cC^\#}\subset \Mor^1_{\cC^\#} \times \Mor^1_{\cC^\#}$ given by
pairs of 1-morphisms that are related via a sequence of Cerf moves.
\end{itemize}
For this to define a 2-category, one should allow for empty chains as identity 1-morphisms. 
Vertical composition of 2-morphisms is well defined, associative and unital since relation via Cerf moves is an equivalence relation. Horizontal composition of 2-morphisms $(\uL_{12},\uL_{12}')\circ^2_{\rm h} (\uL_{23},\uL_{23}') := (\uL_{12}\#\uL_{23}, \uL_{12}'\#\uL_{23}')$ is compatible with vertical composition because the equivalence via Cerf moves is designed to be compatible with concatenation.
Now -- although the composition in $\cC$ is encoded in $\cC^\#$ only via the Cerf moves -- the original category $\cC$ can be reconstructed as the quotient $\cC\cong \cC^\#/\!\!\sim\; ={|\cC^\#|}$ defined in Remark~\ref{rmk:2cat-quotient} below.

This is exactly how the symplectic category was constructed in Definition~\ref{def:1symp}.
Indeed, for $\cC=\Symp$ the above construction reproduces Definition~\ref{def:extsymp} of the extended symplectic category $\cC^\#=\Symp^\#$ and extends it to a 2-category by adding the geometric composition moves as 2-morphisms, whose vertical and horizontal compositions are well defined due to the equivalence relation $\sim$ being compatible with the horizontal 1-composition $\circ^1_{\rm h}=\#$ by concatenation in $\Symp^\#$. 
If one wishes to avoid empty chains, it can be viewed as a bicategory in which horizontal composition is strictly associative, but the diagonals $\Delta_M\subset M^-\times M$ only provide identities up to 2-isomorphism given by the embedded compositions $L_{01}\circ\Delta_{M_1}=L_{01}$ and $\Delta_{M_1} \circ L_{12}=L_{12}$.
\end{example}

Conversely, any bicategory (not just those arising from Cerf decompositions as in Example~\ref{ex:ext0}) gives rise to a 1-category by the following quotient construction.

\begin{remark}[Quotient of a bicategory by 2-morphisms] \label{rmk:2cat-quotient} \rm 
Let $\cD$ be a bicategory.
Two 1-morphisms $f,g\in \Mor_\cD^1(x,y)$ are called {\bf isomorphic} $\mathbf{f\sim g}$ if there exist 2-morphisms $\alpha,\beta\in\Mor_\cD^2(f,g)$ whose vertical compositions 
$\alpha \circ_{\rm v} \beta = \id_f$ and $\beta \circ_{\rm v} \alpha = \id_g$ are the identites.
This defines an evidently symmetric relation, which is moreover transitive and reflexive because the vertical composition $\circ_{\rm v}$ is associative and unital.

Now the bicategory $\cD$ induces a quotient 1-category $|\cD|:=\cD/\!\!\sim\,$ with the same objects $\Obj_{|\cD|}:=\Obj_\cD$ and morphisms $\Mor_{|\cD|}:=\Mor^1_{\cD}/\!\!\sim\,$ given as 1-morphisms modulo the equivalence relation $\sim$.  
Here the horizontal 1-composition in $\cD$ descends to a well defined composition on the quotient $|\cD|$ due to its compatibility with the equivalence relation -- which is the content of the assumption that the horizontal bifunctor in $\cD$ is compatible with identities and vertical composition. 
\end{remark}

\subsection{Higher bordism categories} \label{ss:bord}

We begin with a more rigorous construction of the bordism category $\Bor_{d+1}$ in Example~\ref{ex:1bor} as the quotient (as in Remark~\ref{rmk:2cat-quotient}) of a bicategory $\Bor_{d+1+\eps}$ comprising $d$-manifolds, $(d+1)$-cobordisms, and diffeomorphisms of $(d+1)$-cobordisms.
After that, we restrict to the case $d=2$ and extend $\Bor_{2+1+\eps}$ to a rigorous construction of a bordism bicategory $\Bor_{2+1+1}$ comprising 2-manifolds, 3-cobordisms, and 4-dimensional manifolds with boundary and corners.

\begin{example} \label{ex:1epsbor} \rm 
The {\bf $\mathbf{d+1+}\boldsymbol{\epsilon}$ bordism bicategory ${\rm \mathbf{Bor}}_{\mathbf{d+1+\boldsymbol{\epsilon}}}$ of d-manifolds, cobordisms, and diffeomorphisms} is constructed as follows, with illustrations in Figures~\ref{fig:21eps} and ~\ref{fig:21epsprime}.
\begin{itemize}
\item
Objects in $\Obj_{\Bor_{d+1+\eps}}$ are closed, oriented, $d$-dimensional manifolds.
\item
1-morphisms in $\Mor^1_{\Bor_{d+1+\eps}}(\Sigma_-,\Sigma_+)$ are the representatives of morphisms in $\Bor_{d+1}$, that is triples $(Y,\iota^-,\iota^+)$ consisting of a compact, oriented, $(d+1)$-dimensional manifold $Y$ with boundary and orientation perserving embeddings $\iota^\pm:[0,1]\times\Sigma_\pm\to Y$ to tubular neighbourhoods of the boundary components $\partial Y=\iota^-(0,\Sigma_-)\sqcup \iota^+(1,\Sigma_+)$. For reasons that we will explain in item $\circ$ below, we also require the images of $\iota^\pm$ to be disjoint in $Y$.
\item
2-morphisms in $\Mor^2_{\Bor_{d+1+\eps}}((Y,\iota^\pm_Y),(Z,\iota^\pm_Z))$ between 1-morphisms $(Y,\iota^\pm_Y)$, $(Z,\iota^\pm_Z)\in \Mor^1_{\Bor_{d+1+\eps}}(\Sigma_-,\Sigma_+)$ are orientation preserving diffeomorphisms $\Psi:Y\to Z$ that intertwine the tubular neighbourhood embeddings, $\Psi\circ\iota^\pm_Y = \iota^\pm_Z$.
\begin{figure}[!h]
\centering
\includegraphics[width=5.5in]{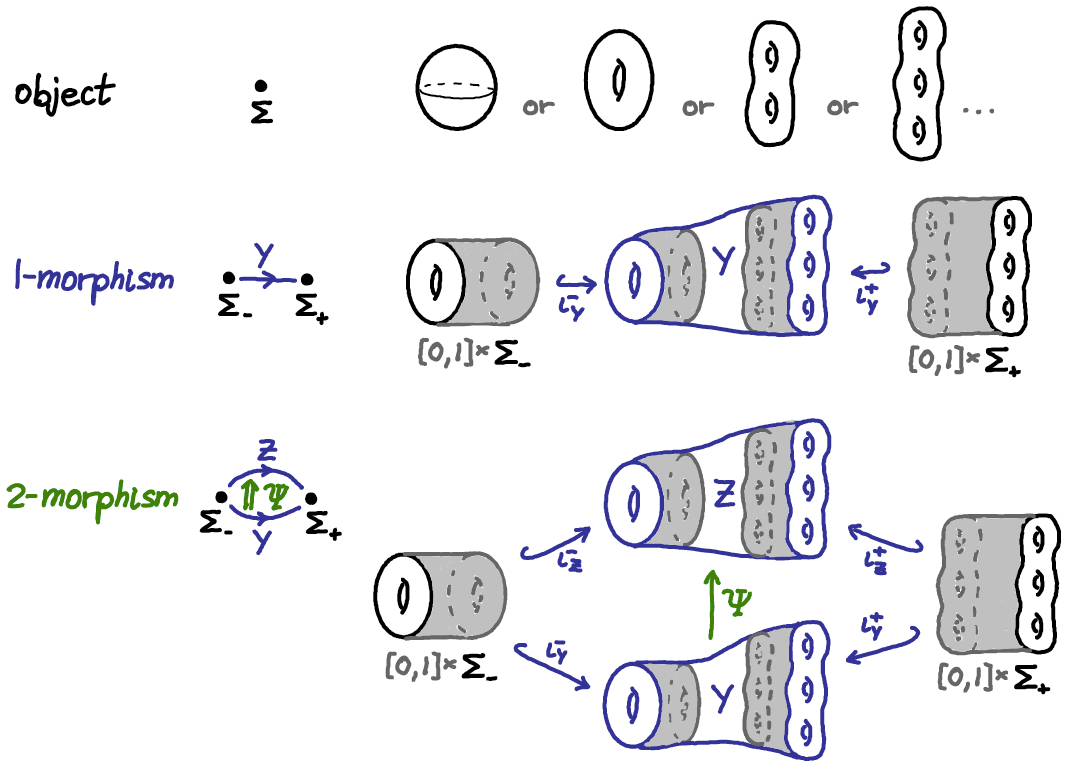}
\caption{Objects, 1-morphisms, and 2-morphisms of the bordism bicategory ${\rm Bor}_{2+1+\eps}$.}
\label{fig:21eps}
\end{figure}
\item
Vertical composition is 
by composition of diffeomorphisms ${\Phi \circ_{\rm v} \Psi= \Psi \circ \Phi}$. This is evidently associative and has units $\id_{(Y,\io^\pm)} = \id_Y$, so the set $\Mor^1_{\Bor_{d+1+\eps}}(\Sigma_-,\Sigma_+)$ of 1-morphisms between fixed objects $\Sigma_\pm$ forms a category.
\item
Horizontal 1-composition is given by the gluing operation\footnote{
Here we could allow overlapping tubular neighbourhoods, since $\iota^-_{01},\iota^+_{12}$ induce well defined tubular neighbourhoods in the glued cobordism even if their images are not disjoint from the gluing region $\im\, \iota^+_{01}\simeq \im\,\iota^+_{12}$.
}
$$
(Y_{01},\iota^-_{01},\iota^+_{01}) \circ^1_{\rm h} (Y_{12},\iota^-_{12},\iota^+_{12})
\,:=\;
\left(  \, \quo{Y_{01}\sqcup Y_{12}}{\iota_{01}^+(s,x)\sim \iota^-_{12}(s,x)} \, ,\iota^-_{01},\iota^+_{12} \right).
$$
\item 
Horizontal 2-composition
of $\Psi_{ij}\in \Mor^2_{\Bor_{d+1+\eps}}((Y_{ij},\iota^\pm_{Y_{ij}}),(Z_{ij},\iota^\pm_{Z_{ij}}) )$ for $ij= 01$ and $ij=12$ is given by gluing of the diffeomorphisms,
\begin{align*}
\Psi_{01} \circ^2_{\rm h} \Psi_{12} \,:\; 
 \quo{Y_{01}\sqcup Y_{12}}{\iota_{Y_{01}}^+(s,x)\sim \iota^-_{Y_{12}}(s,x)}
& \;\longrightarrow\;\;
 \quo{Z_{01}\sqcup Z_{12}}{\iota_{Z_{01}}^+(s,x)\sim \iota^-_{Z_{12}}(s,x)} \\
y\in Y_{ij} & \;\longmapsto\;\; \Psi_{ij}(y) ,
\end{align*}
which is well defined since for $\io^+_{Y_{01}}(s,x)\sim \io^-_{Y_{12}}(s,x)$ we have 
$\Psi_{01}(\io^+_{Y_{01}}(s,x)) = \io^+_{Z_{01}}(s,x) \sim
\io^-_{Z_{12}}(s,x) =\Psi_{12}( \io^-_{Y_{12}}(s,x))$. 
\begin{figure}[!h]
\centering
\includegraphics[width=5.5in]{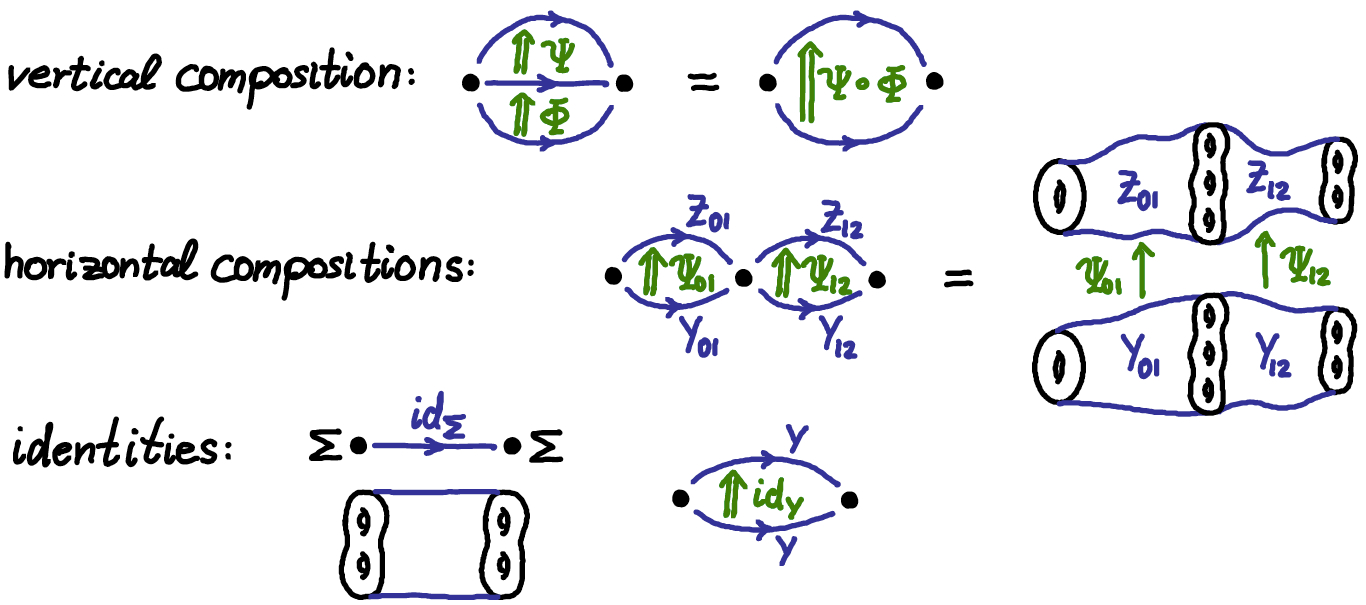}
\caption{Composition and identities in ${\rm Bor}_{2+1+\eps}$.}
\label{fig:21epsprime}
\end{figure}
\item
Horizontal composition is compatible with identities since for composable cobordisms $Y_{01},Y_{12}$ both 2-morphisms
$\id_{Y_{01}}\circ_{\rm h}^2 \id_{Y_{12}}$ and $\id_{Y_{01}\circ_{\rm h}^1 Y_{12}}$ are the identity map 
on $(Y_{01}\sqcup Y_{12})/\!\!\sim\,$.
\item
Horizontal 2-composition $\circ_{\rm h}^2$ is compatible with vertical composition since, when given diffeomorphisms $\Phi_{ij}\in \Mor^2_{\Bor_{d+1+\eps}}((Y_{ij},\iota^\pm_{Y_{ij}}),(Z_{ij},\iota^\pm_{Z_{ij}}) )$ 
for $ij=01, 12$, both $(\Phi_{01} \circ^2_{\rm h} \Phi_{12}) \circ_{\rm v} ( \Psi_{01} \circ^2_{\rm h} \Psi_{12})$ 
and $(\Phi_{01} \circ_{\rm v} \Psi_{01}) \circ^2_{\rm h} ( \Phi_{12} \circ_{\rm v} \Psi_{12})$ 
are given by the compositions $\Psi_{ij}\circ\Phi_{ij}$ on each part $Y_{ij} \subset \qu{Y_{01}\sqcup Y_{12}}{\iota^+\sim \iota^-}$.
\item
Horizontal 1-composition is strictly associative since for composable $(d+1)$-dimensional cobordisms $Y_{01},Y_{12},Y_{23}$ both $Y_{01} \circ_{\rm h}^1 \bigl( Y_{12} \circ_{\rm h}^1 Y_{23}  \bigr)$
and $\bigl( Y_{01} \circ_{\rm h}^1 Y_{12} \bigr)\circ_{\rm h}^1 Y_{23}$
are given by the disjoint union $Y_{01} \sqcup Y_{12} \sqcup Y_{23}$ 
modulo the equivalence relation\footnote{
If the embeddings $\iota^\pm_{12}$ had overlapping images, we could make the same construction by completing the equivalence relation with compositions
$\iota_{01}^+(s,x)\sim \iota^-_{12}(s,x)=\iota_{12}^+(s',x')\sim \iota^-_{23}(s',x')$.
} 
given by 
$\iota_{01}^+(s,x)\sim \iota^-_{12}(s,x)$, $\iota_{12}^+(s,x)\sim \iota^-_{23}(s,x)$.
\item
Horizontal 2-composition is associative since for composable $Y_{ij}$ and $Z_{ij}$ as above and 
$\Phi_{ij}\in \Mor^2_{\Bor_{d+1+\eps}}(Y_{ij},Z_{ij})$ both
$\Phi_{12}\circ^2_{\rm h}( \Phi_{23}\circ^2_{\rm h}\Phi_{34})$
and $(\Phi_{12}\circ^2_{\rm h}\Phi_{23}) \circ^2_{\rm h}\Phi_{34}$
are given by $\Phi_{ij}$ on each part 
$Y_{ij} \subset \qu{Y_{01}\sqcup Y_{12}\sqcup Y_{23}}{\iota_{01}^+\sim \iota^-_{12}, \iota_{12}^+\sim \iota^-_{23}}$.

\item[$\circ$]
Horizontal 1-composition $\circ_{\rm h}^1$ can also be made strictly unital, thus giving rise to a 2-category, if we allow the tubular neighbourhood embeddings to have overlapping image. Then $1_{\Sigma,1}:=([0,1]\times\Sigma,\io^\pm)$ with the canonical embeddings $\io^\pm=\id_{[0,1]\times\Sigma}$ would be a strict unit.
However, other cylindrical cobordisms
$$
1_{\Sigma,\delta}:=([0,1]\times\Sigma, \io^\pm_\delta)
\quad\text{with}\quad \io^-_\delta(s,x):=(\delta s,x),\;  \io^+_\delta(s,x):=(1 - \delta +\delta s ,x)
$$ 
for $0<\delta<1$ would have no 2-morphisms to the unit since such a diffeomorphism on $[0,1]\times\Sigma$ would be required to map $\im\, \io^-=[0,1]\times\Sigma$ to $\im\,\io^-_\delta=[0,\delta]\times\Sigma$ and $\im\, \io^+=[0,1]\times\Sigma$ to $\im\,\io^+_\delta=[1-\delta,1]\times\Sigma$. 
So this bordism 2-category $\Bor'_{d+1+\eps}$ would have too few 2-morphisms to achieve our topological vision of having just one morphism in $|\Bor'_{d+1+\eps}|= \Bor_{d+1}$ that is represented by $[0,1]\times\Sigma$ with the canonical boundary identifications.

By requiring the tubular neighbourhood embeddings to have disjoint images, we disallow $1_{\Sigma,\delta}$ for $\delta\ge\frac 12$ as 2-morphism. On the other hand, the cylindrical cobordisms for $0<\delta\neq\delta'<\frac 12$ are all equivalent,
$$
1_{\Sigma,\delta}\sim 1_{\Sigma,\delta'}
\quad\text{since}\quad 
\id_{[0,1]\times\Sigma}=\Psi\circ\Psi^{-1}, \; \id_{[0,1]\times\Sigma}=\Psi^{-1}\circ\Psi ,
$$
for any diffeomorphism $\Psi$ of $[0,1]\times\Sigma$ which extends
$$
(\iota^\pm_{\delta'})^{-1}\circ \iota^\pm_{\delta} :  
\bigl([0,\delta]\cup [1-\delta,1]\bigr)\times\Sigma \to \bigl([0,\delta']\cup [1-\delta',1]\bigr)\times\Sigma .
$$
\item
Horizontal 1-composition $\circ_{\rm h}^1$ in $\Bor_{d+1+\eps}$ is unital up to 2-isomorphism, with weak units for any manifold $\Sigma\in \Obj_{\Bor_{d+1+\eps}}$ given by the cylindrical cobordisms $1_{\Sigma,\delta}\in\Mor_{\Bor_{d+1+\eps}}^1(\Sigma,\Sigma)$ for any $0<\delta< \frac 12$. Indeed, for any appropriate
$(Y,\io^\pm_Y), (Z,\io^\pm_Z) \in\Mor_{\Bor_{d+1+\eps}}^1$ we have 
$$
(Y,\io^\pm_Y) \circ_{\rm h}^1 1_{\Sigma,\delta} \sim  (Y,\io^\pm_Y) , \qquad\qquad
1_{\Sigma,\delta}  \circ_{\rm h}^1 (Z,\io^\pm_Z) \sim  (Z,\io^\pm_Z)  
$$
via appropriate diffeomorphisms 
$$
\Psi: \quotient{Y\sqcup ([0,1]\times\Sigma)}{\io^+_Y\sim \io^-_\delta} \to Y,
 \qquad
\Phi: \quotient{([0,1]\times\Sigma) \sqcup Z}{\io^+_\delta \sim \io^-_Z} \to Z .
$$ 
Here $\Psi$ (and $\Phi$ similarly) is constructed as follows:
Extend $\iota^+_Y$ to an embedding $\tilde\io^+_Y:[-1,1]\times\Sigma\to Y\less\im\,\io^-_Y$.
Then we have a natural diffeomorphism 
$$
\quot{Y\sqcup ([0,1]\times\Sigma)}{\io^+_Y(s,\cdot)\sim \io^-_\delta(s,\cdot) \; \forall s\in[0,1]}
\simeq 
\quot{Y \sqcup ( [-\delta,1]\times\Sigma )}{\tilde\io^+_Y(s,\cdot)\sim \tilde\io^-_\delta(s, \cdot)
\; \forall s\in[-1,1]} 
$$
with $\tilde\io^-_\delta:[-1,1]\times\Sigma \to [-\delta,1]\times\Sigma, (s,x)\mapsto (\delta s,x)$.
Now we can construct $\Psi$ by the identity on $Y\less\im\,\tilde\io^+_Y$ and a diffeomorphism
$[-\delta,1]\times\Sigma \to \tilde\io^+_Y([-1,1]\times\Sigma)$ given by $\tilde\io^+_Y\circ(\tilde\io^-_\delta)^{-1}$ near $\{-\delta\}\times\Sigma \to \tilde\io^+_Y(-1,\Sigma)$ and $\io^+_Y\circ(\io^+_\delta)^{-1}$ on $[1-\delta,1]\times\Sigma \to \io^+_Y([0,1]\times\Sigma)$. It intertwines the boundary embeddings $\Psi\circ\io^+_\delta = \io^+_Y$ by construction, as required for a 2-morphism from $(Y,\io^\pm_Y) \circ_{\rm h}^1 1_{\Sigma,\delta}$ to $(Y,\io^\pm_Y)$. 
\end{itemize}
This finishes the construction of the bordism bicategory ${\rm Bor}_{d+1+\epsilon}$.
It particularly contains representatives of the cylindrical cobordism $Z_\phi\in\Mor_{\Bor_{d+1}}(\Sigma_0,\Sigma_1)$ associated to a diffeomorphism $\phi:\Sigma_0\to\Sigma_1$ in \eqref{eq:Zphi},  
for any $0<\delta<\frac 12$ given by
\[
\widehat Z_{\phi,\delta}:=\left(
\begin{aligned} 
[0,1]\times \Sigma_1 \,,\, 
& \io^-_\delta: (s,x)\mapsto (\delta s, \phi(x)) \,,\,  \\
& \io^+_\delta: (s,x)\mapsto (1-\delta+\delta s,x) 
\end{aligned} \right) 
 \in\Mor^1_{\Bor_{d+1}}(\Sigma_0,\Sigma_1) .
\]
These also reproduce the identity 1-morphisms $\widehat Z_{\id,\delta} = 1_{\Sigma,\delta} \in\Mor^1_{\Bor_{d+1}}(\Sigma,\Sigma)$.

The {\bf connected $\mathbf{d+1+}\boldsymbol{\epsilon}$ bordism bicategory ${\rm \mathbf{Bor}^{\rm \mathbf{conn}}}_{\mathbf{d+1+\boldsymbol{\epsilon}}}$ of connected d-manifolds, connected cobordisms, and diffeomorphisms} is constructed analogously, using the objects and representatives of morphisms of $\Bor^{\rm conn}_{d+1}$.
\end{example}

\begin{remark}[Quotient construction of the d+1 bordism category] \label{rmk:1epsbor} \rm 
Taking the quotient of the bordism bicategories $\Bor_{d+1+\eps}$ and $\Bor^{\rm conn}_{d+1+\eps}$ by their 2-morphisms (i.e.\ diffeomorphisms of $d+1$-cobordisms) as in Remark~\ref{rmk:2cat-quotient} now yields rigorous definitions of the (connected) bordism categories $\Bor_{d+1}:=\Bor_{d+1+\eps}/\!\!\sim\,$ and $\Bor^{\rm conn}_{d+1}:=\Bor^{\rm conn}_{d+1+\eps}/\!\!\sim\,$ 
outlined in Examples~\ref{ex:1bor} and \ref{ex:connbor}.

In particular, the cylindrical cobordism $Z_\phi$ associated in \eqref{eq:Zphi} to a diffeomorphism $\phi:\Sigma_0\to\Sigma_1$ is more rigorously defined as the equivalence class 
$Z_\phi := [\widehat Z_{\phi,\delta}]\in\Mor_{\Bor_{d+1}}(\Sigma_0,\Sigma_1)$,
which is independent of $0<\delta<\frac 12$.
This also reproduces the identity morphisms $Z_{\id} = [1_{\Sigma,\delta}]\in\Mor_{\Bor_{d+1}}(\Sigma,\Sigma)$ in $\Bor_{d+1}$.
\end{remark}

For our applications we will restrict to dimension $d=2$ and extend the above constructions to a bordism bicategory which includes ``cobordisms of cobordisms'' given by 4-manifolds with boundaries and corners.
Again, we give general constructions for $d\ge0$ and restrict to $d\ge 2$ to obtain a connected theory. 
We also use the illustrations in case $d=2$ to give a preview of the string diagram notation in \S\ref{ss:quilt}.

\begin{figure}[!h]
\centering
\includegraphics[width=5.5in]{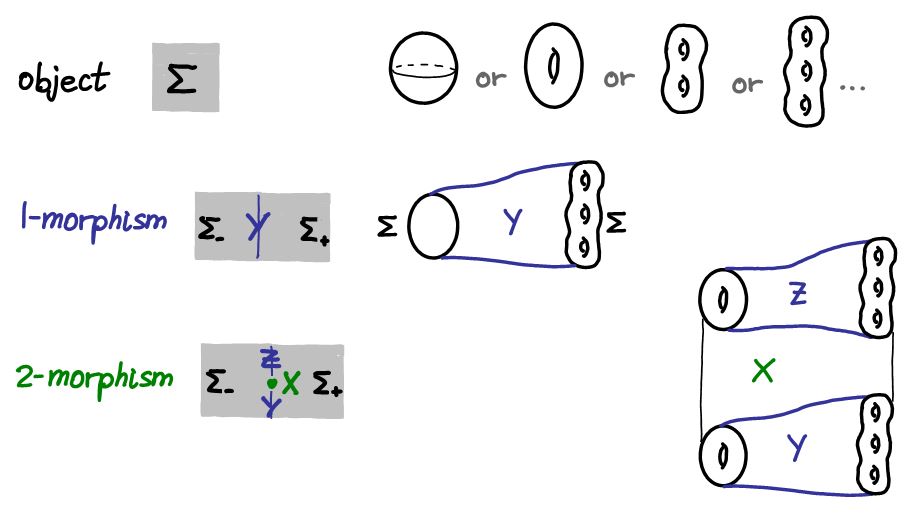}
\caption{Objects, 1-morphisms, and 2-morphisms of ${\rm Bor}_{2+1+1}$ and their string diagram notation. These basic diagrams represent 4-manifolds given by squares times surfaces, intervals times 3-cobordisms, and 4-cobordisms with corners.
}
\label{fig:2bordismstuff}
\end{figure}

\begin{example} \label{ex:2bor} \rm 
The {\bf bordism bicategory ${\rm \mathbf{Bor}}_{\mathbf{d+1+1}}$} for $d\ge 0$ consists of the following, with representation by string diagrams illustrated in Figures~\ref{fig:2bordismstuff} and \ref{fig:borcomp}.

\begin{itemize}
\item
Objects in $\Obj_{\Bor_{d+1+1}}$ are closed oriented d-manifolds $\Sigma$ as in $\Bor_{d+1+\eps}$.
\item
1-morphisms in $\Mor^1_{\Bor_{d+1+1}}(\Sigma_-,\Sigma_+)$ are triples $(Y,\iota^-,\iota^+)$ of a compact, oriented (d+1)-cobordism $Y$ with disjoint embeddings $\iota^\pm:[0,1]\times\Sigma_\pm\to Y$ to neighbourhoods of the boundary parts $\partial Y=\iota^-(0,\Sigma_-)\sqcup \iota^+(1,\Sigma_+)$
as in $\Bor_{d+1+\eps}$. 
\item
2-morphisms, i.e.\ morphisms in the category $\Mor^1_{\Bor_{d+1+1}}(\Sigma_-,\Sigma_+)$
are equivalence classes of tuples $\bigl[(X,\iota^+_X,\iota^-_X,\kappa^-,\kappa^+)\bigr]\in \Mor^2_{\Bor_{d+1+1}}((Y,\iota^\pm_Y),(Z,\iota^\pm_Z))$ consisting of a compact, oriented (d+2)-manifold $X$ with boundary and corners and four orientation preserving embeddings as indicated in Figure~\ref{fig:2bordism},
$$
\iota^\pm_X:  [0,1]\times[0,1] \times \Sigma_\pm \to X , \quad
\kappa^-:[0,1]\times Y\to X, \quad
\kappa^+:[0,1]\times Z\to X .
$$ 
\begin{figure}[!h]
\centering
\includegraphics[width=5.5in]{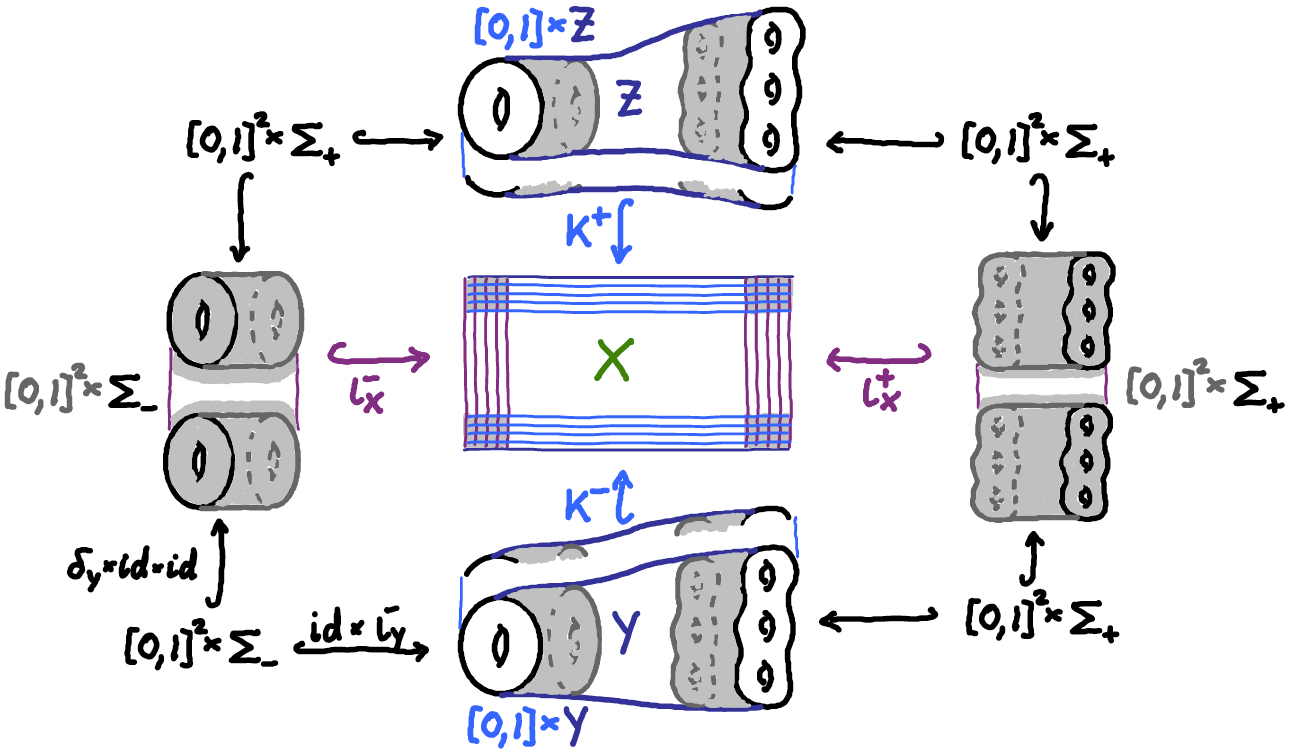}
\caption{A 2-morphism in ${\rm Bor}_{2+1+1}$ consists of a 4-manifold $X$ with a number of embeddings to collar neighbourhoods of its boundary strata, which are compatible near the corners.}
\label{fig:2bordism}
\end{figure}
The embeddings $\iota^\pm$ and $\kappa^\pm$ are required to cover the boundary 
$$
\partial X \,=\, \kappa^-(0,Y) \,\sqcup\, \kappa^+(1,Z) \,\sqcup\, \iota^-((0,1),0,\Sigma_-) \,\sqcup\, \iota^+((0,1),1,\Sigma_+)
$$
in such a way that both pairs $\kappa^\pm$ and $\iota^\pm$ have disjoint images, but we have mixed overlaps on which $\kappa^\pm$ intertwines $\iota^\pm$ with the boundary identifications $\iota^\pm_Y,\iota^\pm_Z$ in the sense that for some $0<\delta_Y^\pm,\delta_Z^\pm<\frac 12$ and all $s,t\in[0,1]$, $x\in\Sigma_\pm$ we have
\begin{align}\label{2borcomp}
\kappa^-\bigl( s, \iota^\pm_Y (t,x) \bigr) = \iota^\pm_X (\delta_Y^\pm s, t, x)  , 
\quad
\kappa^+ \bigl( s, \iota^\pm_Z (t,x) \bigr) = \iota^\pm_X (1-\delta_Z^\pm + \delta_Z^\pm s, t, x) .
\end{align}
Two such tuples are equivalent,
$\bigl(X_0,\iota^+_0,\iota^-_0,\kappa_0^-,\kappa_0^+ \bigr) \sim \bigl(X_1,\iota^+_1,\iota^-_1,\kappa_1^-,\kappa_1^+ \bigr)$, if there exists a diffeomorphism $F:X_0\to X_1$ that intertwines the embeddings, i.e.\ 
$F\circ \iota^\pm_0 =  \iota^\pm_1$ and $F\circ \kappa^\pm_0 =  \kappa^\pm_1$.
\item[$\circ$]
The 2-morphisms $\Psi:Y\to Z$ in $\Bor_{d+1+\eps}$ appear in $\Bor_{d+1+1}$
as the cylindrical cobordisms of cobordisms 
$I_\Psi:=\bigl[([0,1]\times Z,\iota^\pm,\kappa_\delta^\pm)\bigr]\in \Mor^2_{\Bor_{d+1+1}}$
with
$$
\io^\pm(s,t,x):=(s,\io^\pm_Z(t,x)), \quad
\kappa^-_\delta(s,y):=\bigl(\delta s, \Psi(y)\bigr),\quad 
\kappa^+_\delta(s,z):=\bigl(1-\delta+\delta s,z\bigr) .
$$
This is illustrated in Figure~\ref{fig:inclusion} and may help with understanding the compatibility conditions \eqref{2borcomp} for the embeddings, which are naturally satisfied by $\Psi\circ\iota^\pm_Y = \iota^\pm_Z$, 
\begin{align*}
\kappa^-_\delta\bigl( s, \iota^\pm_Y (t,x) \bigr) 
&= \bigl(\delta s, \Psi( \iota^\pm_Y (t,x) )\bigr)
= \bigl(\delta s, \iota^\pm_Z (t,x) \bigr)
=\iota^\pm (\delta s, t, z)   , \\
\kappa^+_\delta \bigl( s, \iota^\pm_Z (t,x) \bigr) 
&= \bigl(1-\delta+\delta s, \iota^\pm_Z (t,z)  \bigr)
= \iota^\pm_X (1-\delta + \delta s, t, x) .
\end{align*}
\begin{figure}[!h]
\centering
\includegraphics[width=5.5in]{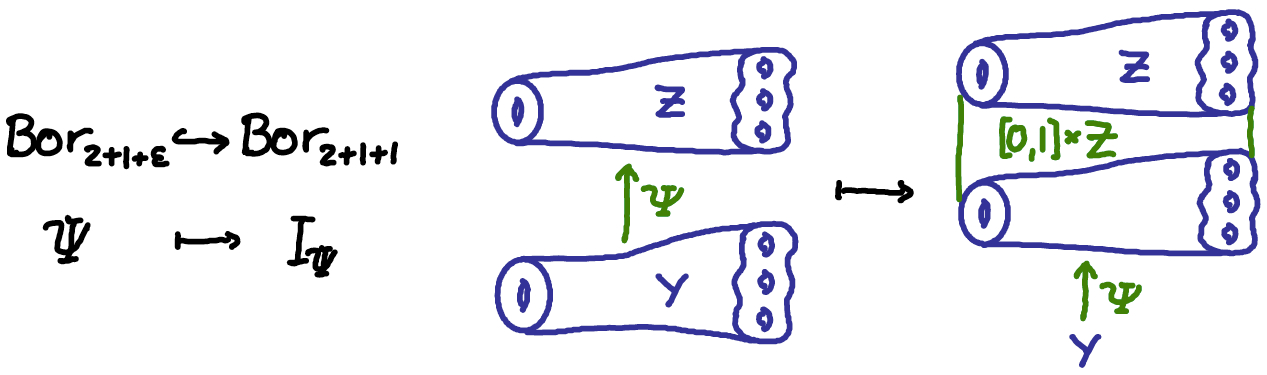}
\caption{Inclusion of $\Bor_{d+1+\eps}$ in $\Bor_{d+1+1}$.}
\label{fig:inclusion}
\end{figure}
\item
Vertical composition of $\bigl[\bigl(X^{ij},\io^\pm_{ij},\kappa^\pm_{ij}\bigr)\bigr]
\in \Mor^2_{\Bor_{d+1+1}}\bigl((Y^i,\iota_{Y^i}^\pm),(Y^j,\iota_{Y^j}^\pm))$ labeled by $ij=01$ and $ij=12$ 
between $(Y^i,\iota_{Y^i}^\pm)\in\Mor^1_{\Bor_{d+1+1}}(\Sigma_-,\Sigma_+)$ for $i=0,1,2$ is given by gluing the (d+2)-manifolds and embeddings as illustrated in
Figure~\ref{fig:borcomp}
$$
\bigl( X^{01},\iota^\pm_{01},\kappa^\pm_{01} \bigr) \circ_{\rm v} \bigl( X^{12},\iota^\pm_{12},\kappa^\pm_{12} \bigr)
\,:=\;
\left(  \, \quo{X^{01}\sqcup X^{12}}{\kappa_{01}^+(s,y)\sim \kappa^-_{12}(s,y)} \, ,
\iota^\pm_{01} \circ_{\rm h} \iota^\pm_{12}, \kappa^-_{01} , \kappa^+_{12} \right).
$$
\begin{figure}[!h]
\centering
\includegraphics[width=5.5in]{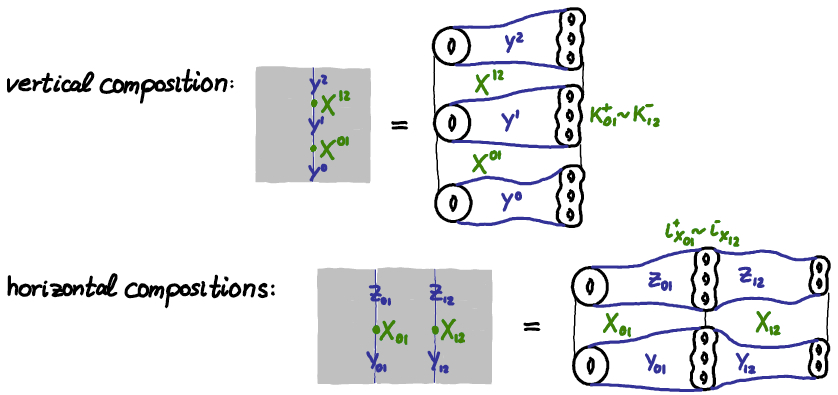}
\caption{Compositions of 2-morphisms in $\Bor_{2+1+1}$ and their string diagram notation. These more complicted diagrams represent constructions of 4-manifolds by gluing squares times surfaces, intervals times 3-cobordisms, and 4-cobordisms with corners along common boundary strata.
}
\label{fig:borcomp}
\end{figure}
Here the tubular neighbourhoods of the common boundaries $\partial Y^i \simeq \Sigma_-^-\sqcup \Sigma_+$ 
are glued in the same way as the horizontal 2-composition $\circ^2_{\rm h}$ in $\Bor_{d+1+\eps}$, that is
\begin{align*}
\io^\pm_{01} \circ_{\rm h} \io^\pm_{12} \,:\; 
\quo{ Q_\pm^{01}\sqcup Q_\pm^{12}}{\iota_{Q_\pm^{01}}^+\sim \iota^-_{Q_\pm^{12}}}
& \;\longrightarrow\;\;
 \quo{X^{01}\sqcup X^{12}}{\kappa_{01}^+\sim \kappa^-_{12}} \\
(s,t,x) \in Q_\pm^{ij} & \;\longmapsto\;\; \io^\pm_{ij}(s,t,x) 
\end{align*}
where $Q_\pm^{ij}:=[0,1]^2\times \Sigma_\pm$ are representatives of 2-morphisms with the embeddings 
$\iota^\pm_{Q_\pm^{ij}}:[0,1]\times [0,1]\times \Sigma \to Q_\pm^{ij}$
chosen so as to make the glued map well defined due to the compatibility conditions in \eqref{2borcomp}
for some $\delta^\pm_{ij}>0$,
\begin{align*}
 \kappa^+_{01} \bigl( s, \iota^\pm_{Y^1} (t,x) \bigr)
=
\iota^\pm_{01} (1-\delta^\pm_{01} + \delta^\pm_{01} s, t, x) 
&=:
\iota^\pm_{01} \bigl( \iota^+_{Q_\pm^{01}} (s,t,x) \bigr) \\
 \sim  \kappa^-_{12}\bigl( s, \iota^\pm_{Y^1} (t,x) \bigr) 
= 
\iota^\pm_{12} (\delta^\pm_{12} s, t, x)
&=:
\iota^\pm_{12} \bigl( \iota^-_{Q_\pm^{12}} (s,t,x) \bigr) .
\end{align*}
That is, we set
$\iota^+_{Q_\pm^{01}} (s,t,x) := (1-\delta^\pm_{01} + \delta^\pm_{01} s, t, x) $ 
and
$\iota^-_{Q_\pm^{12}} (s,t,x) := (\delta^\pm_{12} s, t, x)$.

This gluing construction on the level of representatives yields a well defined vertical composition of equivalence classes because the equivalences are given by diffeomorphisms which intertwine the embeddings that are used to glue. Associativity follows from direct associativity of gluing, and units are provided by the cylindrical cobordisms of cobordisms $\id_{(Y,\iota^\pm)}:=I_{\id_Y}$ associated above to the identity map $\Psi=\id_Y:Y\to Y$ just as for $\circ^1_{\rm h}$ in Example~\ref{ex:1epsbor}.
Thus $\Mor^1_{\Bor_{d+1+1}}(\Sigma_-,\Sigma_+)$ for fixed objects 
$\Sigma_\pm$ forms a category.
\item[$\circ$]
Note that the vertical composition of 2-morphisms arising from diffeomorphisms in $\Mor^2_{\Bor_{d+1+\eps}}$ is compatible with the vertical composition in $\Bor_{d+1+\eps}$, that is
\begin{equation}\label{eq:Ipsiphi}
I_{\Phi_{01}} \circ_{\rm v}  I_{\Phi_{12}} = I_{\Phi_{12}\circ\Phi_{01}} \qquad\forall
\Phi_{ij} \in \Mor^2_{\Bor_{d+1+\eps}}((Y_{i},\iota^\pm_{Y_{i}}),(Y_{j},\iota^\pm_{Y_{j}}) ).
\end{equation}
Indeed, we have
$I_{\Phi_{01}} \circ_{\rm v}  I_{\Phi_{12}} =\bigl[\bigl( [0,1]\times Y_1,\iota^\pm_{01} ,\kappa^\pm_{\delta,\Phi_{01}} \bigr)\bigr] \circ_{\rm v} \bigl[\bigl( [0,1]\times Y_2,\iota^\pm_{12},\kappa^\pm_{\delta,\Phi_{12}} \bigr)\bigr]$
with $\iota^\pm_{01} =\id_{[0,1]}\times \iota^\pm_{Y_1}$, $\iota^\pm_{12} =\id_{[0,1]}\times \iota^\pm_{Y_2}$, 
$\kappa^-_{\delta,\Psi}(s,y)=\bigl(\delta s, \Psi(y)\bigr)$, and 
$\kappa^+_{\delta,\Psi}(s,z)=\bigl(1-\delta+\delta s,z\bigr)$ 
is represented by the (d+2)-manifold
$$
\quo{\bigl([0,1]\times Y_1\bigr) \sqcup \bigl([0,1]\times Y_2\bigr) }{
(1-\delta+\delta s,y) \sim (\delta s, \Phi_{12}(y))} 
\;\simeq\;
[0,2-\delta]\times Y_2 ,
$$
via the diffeomorphism induced by $\id_{[0,1]} \times \Phi_{12} : [0,1]\times Y_1 \to [0,1]\times Y_2$
and
$(r\mapsto r + 1-\delta)\times \id_{Y_2}:  [0,1]\times Y_2 \to [1-\delta,2-\delta]\times Y_2$.
The corresponding embeddings are
\begin{align*}
\iota^\pm_{01} \circ_{\rm h} \iota^\pm_{12} 
&\;=\; 
\bigl(\id_{[0,1]}\times \iota^\pm_{Y_1}\bigr) \circ_{\rm h} \bigl(\id_{[0,1]}\times \iota^\pm_{Y_2}\bigr)
\;\simeq\; 
\id_{[0,2-\delta]} \times \iota^\pm_{Y_2} , \\
\kappa^-_{\delta,\Phi_{01}}(s,y) &\;=\; \bigl(\delta s, \Phi_{01}(y)\bigr)
\;\simeq\;  \bigl(\delta s , \Phi_{12}\bigl(\Phi_{01}(y)\bigr)\bigr), \\
\kappa^+_{\delta,\Phi_{12}}(s,y) &\;=\; \bigl(1-\delta+\delta s,z\bigr) 
\;\simeq\;  \bigl(2-2\delta+\delta s,z\bigr) .
\end{align*}
This is equivalent to the representative of $I_{\Phi_{12}\circ\Phi_{01}}$ with constant $0<\frac{\delta}{2-\delta}<\frac 12$ via linear rescaling in the first factor $[0,2-\delta]\times Y_2\simeq[0,1]\times Y_2$.
\item
Horizontal 1-composition is given by gluing as in $\Bor_{d+1+\eps}$, 
$$
(Y_{01},\iota^-_{01},\iota^+_{01}) \circ^1_{\rm h} (Y_{12},\iota^-_{12},\iota^+_{12})
\,:=\;
\left(  \, \quo{Y_{01}\sqcup Y_{12}}{\iota_{01}^+(s,x)\sim \iota^-_{12}(s,x)} \, ,\iota^-_{01},\iota^+_{12} \right) .
$$
\item 
Horizontal 2-composition of $\bigl[\bigl(X_{ij},\io^\pm_{ij},\kappa^\pm_{ij}\bigr)\bigr] \in \Mor^2_{\Bor_{d+1+1}}((Y_{ij},\iota^\pm_{Y_{ij}}),(Z_{ij},\iota^\pm_{Z_{ij}}) )$ between
$(Y_{ij},\iota^\pm_{Y_{ij}}),(Z_{ij},\iota^\pm_{Z_{ij}})\in\Mor^1_{\Bor_{d+1+1}}(\Sigma_i,\Sigma_j)$ 
for $ij= 01$ and $ij=12$ is given by gluing the (d+2)-manifolds and embeddings  as illustrated in Figure~\ref{fig:borcomp}
\begin{align*}
\bigl( X_{01},\iota^\pm_{01},\kappa^\pm_{01} \bigr) \circ^2_{\rm h} \bigl( X_{12},\iota^\pm_{12},\kappa^\pm_{12} \bigr)
&\,:=\;
\left(  \, \quo{X_{01}\sqcup X_{12}}{\iota_{01}^+\sim \iota^-_{12}} \, ,
\iota^-_{01}, \iota^+_{12}, \kappa^\pm_{01}  \circ_{\rm h}  \kappa^\pm_{12} \right),\\
\kappa^-_{01} \circ_{\rm h} \kappa^-_{12} \,:\; 
[0,1] \times \quo{ Y_{01}\sqcup Y_{12}}{\iota_{Y_{01}}^+\sim \iota^-_{Y_{12}}}
& \;\to\;
X_{01}\circ^2_{\rm h} X_{12}  , \quad
\bigl(s,y\in Y_{ij} \bigr) \;\mapsto\; \kappa^-_{ij}(s,y) , \\
\kappa^+_{01} \circ_{\rm h} \kappa^+_{12} \,:\; 
[0,1] \times \quo{ Z_{01}\sqcup Z_{12}}{\iota_{Z_{01}}^+\sim \iota^-_{Z_{12}}}
& \;\to\;
X_{01}\circ^2_{\rm h} X_{12} , \quad
\bigl(s,z\in Z_{ij} \bigr) \;\mapsto\; \kappa^+_{ij}(s,z).
\end{align*}
For the boundary embeddings $\kappa^\pm_{01}  \circ_{\rm h}  \kappa^\pm_{12}$ to be well defined, we need to take account of the scaling factors in \eqref{2borcomp} for the two cobordisms $X_{ij}$ to achieve
\begin{align*}
\kappa_{01}^-(s,\io^+_{Y_{01}}(t,x)) 
= \iota^+_{01} (\delta^+_{Y_{01}} s, t, x)
&\sim 
\iota^-_{12} (\delta^-_{Y_{12}} s, t, x)
=\kappa_{12}^-(s,\io^-_{Y_{12}}(t,x)) 
,\\
\kappa_{01}^+(s,\io^+_{Z_{01}}) 
= \iota^+_{01} (1-\delta_{Z_{01}}^+ + \delta_{Z_{01}}^+ s,..)
&\sim 
\iota^-_{12} (1-\delta_{Z_{12}}^- + \delta_{Z_{12}}^- s,..)
=\kappa_{12}^+(s,\io^-_{Z_{12}}) .
\end{align*}
That is, we define the relation $\sim$ in the construction of the glued (d+2)-manifold
$X_{01}\circ^2_{\rm h} X_{12}:=(X_{01}\sqcup X_{12})/\!\sim$ by
$\io_{01}^+(s,t,z)\sim \io^-_{12}(\phi(s),t,z)$ for some diffeomorphism $\phi:[0,1]\to[0,1]$ with
$\phi(r)=\delta^-_{Y_{12}} r / \delta^+_{Y_{01}}$ for $0\leq r \leq \delta^+_{Y_{01}}$ and
$\phi(1-r)=1- \delta^-_{Z_{12}} r / \delta^+_{Z_{01}}$ for $0\leq r \leq \delta^+_{Z_{01}}$.
Such $\phi$ exists since all $\delta$-factors in \eqref{2borcomp} are less than $\frac 12$. 
Finally, one needs to check that different choices of $\phi$ yield equivalent tuples of (d+2)-manifolds and  embeddings, and thus the same 2-morphism.
\item 
Horizontal composition is compatible with identities since for composable cobordisms $Y_{12},Y_{23}$ both
$\id_{Y_{12}}\circ_{\rm h}^2 \id_{Y_{23}} = I_{\id_{Y_{12}}}\circ_{\rm h}^2 I_{\id_{Y_{23}}}$
and
$\id_{Y_{12}\circ_{\rm h}^1 Y_{23}} = I_{\id_{Y_{12}\circ_{\rm h}^1 Y_{23}}}$
are represented by the (d+2)-manifold 
$$
\quo{( [0,1]\times Y_{12})\sqcup  ([0,1]\times Y_{23})}{\id_{[0,1]}\times\iota_{12}^+\, \sim \, \id_{[0,1]}\times\iota^-_{23}}
\;\simeq\; 
[0,1]\times \quo{Y_{12}\sqcup Y_{23}}{\iota_{12}^+\sim \iota^-_{23}}
$$
with embeddings -- arising from a universal choice of $\delta$ -- given by 
\begin{align*}
\iota^- &: \; (s,t,x) \mapsto \bigl(s,\io^-_{Y_{12}}(t,x) \bigr) =  \bigl(s,\io^-_{Y_{12}\circ_{\rm h}^1 Y_{23}}(t,x) \bigr) 
, \\
\iota^+ &: \; (s,t,x) \mapsto (s,\io^+_{Y_{23}}(t,x)) =  \bigl(s,\io^+_{Y_{12}\circ_{\rm h}^1 Y_{23}}(t,x) \bigr)
, \\
\kappa^-_\delta &: \; 
(s,y)\mapsto  \bigl( \delta\id_{[0,1]}\times \id_{Y_{12}} \bigr) \circ_{\rm h} \bigl( \delta\id_{[0,1]}\times \id_{Y_{23}}
\bigr) (s,y) =  \bigl(\delta s, \id_{Y_{12}\circ_{\rm h}^1 Y_{23}}(y)\bigr)
, \\ 
\kappa^+_\delta &: \; 
(s,z)\mapsto \bigl(1-\delta+\delta s, \bigl( \id_{Y_{12}} \circ_{\rm h} \id_{Y_{23}}\bigr) (z) \bigr)
 =  \bigl(1-\delta+\delta s, \id_{Y_{12}\circ_{\rm h}^1 Y_{23}}(z)\bigr) .
\end{align*}
\item 
Compatibility of horizontal 2-composition with vertical composition requires
$$
\bigl( [W_{01}] \circ^2_{\rm h} [W_{12}] \bigr) \circ_{\rm v} \bigl( [X_{01}] \circ^2_{\rm h} [X_{12}] \bigr)
=
\bigl( [W_{01}] \circ_{\rm v} [X_{01}] \bigr) \circ^2_{\rm h} \bigl( [W_{12}] \circ_{\rm v} [X_{12}]\bigr)
$$
for any
$\bigl[\bigl(W_{ij},\io^\pm_{W_{ij}},\kappa^\pm_{W_{ij}}\bigr)\bigr] \in \Mor^2_{\Bor_{d+1+1}}((V_{ij},\iota^\pm_{V_{ij}}),(Y_{ij},\iota^\pm_{Y_{ij}}) )$ and
$\bigl[\bigl(X_{ij},\io^\pm_{X_{ij}},\kappa^\pm_{X_{ij}}\bigr)\bigr] \in \Mor^2_{\Bor_{d+1+1}}((Y_{ij},\iota^\pm_{Y_{ij}}),(Z_{ij},\iota^\pm_{Z_{ij}}) )$, 
which form two pairs of equivalence classes  (d+2)-co\-bor\-disms of cobordisms for $ij=01, 12$ between
(d+1)-cobordisms $(V_{ij},\iota^\pm_{V_{ij}})$, $(Y_{ij},\iota^\pm_{Y_{ij}})$, $(Z_{ij},\iota^\pm_{Z_{ij}}) \in  \Mor^1_{\Bor_{d+1+1}}(\Sigma_i,\Sigma_j)$
for fixed surfaces $\Sigma_0,\Sigma_1,\Sigma_2\in\Obj_{\Bor_{d+1+1}}$.

Here both (d+2)-manifolds are of the form 
$\bigl(W_{01} \sqcup W_{12}\sqcup X_{01}\sqcup X_{12}\bigr)/\!\sim$, 
where in the first gluing, the equivalence relation $\sim$ is generated by 
$$
\iota_{W_{01}}^+\sim \iota^-_{W_{12}}\circ\phi_W,
\qquad
\iota_{X_{01}}^+\sim \iota^-_{X_{12}}\circ\phi_X,
\qquad
\kappa_{W_{01}}^+ \circ_{\rm h} \kappa_{W_{12}}^+\sim \kappa^-_{X_{01}} \circ_{\rm h} \kappa_{X_{12}}^- ,
$$
whereas in the second gluing, the equivalence relation $\sim$ is generated by 
$$
\kappa_{W_{01}}^+\sim \kappa^-_{X_{01}} ,
\qquad
\kappa_{W_{12}}^+\sim \kappa^-_{X_{12}} ,
\qquad
\iota_{W_{01}}^+\circ_{\rm h} \iota_{X_{01}}^+\sim \bigl(\iota^-_{W_{12}} \circ_{\rm h} \iota^-_{X_{12}} \bigr) \circ\phi_{WX}.
$$
This amounts to the same relation if we choose the diffeomorphism $\phi_{WX}$ of $[0,1]=[0,1]\circ_{\rm h}[0,1]$ as the gluing of $\phi_W$ and $\phi_X$.
The various embeddings are identified analogously.
\item 
Horizontal 1-composition is strictly associative as in $\Bor_{d+1+\eps}$.
\item
Horizontal 2-composition is associative since for composable 1-morphisms $Y_{ij}$ and $Z_{ij}$ and 
$\bigl[\bigl(X_{ij},\ldots\bigr] \in \Mor^2_{\Bor_{d+1+1}}(Y_{ij},Z_{ij})$ both 
$X_{12}\circ^2_{\rm h}( X_{23}\circ^2_{\rm h}X_{34})$
and $(X_{12}\circ^2_{\rm h}X_{23}) \circ^2_{\rm h}X_{34}$
are given by the same gluing of $(d+2)$-manifolds
$\bigl(X_{01} \sqcup X_{12}\sqcup X_{34}\bigr)/\!(\iota_{X_{01}}^+\sim \iota^-_{X_{12}}, \iota_{X_{12}}^+\sim \iota^-_{X_{23}})$ and the corresponding tubular neighbourhood embeddings. 
\item
Horizontal 1-composition is unital up to 2-isomorphism as in $\Bor_{d+1+\eps}$, that is for any surface 
$\Sigma\in \Obj_{\Bor_{d+1+1}}$ and $0<\delta< \frac 12$ the cylindrical cobordism $1_{\Sigma,\delta}\in\Mor_{\Bor_{d+1+1}}^1(\Sigma,\Sigma)$ is a weak unit. 
To prove the latter we start by proving equivalence 
$(Y,\io^\pm_Y) \circ^1_{\rm h} 1_{\Sigma,\delta}  \sim  (Y,\io^\pm_Y)$ in $\Mor^1_{\Bor_{d+1+1}}(\Sigma_0,\Sigma)$ for any $(Y,\io^\pm_Y) \in\Mor_{\Bor_{d+1+\eps}}^1(\Sigma_0,\Sigma)$. 
For that purpose we can use the diffeomorphism $\Psi$ constructed in Example~\ref{ex:1epsbor} to obtain 2-morphisms $I_\Psi, I_{\Psi^{-1}} \in \Mor^2_{\Bor_{d+1+1}}$ (represented by arrows below) whose vertical $\circ_{\rm v}$ compositions are the identities
\begin{align*}
(Y,\io^\pm_Y) 
\;\overset{I_{\Psi^{-1}}}{\longrightarrow}\;
(Y,\io^\pm_Y)  \circ^1_{\rm h} 1_{\Sigma,\delta} 
\;\overset{I_\Psi}{\longrightarrow} \;
(Y,\io^\pm_Y)  
\quad &= \quad
(Y,\io^\pm_Y) 
\; \overset{\id_{(Y,\iota^\pm)}}{\longrightarrow}\;
(Y,\io^\pm_Y)  , \\
(Y,\io^\pm_Y)  \circ^1_{\rm h} 1_{\Sigma,\delta} 
\;\overset{I_\Psi}{\longrightarrow} \;
(Y,\io^\pm_Y)  
\;\overset{I_{\Psi^{-1}}}{\longrightarrow}\;
(Y,\io^\pm_Y)  \circ^1_{\rm h} 1_{\Sigma,\delta} 
\quad &= \quad
\id_{(Y,\io^\pm_Y)  \circ^1_{\rm h} 1_{\Sigma,\delta} } .
\end{align*}
Indeed, we have 
$I_{\Psi^{-1}} \circ_{\rm v}  I_\Psi = I_{\Psi\circ\Psi^{-1}} = I_{\id_Y} = \id_{(Y,\iota^\pm)}$
due to \eqref{eq:Ipsiphi}, 
and similarly
$I_{\Psi} \circ_{\rm v}  I_{\Psi^{-1}} = I_{\Psi^{-1}\circ\Psi} = I_{\id_{(Y\cup [0,1]\times\Sigma)/\!\sim}} =
\id_{(Y,\io^\pm_Y)  \circ^1_{\rm h} 1_{\Sigma,\delta} }$.
This proves the claimed equivalence for any $Y \in\Mor_{\Bor_{d+1+1}}^1(\Sigma_0,\Sigma)$, and the other required equivalences $1_{\Sigma,\delta}  \circ^1_{\rm h} (Z,\io^\pm_Z) \sim  (Z,\io^\pm_Z)$ for $(Z,\io^\pm_Z) \in\Mor_{\Bor_{d+1+1}}^1(\Sigma,\Sigma_1)$ arise in the same way from the diffeomorphisms $\Phi: ([0,1]\times\Sigma \cup Z) /\! \sim \to Z$ constructed in Example~\ref{ex:1epsbor}.
\end{itemize} 
This finishes the construction of the bordism bicategory ${\rm Bor}_{d+1+1}$.
Moreover, the {\bf connected $\mathbf{d+1+1}$ bordism bicategory ${\rm \mathbf{Bor}^{\rm \mathbf{conn}}}_{\mathbf{d+1+1}}$} for $d\ge 2$ is constructed analogously, using the objects and representatives of morphisms of $\Bor^{\rm conn}_{d+1}$.
\end{example}

\subsection{Functors between bi- and 2-categories} \label{ss:funk}

The purpose of this section is to make sense of a notion of extending 2+1 Floer field theory to dimension $2+1+1=4$, which is the case $d=2$ of the following notion.

\begin{definition} \label{def:efft}
A {\bf (connected) d+1+1 Floer field theory} is a 2-functor $\Bor_{d+1+1}\to\Cat$ (resp.\ $\Bor_{d+1+1}^{\rm conn}\to\Cat$) that factorizes through a symplectic 2-category and preserves adjunctions.
\end{definition}

Here one should use the connected bordism bicategory $\Bor_{d+1+1}^{\rm conn}$ in order to fit the gauge theoretic examples from \S\ref{ss:ex}. An appropriate symplectic 2-category is constructed in \cite{ww:cat} and will be outlined in \S\ref{ss:symp2}. So it remains to spell out the functoriality requirements. We begin with 2-functors between 2-categories, and will develop the relevant notion for bicategories in Definition~\ref{def:bifunk}.

\begin{definition}\label{def:2funk}
A {\bf 2-functor} $\cF:\cC\to\cD$ between two 2-categories $\cC,\cD$ consists of
\begin{itemize}
\item
a map $\cF:\Obj_\cC\to \Obj_\cD$ between the sets of objects,
\item
functors $\cF_{x_1,x_2}:\Mor_\cC(x_1,x_2) \to \Mor_\cD(\cF(x_1),\cF(x_2))$ for each $x_1,x_2\in \Obj_\cC$,~i.e. 
\begin{itemize}
\item 
maps $\cF^1_{x_1,x_2}: \Mor^1_\cC(x_1,x_2) \to \Mor^1_\cD(\cF(x_1),\cF(x_2))$,
\item 
maps
$\cF^2_{x_1,x_2}: \Mor^2_\cC(f_{12},g_{12}) \to \Mor^2_\cD(\cF^1_{x_1,x_2}(f_{12}),\cF^1_{x_1,x_2}(g_{12}))$
for each pair $f_{12},g_{12}\in\Mor^1_\cC(x_1,x_2)$,
\item
compatibility with identities $\cF^2_{x_1,x_2}(\id_{f_{12}}) = \id_{\cF^1_{x_1,x_2}(f_{12})}$, 
\item
compatibility with vertical composition,
$$
\cF^2_{x_1,x_2}(f_{12}\circ_{\rm v} g_{12}) = \cF^2_{x_1,x_2}(f_{12}) \circ_{\rm v} \cF^2_{x_1,x_2}(g_{12}).
$$
\end{itemize}
\end{itemize}
These are required to intertwine the horizontal compositions in $\cC$ and $\cD$ as follows:
\begin{itemize}
\item 
$\cF$ is compatible with identities, $1_{\cF(x)} = \cF_{x,x}^1(1_x)$.
\item
$\cF$ is compatible with composition of 1-morphisms, i.e.\ for each $f_{ij}\in\Mor^1_\cC(x_i,x_j)$
$$
\cF^1_{x_1,x_3}(f_{12}\circ_{\rm h} f_{23}) = \cF^1_{x_1,x_2}(f_{12}) \circ_{\rm h} \cF^1_{x_2,x_3}(f_{23}).
$$ 
\item
$\cF$ is compatible with horizontal composition of 2-morphisms, i.e.\ for each tuple $f_{ij},g_{ij}\in\Mor^1_\cC(x_i,x_j)$ and 
$\alpha_{ij}\in\Mor^2_\cC(f_{ij},g_{ij})$
$$
\cF^2_{x_1,x_3}(\alpha_{12}\circ_{\rm h} \alpha_{23}) = \cF^2_{x_1,x_2}(\alpha_{12}) \circ_{\rm h} \cF^2_{x_2,x_3}(\alpha_{23}).
$$
\end{itemize}
\end{definition}

Before discussing the appropriate generalization of this notion to a 2-functor from a bicategory such as $\Bor_{2+1+1}$ to a 2-category such as $\Cat$ or $\Symp$, let us note that 2-categories such as $\Symp$ (with canonical base objects such as the symplectic manifold consisting of a point) come with natural 2-functors to $\Cat$. 
This reduces the construction of a 2+1+1 Floer field theory to the construction of a 2-functor $\Bor_{2+1+1}\to\Symp$, which can then be composed with the Yoneda functor $\Symp\to\Cat$ that is defined below and further discussed in Lemma~\ref{le:sympcat}.

\begin{lemma}\label{lem:Yoneda}
Let $\cC$ be a 2-category. Then any choice of distinguished object $x_0\in\Obj_\cC$ induces a 
{\bf Yoneda 2-functor} $\cY_{x_0} : \cC \to {\rm Cat}$ as follows.
\begin{itemize}
\item
To an object $x\in\Obj_\cC$ we associate the category $\cY_{x_0}(x):= \Mor_\cC(x_0,x)$. 
\item
To $f\in\Mor^1_\cC(x_1,x_2)$ we associate the functor $\cY_{x_0}(f):\cY_{x_0}(x_1) \to \cY_{x_0}(x_2)$
given by horizontal composition
with $f$ and its identity 2-morphism $\id_f\in \Mor^2_\cC(f,f)$, 
\begin{align*}
\Obj_{\cY_{x_0}(x_1)}=\Mor^1_\cC(x_0,x_1) &\,\longrightarrow\; \Mor^1_\cC(x_0,x_2) = \Obj_{\cY_{x_0}(x_2)},  \\
f_{01} &\,\longmapsto\; f_{01}\circ_{\rm h} f  ;\\
\Mor_{\cY_{x_0}(x_1)} \supset \Mor^2_\cC(f_{01},g_{01}) &\,\longrightarrow\; \Mor^2_\cC(f_{01}\circ_{\rm h} f ,
g_{01}\circ_{\rm h} f) \subset \Mor_{\cY_{x_0}(x_2)} , \\
\alpha &\,\longmapsto\; \alpha \circ_{\rm h} \id_f .
\end{align*}
\item
To a 2-morphism $\beta\in \Mor^2_\cC(g_{12},h_{12})$ between $g_{12},h_{12}\in \Mor^1_\cC(x_1,x_2)$ we associate the natural transformation $\cY_{x_0}(\beta): \cY_{x_0}(g_{12}) \Rightarrow \cY_{x_0}(h_{12})$
which takes each $f_{01}\in\Obj_{\cY_{x_0}(x_1)}=\Mor^1_\cC(x_0,x_1)$ to $\id_{f_{01}}\circ_{\rm h} \beta \in\Mor^2_\cC(f_{01}\circ_{\rm h}g_{12},f_{01}\circ_{\rm h}h_{12} )\subset \Mor_{\cY_{x_0}(x_2)}$. 
\end{itemize}
\end{lemma}
\begin{proof}
$\cY_{x_0}(x)$ is a category and $\cY_{x_0}(f)$ is a functor by Definition~\ref{def:2cat} of a 2-category.
$\cY_{x_0}(\beta)$ is a natural transformation since the required diagram for $\alpha\in\Mor_\cC(f_{01},f'_{01})$ commutes by compatibility of horizontal and vertical composition, 
\begin{align*}
\bigl(\alpha\circ_{\rm h} \id_{g_{12}} \bigr) \circ_{\rm v} \bigl(\id_{f'_{01}} \circ_{\rm h} \beta \bigr)
&= \bigl(\alpha\circ_{\rm v} \id_{f'_{01}}  \bigr) \circ_{\rm h} \bigl(\id_{g_{12}}  \circ_{\rm v} \beta \bigr) 
= \alpha \circ_{\rm h} \beta \\
&=  \bigl(\id_{f_{01}} \circ_{\rm v} \alpha\bigr) \circ_{\rm h} \bigl(  \beta \circ_{\rm v} \id_{h_{12}} \bigr) 
= \bigl(\id_{f_{01}} \circ_{\rm h} \beta \bigr) \circ_{\rm v} \bigl(\alpha\circ_{\rm h} \id_{h_{12}} \bigr).
\end{align*}
Next, we need to check that $\cF:=\cY_{x_0}:\Mor_\cC(x_1,x_2) \to \Mor_{\Cat}(\cF(x_1),\cF(x_2))$ is a functor for each $x_1,x_2\in \Obj_\cC$.
It is compatible with identities since both $\cF(\id_{f_{12}})$ and $\id_{\cF(f_{12})}$ are the natural transformation 
$\cF(f_{12}) \Rightarrow \cF(f_{12})$ which takes $f_{01}\in\Mor^1_\cC(x_0,x_1)$ to 
$\id_{f_{01}}\circ_{\rm h} \id_{f_{12}} = \id_{f_{01}\circ_{\rm h}f_{12}}$.
It is compatible with vertical composition since for $\alpha_{12},\beta_{12}\in \Mor^2_\cC(g_{12},h_{12})$
both $\cF(\alpha_{12}\circ_{\rm v} \beta_{12})$ and $\cF(\alpha_{12}) \circ_{\rm v} \cF(\beta_{12})$ are the natural transformation $\cF(g_{12}) \Rightarrow \cF(h_{12})$
which takes $f_{01}\in\Mor^1_\cC(x_0,x_1)$ to 
$$
\id_{f_{01}}\circ_{\rm h} (\alpha_{12}\circ_{\rm v} \beta_{12}) 
= 
( \id_{f_{01}}\circ_{\rm h} \alpha_{12} ) \circ_{\rm v} ( \id_{f_{01}}\circ_{\rm h}\beta_{12}) .
$$ 
Finally, we check compatibility with the horizontal composition.
\begin{itemize}
\item 
Both $1_{\cF(x)}$ and $\cF(1_x)$ are the functor $\Mor_\cC(x_0,x) \to \Mor_\cC(x_0,x)$ given by 
$f_{01} \mapsto f_{01} = f_{01}\circ_{\rm h} 1_x$ and
$\alpha \mapsto \alpha = \alpha \circ_{\rm h} \id_{1_x}$.
\item
For $f_{ij}\in\Mor^1_\cC(x_i,x_j)$ both 
$\cF(f_{12}\circ_{\rm h} f_{23})$ and $\cF^1_{x_1,x_2}(f_{12}) \circ_{\rm h} \cF^1_{x_2,x_3}(f_{23})$
are the functor $\Mor_\cC(x_0,x_1) \to \Mor_\cC(x_0,x_2)$ given by 
$f_{01} \mapsto f_{01}\circ_{\rm h} (f_{12}\circ_{\rm h} f_{23}) =  (f_{01}\circ_{\rm h} f_{12} )\circ_{\rm h} f_{23}$
and 
$\alpha \mapsto \alpha \circ_{\rm h} \id_{f_{12}\circ_{\rm h} f_{23}}
= (\alpha \circ_{\rm h} \id_{f_{12}} ) \circ_{\rm h}  \id_{f_{23}}$.
\item
For each tuple $g_{ij},h_{ij}\in\Mor^1_\cC(x_i,x_j)$ and $\alpha_{ij}\in\Mor^2_\cC(g_{ij},h_{ij})$, 
both $\cF(\alpha_{12}) \circ_{\rm h} \cF(\alpha_{23})$ and 
$\cF(\alpha_{12}\circ_{\rm h} \alpha_{23})$ are the
natural transformation $\cG:=\cF(g_{12}\circ_{\rm h} g_{23}) \Rightarrow \cH:=\cF(h_{12}\circ_{\rm h} h_{23})$
which takes $f_{01}\in\Mor^1_\cC(x_0,x_1)$ to 
%
%
$\id_{f_{01}}\circ_{\rm h} (\alpha_{12}\circ_{\rm h} \alpha_{23})$.
\end{itemize}
\vskip-5mm
\end{proof}

\begin{remark}[Yoneda 2-functor for bicategories]\label{rmk:bifunk} \rm 
If $\cC$ is a bicategory, then the Yoneda construction in Lemma~\ref{lem:Yoneda} still yields
categories $\cY_{x_0}(x)= \Mor_\cC(x_0,x)$, 
functors $\cY_{x_0}(f)$ given by horizontal composition with $f$ and $\id_f$, 
and natural transformations $\cY_{x_0}(\beta)$ given by $f_{01}\mapsto  \id_{f_{01}}\circ_{\rm h} \beta$, 
in such a way that $\cY_{x_0}:\Mor_\cC(x_1,x_2) \to \Mor_{\Cat}(\cF(x_1),\cF(x_2))$ is a functor.
However, $\cY_{x_0}$ is compatible with horizontal composition only up to isomorphisms in $\Cat$ since unitality $f_{01}\circ_{\rm h} 1_x \sim f_{01}$ and associativity $f_{01}\circ_{\rm h} (f_{12}\circ_{\rm h} f_{23}) \sim (f_{01}\circ_{\rm h} f_{12} )\circ_{\rm h} f_{23}$ only hold up to 2-isomorphism in $\cC$.
Thus $\cY_{x_0}:\cC\to\Cat$ can still be viewed as a 2-functor between bicategories in the sense of Definition~\ref{def:bifunk} below.
\end{remark}

To make the Yoneda construction for bicategories as well as our notion of 2+1+1 Floer field theory in Definition~\ref{def:efft} precise, we define the notion of a 2-functor between bicategories $\cC,\cD$ by weakening Definition~\ref{def:2funk} to allow compatibility with the horizontal composition up to isomorphisms. 

\begin{definition}\label{def:bifunk}
A {\bf 2-functor} $\cF:\cC\to\cD$ between two bicategories $\cC,\cD$ consists of
\begin{itemize}
\item
a map $\cF:\Obj_\cC\to \Obj_\cD$ between the sets of objects,
\item
functors $\cF_{x_1,x_2}:\Mor_\cC(x_1,x_2) \to \Mor_\cD(\cF(x_1),\cF(x_2))$ for each $x_1,x_2\in \Obj_\cC$,
\end{itemize}
which are compatible with the horizontal composition in the following sense:
\begin{itemize}
\item 
$1_{\cF(x)} \sim \cF_{x,x}^1(1_x)$ are equivalent 1-morphisms in $\cD$ for any choice of weak units associated to $x\in\Obj_\cC$ and $\cF(x)\in\Obj_\cD$.
\item
$\cF^1_{x_1,x_3}(f_{12}\circ_{\rm h} f_{23}) \sim \cF^1_{x_1,x_2}(f_{12}) \circ_{\rm h} \cF^1_{x_2,x_3}(f_{23})$ 
are equivalent 1-morphisms in $\cD$.
\item
For each tuple $f_{ij},g_{ij}\in\Mor^1_\cC(x_i,x_j)$ and 
$\alpha_{ij}\in\Mor^2_\cC(f_{ij},g_{ij})$ we have
$$
\cF^2_{x_1,x_3}(\alpha_{12}\circ_{\rm h} \alpha_{23}) = \cF^2_{x_1,x_2}(\alpha_{12}) \circ_{\rm h} \cF^2_{x_2,x_3}(\alpha_{23}).
$$
\end{itemize}
\end{definition}

\subsection{Adjunctions, quilt diagrams, and quilted bicategories} \label{ss:quilt}

This section will generalize the notion of string diagrams, which are graphical representations of the structure and axioms of 2-categories, as surveyed in e.g.\  \cite[\S1.1]{willerton}, \cite{youtube}, and \cite{W:slides} in the example of topological and symplectic 2-categories. Then we introduce a notion of quilted bicategory, in which not only string diagrams but the more general quilt diagrams define 2-morphisms, and show how bordism bicategories naturally fit into this notion.

\begin{figure}[!h]
\centering
\includegraphics[width=5.5in]{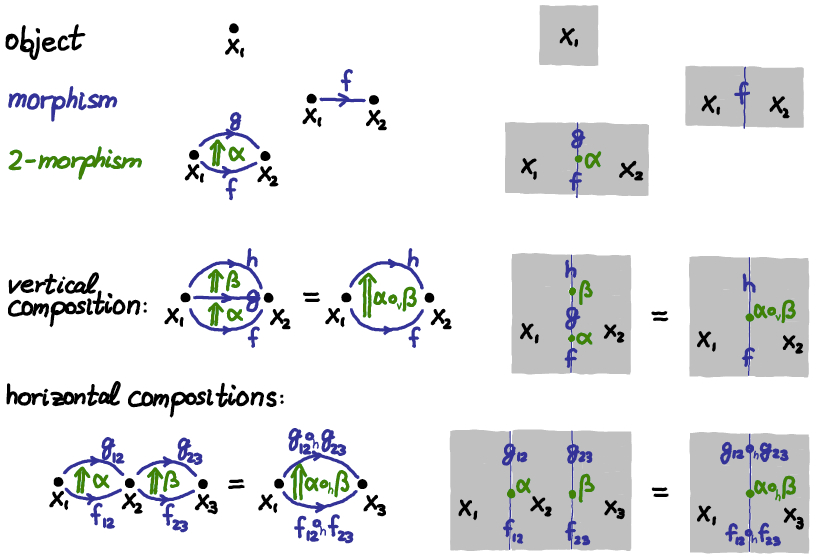}
\caption{The structures of a 2-category or bicategory in string diagram notation.}
\label{fig:2catstructure}
\end{figure}

\begin{remark}[String diagrams] \rm 
Roughly speaking, a string diagram in a bicategory $\cC$ consists of vertical lines drawn in the plane, punctures on the line, and labels in $\cC$. These in turn represent the structure of the bicategory as indicated in Figure~\ref{fig:2catstructure}.
More precisely, the lines separate the plane into connected components, called ``patches'', and the punctures separate the lines into connected components, called ``seams'', each of which lies in the intersection of the closures of exactly two patches; see the left side of Figure~\ref{fig:string2mor} for illustration. The patches / seams / punctures of the diagram are labeled with objects / 1-morphisms / 2-morphisms in $\cC$ in a coherent manner: a seam is labeled by a 1-morphism between the objects associated to the two adjacent patches, and a puncture is labeled by a 2-morphism between the 1-morphisms associated to the two adjacent seams.
Now any such string diagram can be translated into horizontal and vertical compositions of the involved 2-morphisms, and defines a new 2-morphism between the 1-morphisms obtained from composing the labels of the seams running to $+\infty$ resp.\ $-\infty$. Here we read from left to right and from bottom to top, with different choices of order of composition yielding the same result due to associativity and compatibility of horizontal and vertical composition; see the right of Figure~\ref{fig:string2mor} for examples.

\begin{figure}[!h]
\centering
\includegraphics[width=5.5in]{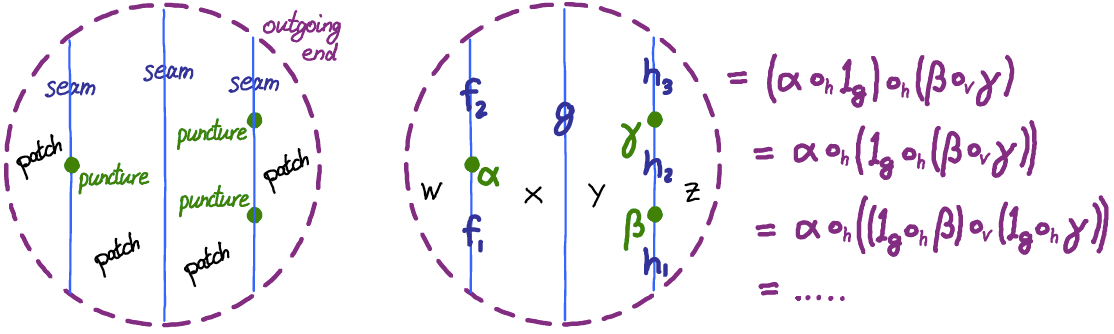}
\caption{String diagrams represent well defined 2-morphisms given by iterated horizontal and vertical compositions applied to 2-morphisms given by labels and identity 2-morphisms.}
\label{fig:string2mor}
\end{figure}

After compactifying the plane to a sphere, we may interpret the punctures in the plane as incoming ends -- at which the 2-morphisms are prescribed -- and the puncture at infinity as the outgoing end -- at which the resulting 2-morphism is read off.
The axioms of a 2-category or bicategory can then also be represented by string diagrams: identities between different diagrams, or the fact that diagrams have invariant meaning -- independent of the order in which composition is being read off.
See Figure~\ref{fig:2cataxioms} for a list of the 2-category axioms as string diagrams, 
\end{remark}

\begin{figure}[!h]
\centering
\includegraphics[width=5.5in]{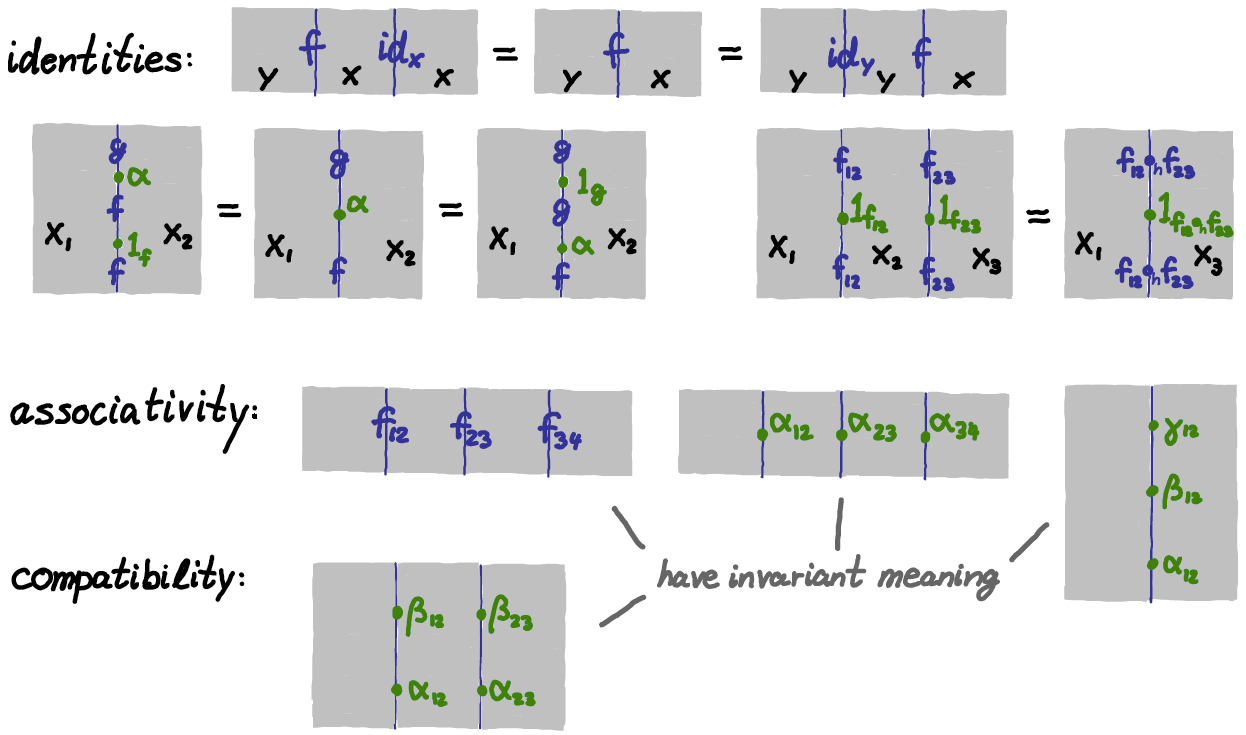}
\caption{2-category axioms in string diagram notation. Also see Figure~\ref{fig:2catcomp} for compatibility.}
\label{fig:2cataxioms}
\end{figure}

\begin{figure}[!h]
\centering
\includegraphics[width=5.5in]{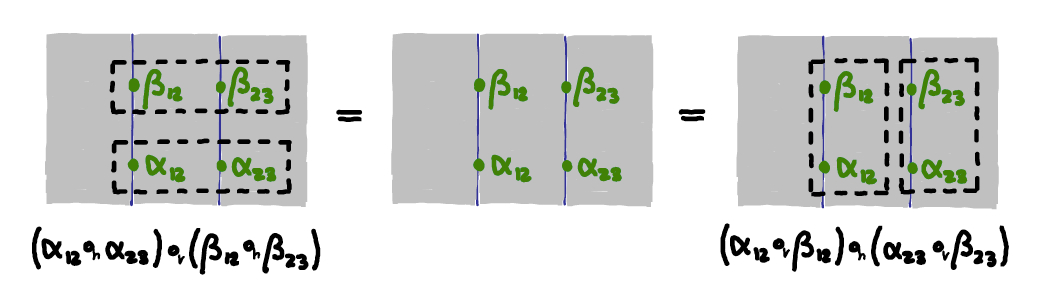}
\caption{Compatibility between horizontal and vertical composition in string diagram notation.}
\label{fig:2catcomp}
\end{figure}

The symplectic 2-category will have string diagrams -- represented by pseudoholomorphic quilts \cite{ww:quilts} described in Remark~\ref{rmk:quilts} -- which lie on more general surfaces (not just the sphere), can have any number of seams running into the punctures, do not require a left/right or top/bottom orientation, but still have exactly one outgoing end and the same meaning as a string diagram:
If we prescribe Floer homology classes (the 2-morphisms) at each incoming end, then the diagram defines a Floer homology class at the outgoing end. These relative quilt invariants are applied to the basic string diagrams in \cite{ww:cat} to construct the symplectic 2-category, but they are defined in higher generality and satisfy algebraic identities arising from forgetting the vertical/horizontal structure of string diagrams.

The purpose of this section is to cast this additional structure on the symplectic 2-category into abstract terms -- giving rise to a notion akin to that of spherical 2-categories developed in \cite{spherical}, but expressing the algebraic properties in a graphical language rather than via monoidal structure. This is useful for a variety of reasons: First, this structure simply exists naturally, not just for the symplectic 2-category but also the bordism bicategories (see Lemma~\ref{lem:borquilt}) and other gauge theoretic categories that can be constructed via PDE's associated to quilt diagrams (see \S\ref{ss:gauge2}). Second, this structure can be expressed without reference to a monoidal structure, which is problematic both in the gauge theoretic and symplectic context (see Remark~\ref{rmk:conn}), and thus also leads us to work with connected bordism categories -- which lack the monoidal structure given by disjoint union. Third, quilt diagrams naturally appear in a generalization of Cerf decompositions from $\Bor_{d+1}$ to $\Bor_{d+1+1}$ which arise from the diagrams of Morse 2-functions in e.g.\ \cite{GayKirby}, as sketched in \cite{W:slides}. These ``quilted Cerf decompositions'' lie at the core of the extension principle for Floer field theories \cite{W:ext}, as outlined in Conjecture~\ref{con:fftext}.

In order to make sense of the labeling in a quilt diagram we will need some symmetry properties of the bicategory, which we will introduce before going into the actual notion of quilt diagram.
First, dropping the distinguished horizontal direction in string diagrams loses the ``from left to right'' designation which determines that a seam is to be labeled by a 1-morphism from the object associated to the left adjacent patch to the object associated to the right adjacent patch. Instead, we will define left/right based on a choice of orientation of each seam and label the two orientations of each seam with adjoint pairs of 1-morphisms.
For that purpose, the following makes the adjunction notion from Remark~\ref{rmk:adj} rigorous.

\begin{definition}\label{def:adj}
A {\bf 2-category with adjoints} is a 2-category $\cC$ as in Definition~\ref{def:2cat} together with an adjunction map $\Mor^1_\cC\to \Mor^1_\cC, Y\mapsto Y^T$ that associates to each $Y\in\Mor^1_\cC(\Sigma_0,\Sigma_1)$ its {\bf adjoint} $Y^T\in\Mor^1_\cC(\Sigma_1,\Sigma_0)$ and satisfies:
\begin{itemize}
\item
Adjunction is reflexive, i.e.\ $(Y^T)^T=Y$. 
\item
Adjoint morphisms are dual to each other in the sense that for $Y\in\Mor^1_\cC(\Sigma_0,\Sigma_1)$ there exist $X_Y\in\Mor_\cC^2(1_{\Sigma_0}, Y\circ^1_{\rm h} Y^T)$ and 
$X_Y^T\in\Mor_\cC^2( Y^T\circ^1_{\rm h} Y, 1_{\Sigma_1})$ 
satisfying identities that are illustrated in Figure~\ref{fig:dual}, 
\begin{equation}\label{eq:dual}
\bigl(X_Y\circ^2_{\rm h} \id_Y\bigr) \circ_{\rm v} \bigl(\id_Y \circ^2_{\rm h} X_Y^T\bigr)  = \id_Y, 
\quad
\bigl(\id_{Y^T}\circ^2_{\rm h}X_Y \bigr) \circ_{\rm v} \bigl( X_Y^T\circ^2_{\rm h} \id_{Y^T}\bigr)  = \id_{Y^T}.
\end{equation}
\end{itemize}
A {\bf bicategory with adjoints} is a bicategory $\cC$ as in Definition~\ref{def:bicategory} together with an adjunction map as above, whose duality property holds for all choices of weak identity morphisms $1_{\Sigma_0}, 1_{\Sigma_1}$.
\end{definition}

\begin{figure}[!h]
\centering
\includegraphics[width=5.5in]{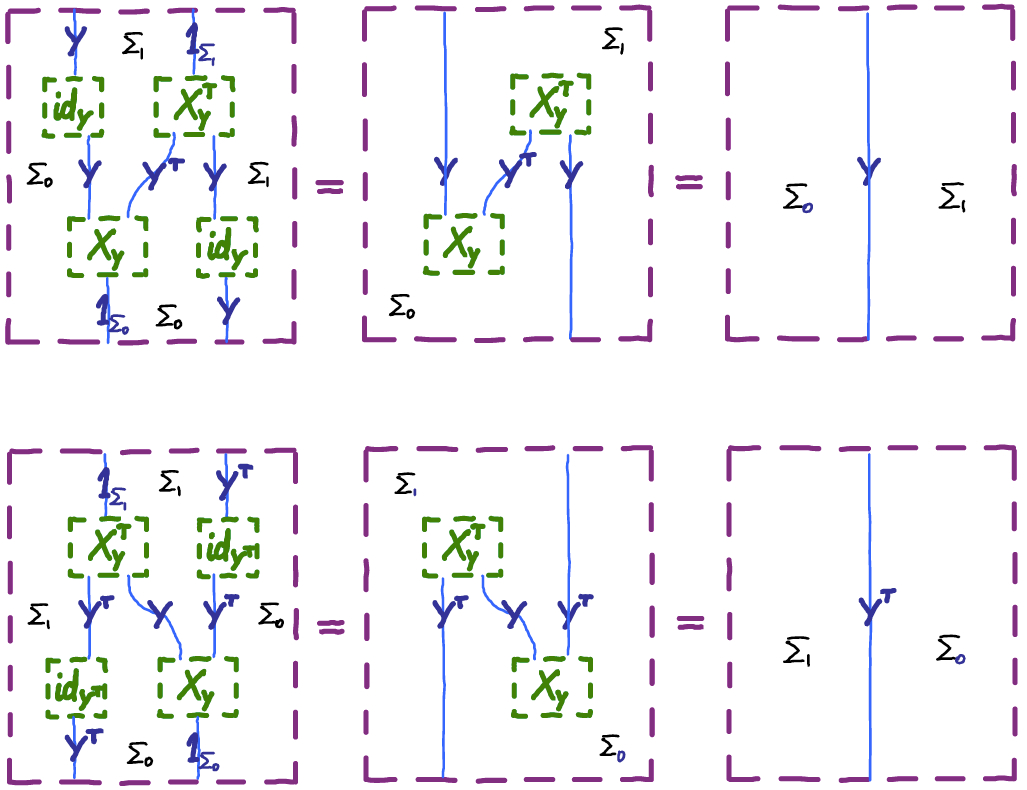}
\caption{The duality identities \eqref{eq:dual} can be represented by slightly generalized string diagrams.}
\label{fig:dual}
\end{figure}

\begin{figure}[!h]
\centering
\includegraphics[width=5.5in]{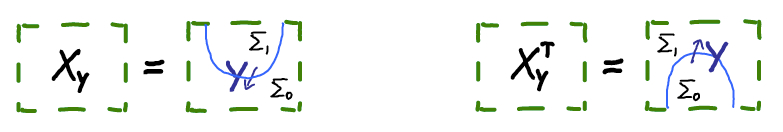}
\caption{In a quilted bicategory (see \S\ref{ss:quilt} below) the adjunction 2-morphisms arise from quilted structure maps that are represented by the above quilt (i.e.\ generalized string) diagrams.
}
\label{fig:quiltadjoints}
\end{figure}

\begin{remark} \label{rmk:adj2} \rm 
\begin{itemize}
\item[(a)]
Adjoints in $\Bor_{d+1+1}$ are obtained by orientation reversal of the 1-morphisms, as sketched for $\Bor_{d+1}$ in Remark~\ref{rmk:adj}.
More precisely, the adjoint of a $(d+1)$-cobordism $(Y, \iota^-_Y,\iota^+_Y) \in\Mor^1_{\Bor_{d+1+1}}(\Sigma_0,\Sigma_1)$ is the cobordism $(Y, \iota^-_Y,\iota^+_Y)^T:= (Y^-, \iota^+_Y\circ\rho_1,\iota^-_Y\circ\rho_0)\in\Mor^1_{\Bor_{d+1+1}}(\Sigma_1,\Sigma_0)$ obtained by reversing the orientation on $Y$, switching the tubular neighbourhood embeddings, and precomposing each with the orientation reversing diffeomorphism $\rho_i(t,z)=(1-t,z)$ of $[0,1]\times\Sigma_i$.

With this reflexive operation established, the adjunction 2-morphisms $X_Y$ and $X_Y^T$ that are required for the duality in Definition~\ref{def:adj} can be constructed from the further generalized string diagrams indicated in Figure~\ref{fig:quiltadjoints}. For example,  $X_Y$ is obtained from a half disk times $\Sigma_1$, a square minus a half disk times $\Sigma_0$, and an interval times $Y$, glued along matching boundary components. 
This and the analogous construction for $X_Y^T$ yields the required 2-morphisms, which satisfy \eqref{eq:dual} since gluing them into the string diagrams in Figure~\ref{fig:dual} yields 4-manifolds with boundary and corners that are diffeomorphic relative to the boundary.

\item[(b)] 
In the symplectic category $\Symp$, the adjoint of a Lagrangian $L\subset M_0^-\times M_1$ is $L^T:=\tau(L)\subset M_1^-\times M_0$ obtained by transposition $\tau(p_0,p_1):=(p_1,p_0)$, and the adjoint of a general 1-morphism 
$\uL=(L_{01},\ldots,L_{(k-1)k})$ is $\uL^T=(L_{(k-1)k}^T,\ldots,L_{01}^T)$.
Again, the adjunction 2-morphisms $X_Y$ and $X_Y^T$ can be obtained from the fact that the generalized string diagrams in Figure~\ref{fig:quiltadjoints} have invariant meaning; see Remark~\ref{rmk:sympquilt}.
\end{itemize}
\end{remark}

A second symmetry property of a bicategory that is required to formalize quilt diagrams comes from the fact that dropping the distinguished vertical direction in string diagrams loses the ``from bottom to top'' designation which determines that a puncture is to be labeled by a 2-morphism from the 1-morphism associated to the bottom adjacent seam to the 1-morphism associated to the top adjacent patch. 
Instead, we are allowing any number of seams to intersect in a puncture of a quilt diagram, and will associate to these seams -- with counterclockwise order induced from an overall orientation of the diagram --  a cyclic 2-morphism space, from which the label for this puncture will be chosen.
This is based on the following cyclic symmetry of the 2-morphisms in a bicategory with adjoints. Here and in the following, we will use $\Z_N:=\qu{\Z}{N\Z}$ to index cyclically ordered sets of $N$ elements with no distinguished first element.

\begin{remark}\label{rmk:cyclic} 
A {\bf cyclic 1-morphism} in a bicategory $\cC$ is a cyclic sequence of 1-morphisms $\underline f=(f_i)_{i\in\Z_N} : \Z_N\to \Mor^1_\cC$ that is composable in the sense that we have $f_i\in\Mor^1_\cC(x_i,x_{i+1})$ for a cyclic sequence of objects $\underline x=(x_i)_{i\in\Z_N}:\Z_N\to \Obj_\cC$.
This implies that the compositions $f_i\circ f_{i+1}\circ\ldots\circ f_{i+k}\in\Mor^1_\cC(x_i,x_{i+k})$ are well defined for every $i\in\Z_N$, $k\in\N$, and in particular 
$f_i\circ f_{i+1}\circ\ldots\circ f_{i+N-1}\in\Mor^1_\cC(x_i,x_{i+N}=x_i)$.

If the bicategory $\cC$ moreover has adjoints in the sense of Definition~\ref{def:adj}, then we can associate to every cyclic 1-morphism $\underline f=(f_i)_{i\in\Z_N}$ a {\bf cyclic 2-morphism space}  
$$
\Mor^2_\cC( \underline f) :=  \Mor^2_\cC(f_i^T , f_{i+1}\circ\ldots\circ f_{i+N-1}) ,
$$
which is independent of the choice of $i\in\Z_N$ and can also be identified with the 2-morphism space
$ \Mor^2_\cC\bigl( (f_i \circ\ldots\circ f_j)^T , f_{j+1}\circ\ldots\circ f_{i-1} \bigr)$
for other partitions of the cyclic 1-morphism.
%
\end{remark}

As a tangential note -- useful for identifying different field theories as in the Atiyah-Floer type conjectures -- 
the following remark explains an algebraic method for localizing proofs of isomorphisms between cyclic 1-morphisms or their associated cyclic 2-morphism spaces. 

\begin{remark}[A ``local to global'' principle for cyclic 1-morphisms] \rm \label{rmk:localtoglobal}
In a 2-category with adjoints, any ``local'' isomorphism between $f,g\in\Mor^1_\cC(x_i,x_{i+1})$ implies ``global isomorphisms'' between any cyclic 1-morphisms that differ by replacing $f$ with $g$, 
$$
f\sim g \quad\Longrightarrow\quad 
( \ldots f_{i-1}, f_i=f , f_{i+1} \ldots ) \sim ( \ldots f_{i-1}, f_i=g , f_{i+1} \ldots ).
$$ 
Here the local isomorphism is given by an invertible 2-morphism $\alpha\in\Mor^2_\cC(f,g)$, i.e.\ 
$\alpha\circ_{\rm v} \alpha^{-1} = 1_f$ and $\alpha^{-1}\circ_{\rm v} \alpha = 1_g$ for some 
$\alpha^{-1}\in\Mor^2_\cC(g,f)$.
It induces global isomorphisms 
$\underline f_j := (f_{j}, \ldots , f_i=f , \ldots , f_{j-1}) \sim (f_{j}, \ldots , f_i=g , \ldots , f_{j-1}) =: \underline g_j$  in $\Mor^1(x_j,x_j)$ for any $j\in\Z_N$, 
in the sense that there exist 2-morphisms given by 
\begin{align*}
\underline \alpha &:= \id_{f_j} \circ_{\rm h} \ldots \id_{f_{i-1}} \circ_{\rm h} \alpha \circ_{\rm h} \id_{f_{i+1}}\ldots  \circ_{\rm h} \id_{f_{j-1}} \;\in\; \Mor^2_\cC(\underline f_j,\underline g_j), \\
\underline \alpha^{-1} &:= \id_{f_j} \circ_{\rm h} \ldots \id_{f_{i-1}} \circ_{\rm h} \alpha^{-1} \circ_{\rm h} \id_{f_{i+1}}\ldots  \circ_{\rm h} 
\id_{f_{j-1}} \;\in\; \Mor^2_\cC(\underline g_j,\underline f_j),
\end{align*}
which satisfy $\underline \alpha \circ_{\rm v} \underline \alpha^{-1} = \id_{\underline f_j}$ and 
$\underline \alpha^{-1} \circ_{\rm v} \underline \alpha = \id_{\underline g_j}$. Indeed, the first (and similarly the second) follows from compatibility of horizontal and vertical composition with each other as well as identities, 
\begin{align*}
\underline \alpha \circ_{\rm v} \underline \alpha^{-1} 
&= 
( \id_{f_j}\circ_{\rm v} \id_{f_j} ) \circ_{\rm h} \ldots ( \id_{f_{i-1}}\circ_{\rm v} \id_{f_{i-1}} )  \circ_{\rm h} (\alpha \circ_{\rm v} \alpha^{-1}) \\
&\qquad\qquad\qquad\qquad\qquad\qquad
\circ_{\rm h} ( \id_{f_{i+1}}\circ_{\rm v} \id_{f_{i+1}} ) \ldots  \circ_{\rm h} ( \id_{f_{j-1}}\circ_{\rm v} \id_{f_{j-1}} ) \\
&= 
\id_{f_j} \circ_{\rm h} \ldots  \id_{f_{i-1}}  \circ_{\rm h} \id_f \circ_{\rm h} \id_{f_{i+1}} \ldots  \circ_{\rm h} \id_{f_{j-1}} \\
&= \id_{f_{j}  \circ_{\rm h} \ldots  f_{i-1}  \circ_{\rm h} f \circ_{\rm h} f_{i+1} \ldots  \circ_{\rm h} f_{j-1}} 
= \id_{\underline f_j}.
\end{align*}
Moreover the local isomorphism also implies an identification between the cyclic 2-morphism spaces,
$$
f\sim g \quad\Longrightarrow\quad 
\Mor^2_\cC( \ldots f_{i-1}, f_i=f , f_{i+1} \ldots ) \simeq \Mor^2_\cC( \ldots f_{i-1}, f_i=g , f_{i+1} \ldots ).
$$ 
\end{remark}

Finally, we introduce quilt diagrams by phrasing the notions of ``quilted surface'' and ``Lagrangian boundary conditions'' from \cite[\S3]{ww:quilts} in abstract terms.

\begin{definition}
\label{def:quilt}
A {\bf quilt} is a tuple $\cQ:=(q_0,Q_0,Q_1,Q_2)$ consisting of a closed oriented surface $Q_2$,
a finite subset of points $Q_0\subset Q_2$, a 1-dimensional submanifold $Q_1\subset Q_2\less Q_0$, and one distinguished point $q_0\in Q_0$. We moreover require 
$Q_1\subset Q_2\less Q_0$ to be a closed subset with finitely many connected components, as illustrated in Figure~\ref{fig:quiltsurf}.

\begin{figure}[!h]
\centering
\includegraphics[width=5.5in]{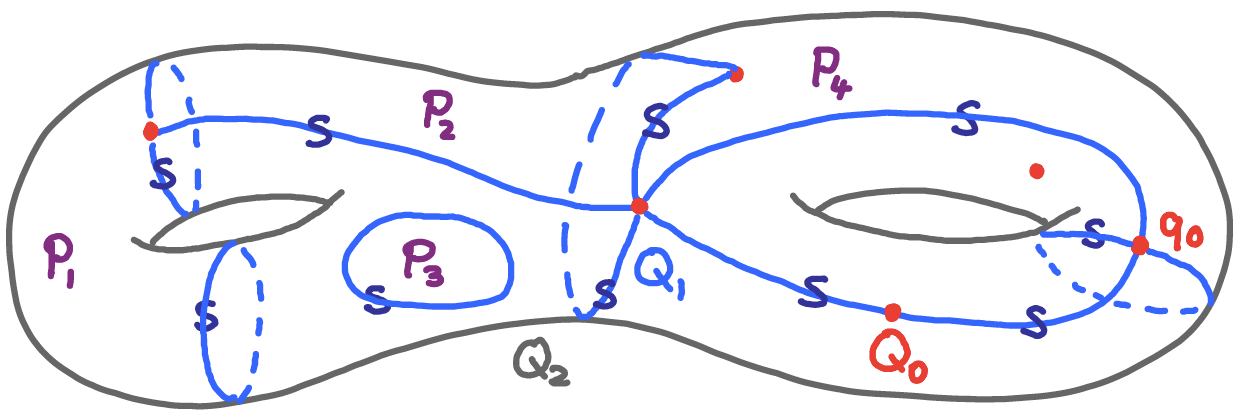}
\caption{A quilt (or quilted surface) is given by a closed surface $Q_2$ and submanifolds $Q_0,Q_1,q_0$. These specify patches (with arbitrary enumeration $P_i$ above) and seams (all labeled by $S$ above).}
\label{fig:quiltsurf}
\end{figure}

\begin{itemize}
\item
The {\bf patches} $P\in\cP_\cQ\cong \pi_0(Q_2\less (Q_0\cup Q_1))$ of $\cQ$ are the connected components $P\subset Q\less (E\cup S)$. 
\item
The {\bf seams} $S\in\cS_\cQ\cong \pi_0(Q_1)$ of $\cQ$ are the connected components $S\subset Q_1$.
The {\bf oriented seams} $S\in\cS^{\rm or}_\cQ\simeq \cS_\cQ\times\Z_2$ are pairs of seams with orientations.
\item
For  $S\in\cS^{\rm or}_\cQ$ we denote by $P^-_S,P^+_S\in\cP_\cQ$ the {\bf adjacent patches} whose oriented boundary contains $S^-$ resp.\ $S$ (i.e.\ which lie to the right resp.\ left of $S$), as illustrated in Figure~\ref{fig:quiltends}.
\item
The {\bf outgoing end} of $\cQ$ is $\cE^+_\cQ:=\{e^+\}:=\{q_0\}$, and the {\bf incoming ends} of $\cQ$ are the points $e\in\cE_\cQ^- :=Q_0\less\{q_0\}$.
\end{itemize}
\end{definition}

\begin{remark} \label{rmk:quilt} \rm 
Each seam is either a circle or an open interval embedded in $Q_2\less Q_0$, and cannot intersect itself or other seams by the submanifold property of $Q_1$. Moreover, the closure of $Q_1\subset Q_2\less Q_0$ implies that the boundary of an interval seam lies in $Q_0$, i.e.\ the seam is a closed interval immersed in $Q_2$ with endpoints mapping to ends in $Q_0$ which may or may not coincide.
We had to add the finiteness condition to avoid ``Hawaiian earrings'' -- sequences of interval seams converging to a puncture.

The finiteness condition for the sets of seams $\cS_\cQ \cong \pi_0(Q_1)$ and ends $\cE^+_\cQ \cup \cE^-_\cQ = Q_0$ also implies finiteness for the set of patches $\cP_\cQ\cong \pi_0(Q_2\less (Q_0\cup Q_1))$. Moreover, each patch is an open subset $P\subset Q_2\less Q_0$, whose boundary $\overline P \less P$ is given by a union of seams. The embedding $Q_0\less Q_2$ gives $\overline P$ the structure of an oriented 2-manifold with boundary, though some seams may lie in its interior -- namely the seams which have $P$ adjacent on both sides, i.e.\ the oriented seams $S\in\cS^{\rm or}_\cQ$ with $P^-_S=P^+_S=P$ (which is equivalent to $P^+_{S^-}=P^+_S=P$).
By cutting along these seams and adding two copies of the seam we obtain another oriented 2-manifold $\widehat P$ with boundary given by the union of all oriented seams $S\in\cS^{\rm or}_\cQ$ with $P^+_S=P$; see Figure~\ref{fig:quiltends} for examples. This ``refinement of the closure in $Q_2\less Q_0$ of each patch'' comes with a natural immersion $\widehat P\to Q_0 \less Q_2$ with image $\overline P$ and self-intersections on the seams in the interior of $\overline P$.
\end{remark}

\begin{figure}[!h]
\centering
\includegraphics[width=5.5in]{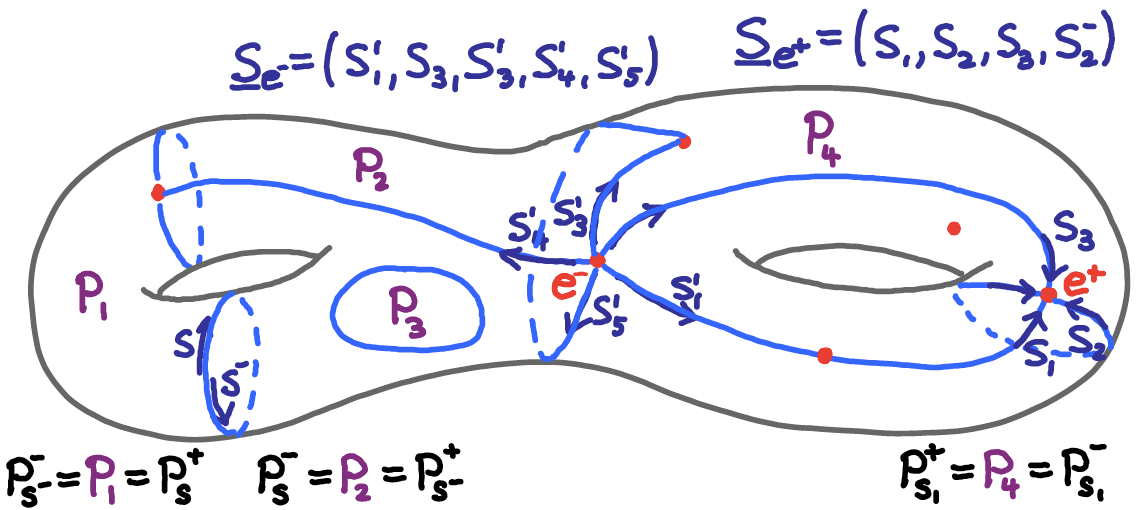}
\caption{A quilt with some examples of oriented seams and their adjacent patches, and cyclic sequences of oriented seams $\ul S_{e^\pm}$ associated to ends $e^\pm\in\cE^\pm_\cQ$. 
While the patch $P_3$ is an open disk, its refined closure $\widehat P_3$ simply is a closed disk. The case of $\widehat P_2$ -- a closed annulus minus one boundary puncture -- shows that these refined closures are usually not compact. 
Finally, $\widehat P_4$ is an example in which the immersion to $Q_2\less Q_0$ is not injective. Here $P_4$ is an open disk with one interior puncture, and its closure $\overline P_4\subset Q_2\less Q_0$ is the complement of a disk in a torus, minus 2 punctures on the boundary and 3 punctures in the interior. However, $\widehat P_4$ is a closed 10-gon minus the corners and one interior puncture, with oriented boundary components $S_2, S_1^-, {S_1'}^-, S_5', {S_3'}^-, S_3, S_2^-, S_3^-, S_1', S_1$.}
\label{fig:quiltends}
\end{figure}

\begin{definition}
\label{def:quiltdiagram}
A {\bf quilt diagram} $\cQ\cD=\bigl(\cQ,(\Sigma_P)_{P\in\cP_\cQ}, (Y_S)_{S\in\cS_\cQ} \bigr)$ in a bicategory $\cC$ with adjoints consists of a quilt $\cQ$ with labels in $\cC$ as follows, and illustrated in Figure~\ref{fig:quiltdiagram}.
\begin{itemize}
\item
Each patch $P\in\cP_\cQ$ is labeled by an object $\Sigma_P\in \Obj_\cC$.
\item
Each oriented seam $S\in\cS^{\rm or}_\cQ$ of $\cQ$ is labeled by a morphism $Y_S\in\Mor_\cC^1(P^-_S, P^+_S)$ such that seams of opposite orientation are labeled with adjoint morphisms, i.e.\ $Y_{S^-}=(Y_S)^T\in\Mor_\cC^1(P^+_S, P^-_S)$, since $P^\pm_{S^-}=P^\mp_S$.
\end{itemize}
\end{definition}

\begin{figure}[!h]
\centering
\includegraphics[width=5.5in]{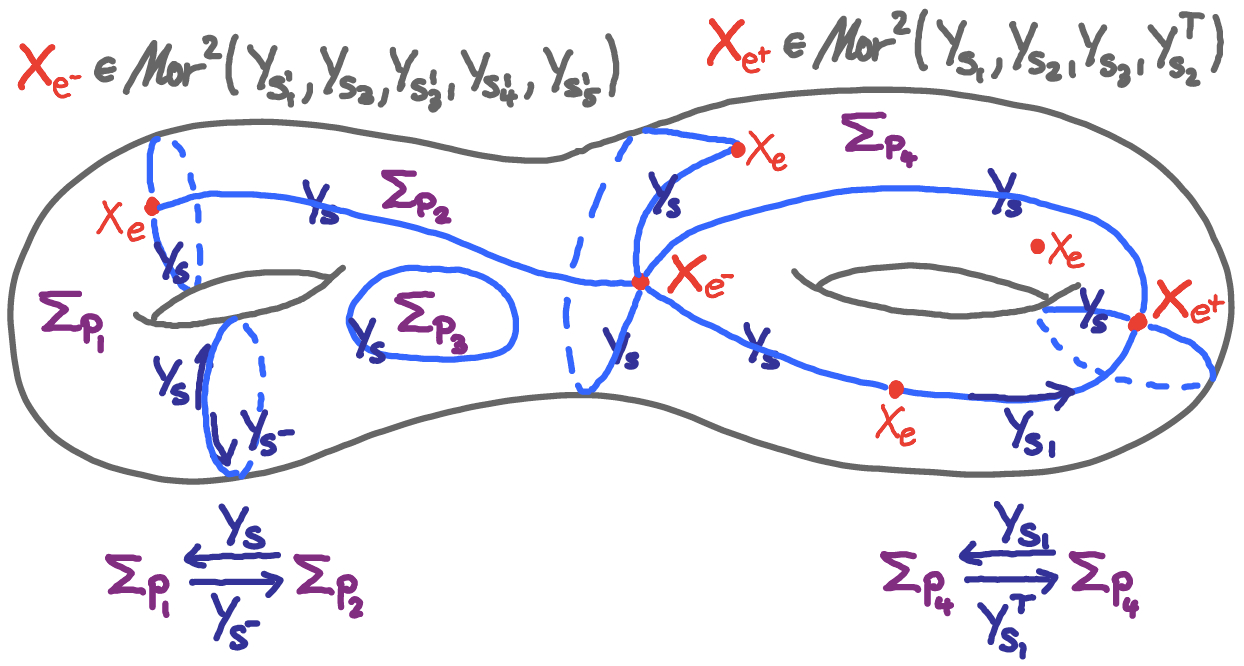}
\caption{In a quilt diagram, each patch $P_i$ is labeled by an object $\Sigma_{P_i}$ and each seam $S$ is labeled by a pair of adjoint 1-morphisms $Y_S,Y_S^T$ (corresponding to the different orientations of the seam). 
One could in addition label each end $e$ by a 2-morphism $X_e$ in the corresponding cyclic 2-morphism space, however, these will instead be viewed as inputs or outputs of a quilted composition map induced by the quilt diagram.}
\label{fig:quiltdiagram}
\end{figure}

To turn a quilt diagram into a generalized string diagram, we should in addition label each incoming end $e\in\cE^-_\cQ$ by a 2-morphism $X_e\in\Mor_\cC^2(\underline{Y}_e)$, and at the outgoing end $e^+=q_0$ have the quilt diagram define a 2-morphism $X_{e^+}\in\Mor_\cC^2(\underline{Y}_{e^+})$ in the {\bf cyclic 2-morphism spaces associated to each end} as follows:
\begin{itemize}
\item
For the outgoing end, we define a cyclic sequence of oriented seams $\underline{S}_{e^+}=(S_i)_{i\in\Z_{N_{e^+}}} : \Z_{N_{e^+}}\to\cS^{\rm or}_\cQ$ given by the oriented seams $S_i\simeq\R$ with $+\infty$-limit $e^+$, ordered by their intersection with a counterclockwise circle around $e^+=q_0\in Q_2$; see Figure~\ref{fig:quiltends} for an example.
Then $\underline{Y}_{e^+}:=\bigl( Y_{S_i} \bigr)_{i\in\Z_{N_{e^+}}} : \Z_{N_{e^+}}\to\Mor^1_\cC$ is a cyclic 1-morphism of $\cC$ in the sense of Remark~\ref{rmk:cyclic}, with a well defined cyclic 2-morphism space $\Mor_\cC^2(\underline{Y}_{e^+})$.
\item
For each incoming end $e$, we obtain the cyclic 1-morphism $\underline{Y}_e$ analogously from the oriented seams $S_i\simeq\R$ with $-\infty$-limit $e$; again see
Figure~\ref{fig:quiltends} for an example.
\item
If an incoming or outgoing end $e$ lies in the interior of a patch $P$, i.e.\ has no adjacent seams, then we associate to it the cyclic 1-morphism $\underline{Y}_e:=1_{\Sigma_P}$.
(In the case of a bicategory $\cC$ one should either disallow ends without seams or ensure identifications between the cyclic 2-morphism spaces associated to different choices of weak identity 1-morphisms.)
\end{itemize}
However, instead of fixing these labels, we will view the quilt diagrams as inducing maps between the cyclic 2-morphism spaces associated to the ends, as indicated in Figure~\ref{fig:quiltmap}. Another example of a quilt map is given in Figure~\ref{fig:quiltdiagramforHFiso}.
In particular, string diagrams already induce such maps via horizontal and vertical composition.
Now we define a quilted 2-category to be a 2-category in which not only the string diagrams but general quilt diagrams define maps on 2-morphism spaces. The analogous definition is made for bicategories.
Here one could make various further specifications such as fixing the genus of the quilt diagram. (For example, spherical 2-categories as in \cite{spherical} could be conjectured to correspond to 2-categories in which quilt diagrams of genus 0 yield well defined maps.)

\begin{figure}[!h]
\centering
\includegraphics[width=5.5in]{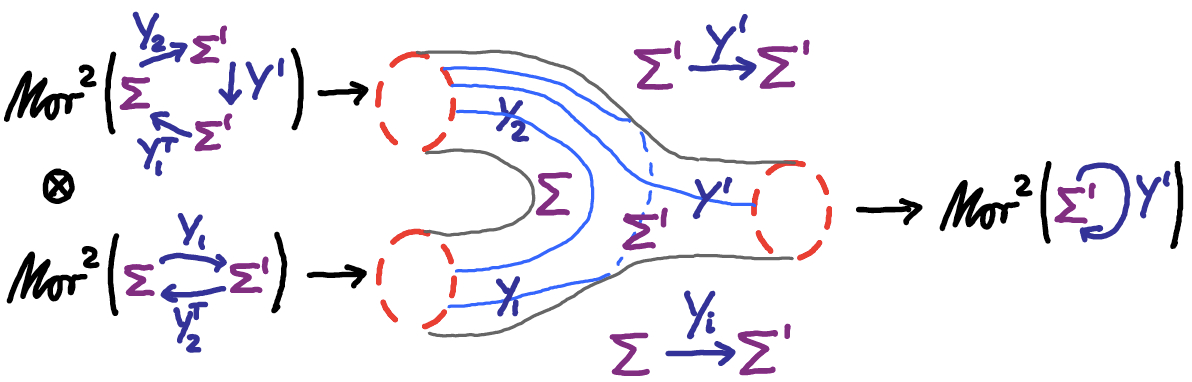}
\caption{Quilt diagrams induce quilted composition maps which -- except in simple cases reducing to string diagrams -- cannot be expressed in terms of the horizontal and vertical composition of 2-morphisms.}
\label{fig:quiltmap}
\end{figure}

\begin{definition} \label{def:quilted}
A {\bf quilted bicategory/2-category} is a bicategory/2-category $\cC$ with adjoints in the sense of Definition~\ref{def:adj} and with quilted composition maps\footnote{
Here and below we write a tensor product $\otimes$ to indicate a Cartesian product of sets which can be replace by a tensor product in the case of 2-morphism spaces given by Floer homology groups. 
}
$$
\Phi_{\cQ\cD}:\otimes_{e\in\cE_\cQ^-} \Mor^2_\cC(\ul Y_e) \to \Mor^2_\cC(\ul Y_{e^+})
$$
for each quilt diagram
${\cQ\cD}=\bigl(\cQ,(\Sigma_P)_{P\in\cP_\cQ}, (Y_S)_{S\in\cS_\cQ}\bigr)$
that satisfy the following:

\smallskip
\noindent
{\bf Deformation Axiom:}
Isomorphic quilt diagrams $\cQ\cD\simeq \cQ\cD'$ as in Figure~\ref{fig:isomorphic} give rise to the same quilted composition maps $\Phi_{\cQ\cD}=\Phi_{\cQ\cD'}$. Here an isomorphism 
$$
\bigl(\cQ,(\Sigma_P)_{P\in\cP_\cQ}, (Y_S)_{S\in\cS_\cQ}\bigr)\simeq \bigl(\cQ',(\Sigma'_{P'})_{P'\in\cP_{\cQ'}}, (Y'_{S'})_{S'\in\cS_{\cQ'}}\bigr)
$$ 
is a homeomorphism $Q_2\simeq Q_2'$ that restricts to an orientation preserving diffeomorphism $Q_2\less Q_0\simeq Q'_2\less Q'_0$ and 
identifies the ends $Q_0\simeq Q_0'$, in particular $q_0\simeq q_0'$, and seams $Q_1\simeq Q_1'$ in such a way that the labels $\Sigma_{P}=\Sigma'_{P'}$ and $Y_S=Y'_{S'}$ coincide under the induced identification of patches $\cP_\cQ \simeq \cP_{\cQ'}$ and oriented seams $\cS^{\rm or}_\cQ \simeq \cS^{\rm or}_{\cQ'}$. 
The identity $\Phi_{\cQ\cD}=\Phi_{\cQ\cD'}$ is with respect to the identification of ends $\cE^\pm_\cQ \simeq \cE^\pm_{\cQ'}$ induced by the bijection $Q_0\simeq Q_0'$. 
\begin{figure}[!h]
\centering
\includegraphics[width=5.5in]{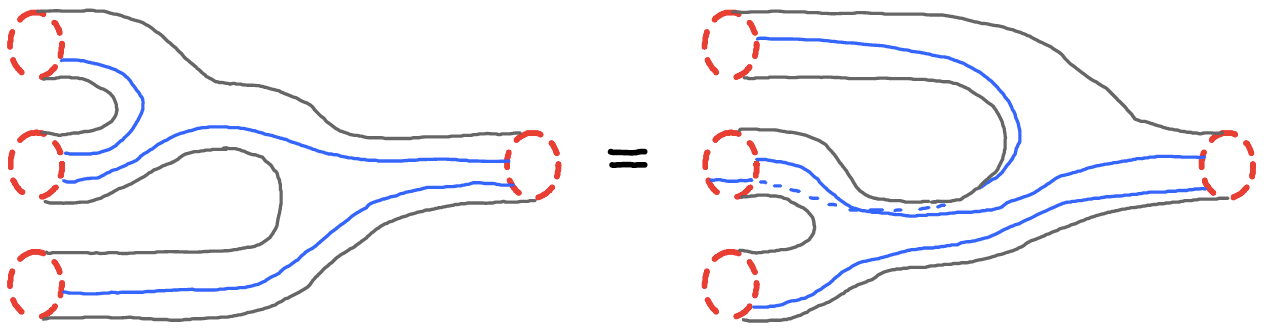}
\caption{Isomorphic quilt diagrams yield the same quilted 
maps.}
\label{fig:isomorphic}
\end{figure}

\noindent
{\bf Cylinder Axiom:}
The invariant associated to a quilted cylinder as in Figure~\ref{fig:cylinder} -- that is $Q_2\less Q_0 \simeq \R\times S^1$ with parallel seams $Q_1\simeq\R\times\{s_1,\ldots,s_N\}$ -- is the identity map on the associated cyclic morphism space.
\begin{figure}[!h]
\centering
\includegraphics[width=5.5in]{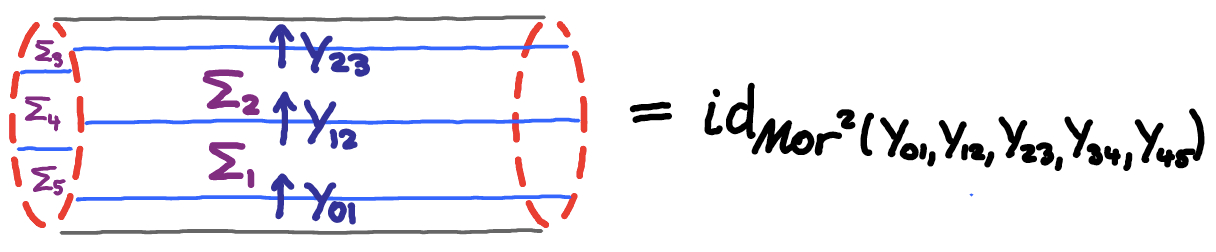}
\caption{Cylindrical quilt diagrams yield the identity map}
\label{fig:cylinder}
\end{figure}

\noindent
{\bf Gluing Axiom:}
Gluing of quilt diagrams as in Figure~\ref{fig:gluing} -- identifying the outgoing end of one diagram with an incoming end of another diagram -- corresponds to composition of the associated quilted composition maps.
\begin{figure}[!h]
\centering
\includegraphics[width=5.5in]{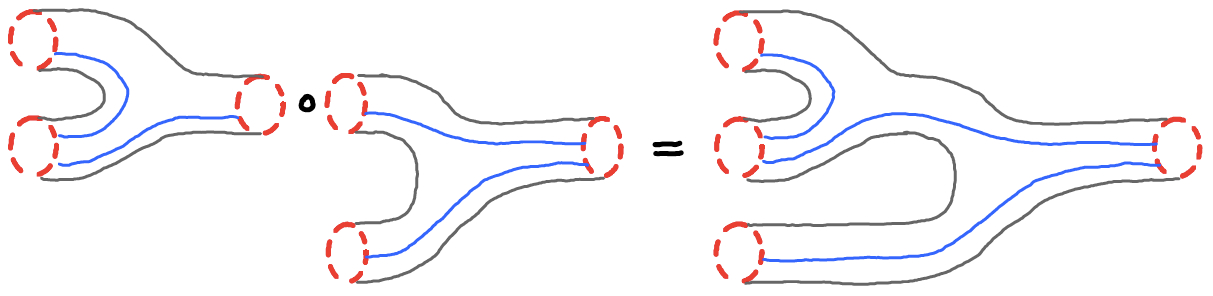}
\caption{Composition of quilted composition maps corresponds to gluing of the quilt diagrams.}
\label{fig:gluing}
\end{figure}

\noindent
{\bf Strip shrinking Axiom:}
Strip or annulus shrinking as in Figure~\ref{fig:shrink} -- removing a patch $P\simeq \R\times S^1$ or $P\simeq [0,1]\times S^1$ and replacing the its two adjacent seams $S,S'$ by a single seam labeled with the composed 1-morphism $Y_S\circ Y_{S'}$ (and its adjoint) -- corresponds to an equality of quilted composition maps.
\begin{figure}[!h]
\centering
\includegraphics[width=5.5in]{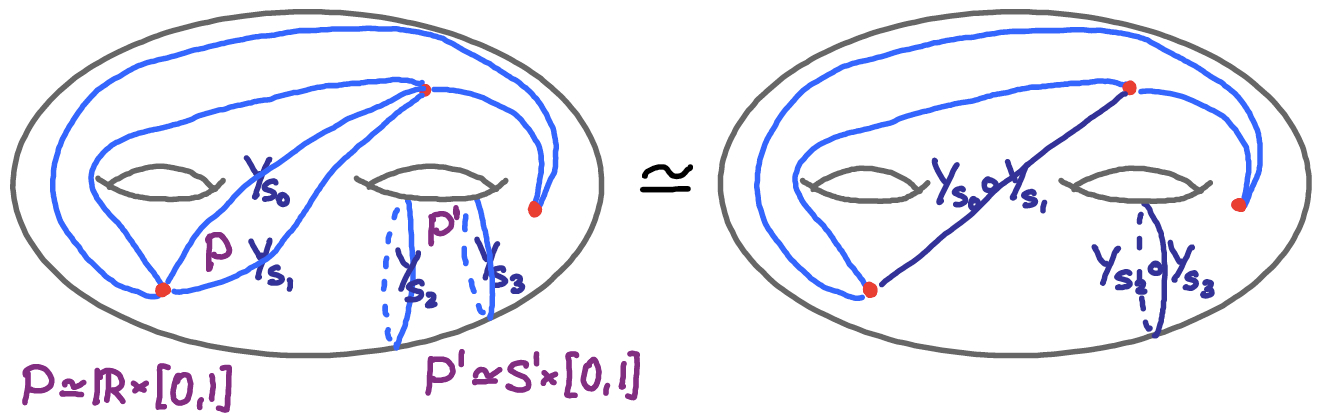}
\caption{Quilt diagrams related by annulus shrinking yield the same quilted composition maps; for strip shrinking they are intertwined via isomorphisms between the cyclic morphism spaces.}
\label{fig:shrink}
\end{figure}
\end{definition}


An example of using the axioms to make graphical calculations for quilt maps is given in Figure~\ref{fig:quiltdiagram-calculation}.

\begin{remark} \label{rmk:quiltcat} \rm 
The adjunction 2-morphisms between adjoint 1-morphisms (see Definition~\ref{def:adj}) are in practice often constructed from quilted composition maps corresponding to Figure~\ref{fig:quiltadjoints} with no incoming ends; e.g.\ as in Remarks~\ref{rmk:adj2} and \ref{rmk:sympquilt}. 
A more fitting notion of quilted bicategory might thus be to require only the reflexive operation on 1-morphisms in Definition~\ref{def:adj} together with well defined cyclic 2-morphism spaces as in Remark~\ref{rmk:cyclic} and quilted composition maps satisfying the same axioms as in Definition~\ref{def:quilted}.
\end{remark}

In \S\ref{ss:symp2} and \S\ref{ss:gauge2} we will see (sketches of) examples in which the quilted composition maps arise from counting solutions to a nonlinear PDE. In those settings, the strip and annulus shrinking is highly nontrivial -- requiring the identification of solution spaces under a degeneration of the PDE as in \cite{ww:isom}.
On the other hand, bordism bicategories have natural quilted composition maps given by appropriate gluing of manifolds with boundaries and corners, in which also strip and annulus shrinking is naturally satisfied. We give a rough explanation here in dimension $d=2$, though more care would be required to construct the cyclic 2-morphism spaces and smooth structures coherently and check the axioms.

\begin{lemma} \label{lem:borquilt}
$\Bor_{2+1+1}$ is a quilted bicategory with adjoints as in Remark~\ref{rmk:adj2}. 
\end{lemma}
\begin{proof}[Sketch]
Since adjunction 2-morphisms $X_Y,X_Y^T$ are constructed from quilt diagrams in Remark~\ref{rmk:adj2}, it remains to associate to a quilt diagram $\cQ\cD=\bigl(\cQ,(\Sigma_P)_{P\in\cP_\cQ}, (Y_S)_{S\in\cS_\cQ}\bigr)$ 
and labels $\bigl(X_e \in \Mor^2_{\Bor_{2+1+1}}(\ul Y_e)\bigr)_{e\in\cE^+_\cQ}$ of the incoming ends a 4-cobordism in the cyclic 2-morphism space $\Mor^2_{\Bor_{2+1+1}}(\ul Y_{e^+})$ associated to the outgoing end. We do so by gluing 4-manifolds as shown in Figures~\ref{fig:2bordismstuff} and \ref{fig:borcomp}:
\begin{itemize}
\item
A patch $P\in\cP_\cQ$ that is labeled by a surface $\Sigma$ is represented by the oriented 4-manifold $X_P:=\widehat P\times\Sigma$ with boundary $\partial X_P = \bigcup_{S\in\cS^{\rm or}_\cQ, P^+_S=P} S \times \Sigma$.
\item
A seam $S\in\cS_\cQ$, i.e.\ a pair of oriented seams $\{S,S^-\}\in\cS^{\rm or}_\cQ$, that is labeled by a 3-cobordism $Y_S\in\Bor_{2+1}(\Sigma_{P_S^-},\Sigma_{P_S^+})$ and its adjoint $Y_{S^-}=Y_S^-$ is represented by the oriented 4-manifold $X_S:=S\times Y_S$ with boundary 
$$
\partial X_S = \bigl( S \times \Sigma_{P_S^-} \bigr)  \cup \bigl(  S \times \Sigma_{P_S^+}^-\bigr)
= \bigl( S \times \Sigma_{P_S^-} \bigr)  \cup \bigl(  S^- \times \Sigma_{P^-_{S^-}}\bigr).
$$ 
(Note that this is independent of the choice of orientation on $S\in\cS_\cQ$.)
\item
We now glue the 4-manifold $\bigsqcup_{P\in\cP_\cQ} X_P$ with boundary $\bigcup_{S\in\cS^{\rm or}_\cQ} S \times \Sigma_{P^+_S}$ to the 4-manifold $\bigsqcup_{S\in\cS_\cQ} X_S$ with boundary $\bigsqcup_{S\in\cS^{\rm or}_\cQ} S \times \Sigma_{P_S^-}$
via the orientation reversing diffeomorphisms 
$$
\partial X_S \;\supset\; S \times \Sigma_{P_S^-} 
\;\longrightarrow\; S^- \times \Sigma_{P_S^-} = S^- \times \Sigma_{P^+_{S^-}} \;\subset\; \partial X_{P^+_{S^-}} = \partial X_{P^-_S} .
$$
\end{itemize}
If we extend the smooth structure by gluing with appropriate collar neighbourhoods, then this yields an oriented 4-manifold $X_{\cQ\cD}$ without boundary, which is compact up to cylindrical ends $\R^\pm\times Y_e \subset X_{\cQ\cD}$ for each end $e\in\cE^\pm_\cQ$.
If we now delete a little neighbourhood of $e\in Q_0\subset Q_2$ from all patches and seams, then each cylindrical end is replaced by boundary and corners as follows:

The boundary strata near $e$ are the 3-cobordisms 
$Y_{S_i}\in\Mor^1_{\Bor_{2+1+1}}(\Sigma_{P_i^-},\Sigma_{P_i^+})$ 
in the cyclic 1-morphism $\underline{Y}_e=\bigl( Y_{S_i} \bigr)_{i\in\Z_{N_e}}$ and identity cobordisms $1_{\Sigma_{P_i^\pm}} = [0,\delta_i]\times \Sigma_{P_i^\pm}$.
The corners are formed by identifications $\im\, \iota_{Y_{S_i}}^-(0,\cdot) \sim \{1\}\times \Sigma_{P_i^-}$ and $\im\, \iota_{Y_{S_i}}^+(1,\cdot) \sim \{0\}\times \Sigma_{P_i^+}$.
This boundary\&corner structure corresponds (with reversed orientations) to the boundary\&corners of $X_e\in \Mor^2_{\Bor_{2+1+1}}(\ul Y_e)$, so that we can glue in these 4-manifolds at each incoming end to obtain a 4-manifold with boundary and corners arising from the outgoing end. This defines the result of the quilted composition map 
$\Phi_{\cQ\cD}\bigl(\otimes_{e\in\cE_\cQ^-} X_e \bigr) \in \Mor^2_\cC(\ul Y_{e^+})$.

This construction is fairly evidently compatible with isomorphisms of quilt diagrams, gluing, and strip shrinking, thus satisfies the axioms required in Definition~\ref{def:quilted} of a quilted bicategory.
\end{proof}

The notion of quilted 2-categories now allows us to formulate the following extension principle which we will further discuss in \cite{W:ext}. Its proof is outlined in \cite{W:slides} and makes crucial use of the fact that, as a result of the theory of Morse 2-functions (see e.g.\ \cite{GayKirby}), $\Bor_{2+1+1}$ is not just a quilted bicategory but in an appropriate sense is quilt-generated by the Cerf decompositions of $\Bor_{2+1}$ of Theorem~\ref{thm:borCerf}.

\begin{conjecture}[Extension principle for Floer field theories] \label{con:fftext}
Let $\cC$ be a quilted 2-category as in Definition~\ref{def:quilted} whose underlying 1-category has Cerf decompositions as in Definition~\ref{def:Cerf}. 
Then any {\rm Cerf}-compatible partial functor as in Lemma~\ref{le:field0},
$\cF: (\Obj_{\Bor_{2+1}},\SMor_{\Bor_{2+1}}) \to (\Obj_\cC,\SMor_\cC)$,
 which preserves adjunctions as in Remark~\ref{rmk:field2} and satisfies a quilted naturality axiom (see \cite{W:slides}), 
has a natural extension to a 2-functor $\Bor_{2+1+1}\to\cC$.
\end{conjecture}

An analogous extension principle can be formulated for bordism bicategories $\Bor_{d+1+1}$ in any dimension $d\ge 0$ and connected bordism bicategories $\Bor^{\rm conn}_{d+1+1}$ in dimension $d\ge 0$.
We propose to apply this principle to the Floer field theories outlined in \S\ref{ss:ex}, where $\cC$ the symplectic 2-category oulined in \S\ref{ss:symp2} or other gauge theoretic 2-categories outlined in \S\ref{ss:gauge2}. 
It should yield ``2+1+1 Floer field theories'' $\Bor_{2+1+1}\to\Cat$ by composition with the Yoneda 2-functor $\cC\to\Cat$  from Lemma~\ref{lem:Yoneda}.
We moreover expect equivalences between these field theories, as phrased in the quilted Atiyah-Floer Conjecture~\ref{con:qAF}.

\subsection{The symplectic 2-category} \label{ss:symp2}

This section gives a brief overview of the construction of a symplectic 2-category in \cite{ww:cat}. 
Conceptually, it can be thought of as starting with the construction of a 2-category in Example~\ref{ex:ext0} from the Cerf decompositions in the extended symplectic category $\Symp^\#$ of Definition~\ref{def:extsymp}, and then replacing the 2-morphisms that were defined from the abstract Cerf moves by a geometrically more meaningful notion, while preserving the isomorphisms $(L_{12},L_{23}) \sim L_{12}\circ L_{23}$ in the sense of Remark~\ref{rmk:2cat-quotient}, as mentioned at the end of \S\ref{ss:symp}.

In the following sketch of the symplectic 2-category we use the same horizontal 1-composition $\circ^1_{\rm h}$ as in $\Symp^\#$, the 2-morphisms are given by (quilted) Floer homology groups as defined in \cite{ww:qhf}. These arise from a complex whose differential is constructed from moduli spaces of solutions of an elliptic PDE that is closely connected to the PDE that we associate to quilt diagrams in Remark~\ref{rmk:quilts}. We also use these moduli spaces to construct the vertical and horizontal 2-composition $\circ_{\rm v}, \circ^2_{\rm h}$ from their respective string diagrams. In order to obtain well defined structures, we have to make further restrictions on the allowable symplectic objects and morphisms as in Remark~\ref{rmk:monotone}, or generalize the notion of 2-category, as discussed in Remark~\ref{rmk:ainfty2symp}.

\begin{example} \label{ex:2symp} \rm 
The {\bf symplectic 2-category ${\rm\mathbf{Symp}}$} roughly consists of the following.
\begin{itemize}
\item
Objects are symplectic manifolds $M$.
\item
For each pair $M,N\in \Obj_{\Symp}$ the category of 1-morphisms is the following Donaldson-Fukaya category of generalized Lagrangians $\Mor_{\Symp}(M,N)$.
\begin{itemize}
\item
1-morphisms $\uL=(L_{01},\ldots,L_{(k-1)k})\in\Mor^1_{\Symp}(M,N)$ are the composable chains of simple Lagrangians $L_{ij}\subset M_i^- \times M_j$ between symplectic manifolds $M=M_0, M_1, \ldots, M_k=N$.
\item
2-morphisms between $\uL,\uL'\in \Mor^1_{\Symp}(M,N)$ are the elements of the quilted Floer homology group $\Mor_{\Symp}^2(\uL,\uL')=HF(\uL,\uL')$; see Remark~\ref{rmk:quilts}.
\item
Vertical composition
$\circ_{\rm v}: HF(\uL,\uL')\otimes HF(\uL',\uL'')\to HF(\uL,\uL'')$
for $\uL,\uL',\uL''\in \Mor^1_{\Symp}(M,N)$ 
arises from counts of pseudoholomorphic quilts representing the associated string diagram.
It is associative by a gluing theorem as in \cite{mau}.
\item 
The identity $\id_{\uL}\in\Mor_{\Symp}^2(\uL,\uL)$ for $\uL\in \Mor^1_{\Symp}(M,N)$ 
arises from counts of pseudoholomorphic quilts representing the associated string diagram.
\end{itemize}
\item
The composition functor $\Mor_{\Symp}(M,N)\times \Mor_{\Symp}(N,P)\to \Mor_{\Symp}(M,P)$ is defined as follows.
\begin{itemize}
\item
Horizontal composition of 1-morphisms
$$
\circ_{\rm h}: \; \Mor_{\Symp}^1(M,N)\times  \Mor_{\Symp}^1(N,P)\to  \Mor_{\Symp}^1(M,P), 
\quad
(\uL,\uL') \mapsto \uL \# \uL'
$$
is given by the evidently associative concatenation
$$
(L_{01},\ldots,L_{(k-1)k}) \# (L'_{01},\ldots,L'_{(k'-1)k'})
:= (L_{01},\ldots,L_{(k-1)k},L'_{01},\ldots,L'_{(k'-1)k'}).
$$
\item
The identities $1_M= (\;) \in\Mor_{\Symp}^1(M,M)$ for $\circ_{\rm h}$ are given by the 
trivial chains.
\item
Horizontal composition of 2-morphisms arises from counts of pseudoholomorphic quilts representing the associated string diagram,
\begin{align*}
 \circ_{\rm h} : \; HF(\uL_{12},\uL'_{12}) \times HF(\uL_{23},\uL'_{23})  &\;\longrightarrow\;  HF(\uL_{12}\#\uL_{23},\uL'_{12}\#\uL'_{23}) .
\end{align*}
Compatibility with identities and vertical composition follows from gluing theorems as in \cite{mau}.
\end{itemize}
\end{itemize}
While this gives a well defined symplectic 2-category, we still have to relate it to the symplectic category defined in \S\ref{ss:symp}, in which horizontal composition of morphisms is given by the geometric composition of Lagrangians -- if the latter is embedded. 
Example~\ref{ex:ext0} shows how the same can be achieved up to isomorphism in a 2-categorical setting. However, the 2-morphisms in the present 2-category are quilted Floer homology classes, so the following becomes a nontrivial result -- proven in \cite{ww:isom} as isomorphism of Floer homologies, which is formulated categorically in \cite{ww:cat}.
\begin{itemize}
\item
For any pair of Lagrangians $L_{12}\subset M_1^- \times M_2$, $L_{23}\subset M_2^- \times M_3$ with embedded geometric composition $L_{12}\circ L_{23}\subset M_1^- \times M_2$ as defined in \eqref{eq:embedded}, the 1-morphisms $L_{12}\circ_{\rm h} L_{23} = L_{12}\# L_{23} \sim L_{12}\circ L_{23}$ are isomorphic in $\Mor_{\Symp}^1(M_1,M_3)$ in the sense of Remark~\ref{rmk:2cat-quotient}:
We have $\alpha \circ_{\rm v} \beta = \id_{L_{12}\# L_{23}}$ and 
$\beta \circ_{\rm v} \alpha = \id_{L_{12}\circ L_{23}}$
for some
2-morphisms 
$\alpha\in\Mor^2_{\Symp}(L_{12}\# L_{23}, L_{12}\circ L_{23})$,
$\beta\in\Mor^2_{\Symp}(L_{12}\circ L_{23}, L_{12}\# L_{23})$.
\end{itemize}
The last item ensures that the symplectic category $\Symp^1:=\Symp^\#/\!\!\sim\,$ of Definition~\ref{def:1symp} and Example~\ref{ex:ext0} and the
quotient $|\Symp|=\Symp/\!\!\sim\,$ as in Remark~\ref{rmk:2cat-quotient} of the symplectic 2-category by isomorphisms are related by a functor
$$
\Symp^1\to |\Symp|, \qquad
M \mapsto M, \qquad  
\quo{\Mor_{\Symp^\#}}{\sim} \ni [\uL] \mapsto [\uL] \in \quo{\Mor^1_{\Symp}}{\sim}
$$
since equivalence $\sim$ in $\Mor_{\Symp^\#}$ implies isomorphism $\sim$ in $\Symp$.
This functor is full, i.e.\ surjective on morphism spaces, but it is not faithful, i.e.\ injective, since generalized Lagrangians $\uL$ in the symplectic 2-category may be Floer-theoretic isomorphic without being related by embedded geometric compositions. In fact, the difference can already be seen for simple Lagrangians $L,L'\subset M^-\times N$, which are equivalent in $\Mor_{\Symp^\#}(M,N)$ only if they are identical, but whenever $L'=\phi(L)$ is the image of $L$ under a Hamiltonian symplectomorphism $\phi:M\to M$, then standard Floer theoretic arguments show that $L\sim L'$ are isomorphic as 1-morphisms in $\Mor^1_{\Symp}(M,N)$.
\end{example}

\begin{remark}[Adjoints and quilted composition maps in $\Symp$] \label{rmk:sympquilt}\rm 
The symplectic 2-category has adjoints as follows:
\begin{itemize}
\item
For a Lagrangian\footnote{
Note that an overall sign change of the symplectic form does not affect the Lagrangian property.
} 
$L\subset M_0^-\times M_1$ the adjoint 1-morphism $L^T\subset M_1^-\times M_0$ 
is given by the image of $L$ under transposition of factors $M_0\times M_1\to M_1\times M_0$.
\item
For a general 1-morphism $\uL=(L_{01},\ldots,L_{(k-1)k})\in\Mor^1_{\Symp}(M,N)$ the adjoint is given by reversal and transposition, $\uL^T=(L_{(k-1)k}^T, \ldots, L_{01}^T)\in\Mor^1_{\Symp}(N,M)$.
\item
Duality for $\uL\in\Mor^1_{\Symp}(M,N)$ is guaranteed by the identity elements
$$
X_{\uL} := \id_{\uL} = \id_{\uL^T} \in HF(\uL\#\uL^T) ,\qquad
X_{\uL}^T := \id_{\uL} = \id_{\uL^T} \in HF(\uL^T\#\uL) 
$$
since the quilted Floer homology in \cite{ww:qhf} has canonical cyclic symmetries 
\begin{align*}
HF(\uL\#\uL^T) & = \Mor_\cC^2(1_{M}, \uL\circ_{\rm h} \uL^T) = \Mor_\cC^2(\uL,\uL)  \\
= HF(\uL^T\#\uL) & = \Mor_\cC^2(\uL^T\circ_{\rm h} \uL, 1_{N}) =  \Mor_\cC^2(\uL^T,\uL^T)   
\end{align*}
which identify these morphism spaces and their identity elements $\id_{\uL}, \id_{\uL^T}$ defined in \cite{ww:cat}, so that the required identities reduce to the compatibility of horizontal and vertical composition with identities, 
$$
\bigl(X_{\uL}\circ_{\rm h} \id_{\uL}\bigr) \circ_{\rm v} \bigl(\id_{\uL} \circ_{\rm h} X_{\uL}^T\bigr)  = \id_{\uL},
\qquad
\bigl(\id_{{\uL}^T}\circ_{\rm h}X_{\uL} \bigr) \circ_{\rm v} \bigl( X_{\uL}^T\circ_{\rm h} \id_{\uL^T}\bigr)  = \id_{\uL^T}.
$$
\end{itemize}
Moreover, $\Symp$ is a quilted 2-category whose quilted composition maps 
$$
\Phi_{\cQ\cD}:\otimes_{e\in\cE_\cQ^-} HF(\ul L_e) \to HF (\ul L_{e^+}) 
$$
are defined in \cite{ww:quilts}, which also proves the axioms for quilted cylinders and gluing of diagrams by standard Floer theoretic arguments.
However, strip and annulus shrinking -- as required in Definition~\ref{def:quilted} for a quilted 2-category -- requires the adiabatic limit analysis in \cite{ww:isom}, which may be obstructed by a novel ``codimension 0 in the boundary'' singularity -- figure eight bubbles.
\end{remark}

\begin{remark}[PDE associated to quilt diagrams in $\Symp$]  \label{rmk:quilts}\rm 
The key step in the construction \cite{ww:quilts} of quilted composition maps is to associate to every quilt diagram 
$\cQ\cD=\bigl(\cQ,(M_P)_{P\in\cP_\cQ}, (\uL_S)_{S\in\cS_\cQ}\bigr)$ an elliptic PDE as follows:
\begin{itemize}
\item
A patch $P$ labeled by a symplectic manifold $M_P$ is represented by a pseudoholomorphic map $u_P:\widehat P\to M_P$ whose domain is the oriented 2-manifold $\widehat P$ that covers the closure $\overline P\subset Q_2\less Q_0$ as in Remark~\ref{rmk:quilt}.
\item
A seam $S$ labeled by a Lagrangian submanifold $L_S\subset M_{P_S^-}^- \times M_{P_S^+}$ is represented by a Lagrangian seam condition: The map $u_{P_S^-}|_S \times u_{P_S^+}|_S:S\to M_{P_S^-}^- \times M_{P_S^+}$ induced by boundary restrictions of the pseudoholomorphic maps associated to the adjacent patches is required to take values in $L_S$.
\item 
A seam $S\simeq \R$ labeled by a sequence of Lagrangians $L_{i(i+1)}\subset M_i^-\times M_{i+1}$ which form a general 1-morphism $\uL=(L_{01},\ldots,L_{(k-1)k}\in \Mor_{\Symp}^1(M_{P_S^-}, M_{P_S^+})$, 
represents pseudoholomorphic strips $u_i:\R\times [0,1]\to M_i$ for $i=1,\ldots,k-1$ with Lagrangian seam conditions $(u_i|_{\{1\}\times\R} \times u_{i+1}|_{\{0\}\times\R})(\R)\subset L_{i (i+1)}$ and
$$
(u_{P_S^-}|_S \times u_1|_{\{0\}\times\R} ) (\R)\subset L_{01}, \qquad
(u_{k-1}|_{\{1\}\times\R} \times u_{P_S^+}|_S )(\R)\subset L_{(k-1)k} .
$$
\item 
A seam $S\simeq S^1$ labeled by a general 1-morphism $\uL$ represents pseudoholomorphic annuli $u_i:S^1\times [0,1]\to M_i$ with the analogous seam conditions.
\end{itemize}
The tuple $(u_P)_{P\in\cP_\cQ}$ (together with the additional maps $(u_i)_{i=1,\ldots k-1}$ from each seam with generalized Lagrangian label) is called a {\bf pseudoholomorphic quilt}. 

This notion generalizes pseudoholomorphic maps $u:Q\to M$ -- which arise from quilt diagrams $\cQ\cD=(Q,M,\emptyset)$ that consist of a closed Riemann surface $Q_2=Q$ without seams or punctures, labeled by a symplectic manifold $M$ -- as well as pseudoholomorphic maps with Lagrangian boundary conditions 
$u:(Q,\partial Q) \to (M,L)$.
To build in boundary, we can for example represent the latter by a quilt diagram $\cQ\cD=(\cQ,M,L)$ whose quilted surface has patches $Q$ and $\overline Q$ (with reversed orientation), a seam for each boundary component of $Q$ (identified with the corresponding boundary component of $\overline Q$), labels $M$ for $Q$, $\pt$ for $\overline Q$, and $L$ for each seam.

We can moreover build in any number of punctures on boundaries, seams, or in the interior.
Note in particular that interior punctures on a patch $P$ are associated to the cyclic 2-morphism space $\Mor_{\Symp}^2(1_{M_P})= HF(\Delta_{M_P})\cong HF(M_P)$, which can be identified with the Hamiltonian Floer homology of the symplectic manifold $M_P$. 
For an introduction to Floer homology see e.g.\ \cite{Sa:lecture, seidel}. These also provide good introductions to the technique of ``counting'' (very specific) moduli spaces of PDEs to construct Floer chain complexes and chain maps between them -- whose homology and induced map on homologies are independent of choices (most notably of perturbations that are chosen to regularize the moduli spaces).
\end{remark}

\begin{remark}[Monotonicity assumptions]  \label{rmk:monotone}\rm 
Moduli spaces of pseudoholomorphic quilts -- just as moduli spaces of pseudoholomorphic curves -- are rarely compact and often do not carry a smooth structure which allows us to ``count'' or ``integrate over'' them to define the structure maps in the symplectic 2-category. 
While the ``Gromov compactification'' of these spaces (in terms of breaking of Floer trajectories and bubbling trees of pseudoholomorphic spheres and disks; see e.g.\ \cite{ms:jhol,seidel}) is well understood, the regularization of the compactified moduli spaces still remains a challenge in general settings (see \cite{MW} for a survey). In fact, bubbling gives actual obstructions to the algebraic requirements for a 2-category -- beginning with disk bubbling obstructing the definition of the 2-morphism spaces (since the Floer differential may fail to square to zero), via additional algebraic terms in the structure equations arising from disk bubbles, to figure eight bubbles obstructing the desired isomorphism $L_{12} \# L_{23} \sim L_{12} \circ L_{23}$ for embedded geometric composition.

The present state of the art is that a rigorous symplectic 2-category $\Symp^\tau$ is constructed in \cite{ww:cat} by restriction to monotone or exact symplectic manifolds and oriented Lagrangian submanifolds with minimal Maslov index $\ge 3$.
While the latter assumption is made to ensure that the Floer differential squares to zero (so that Floer homology is well defined), the monotonicity requires that the Maslov index $I(\ul u)$ and symplectic area $A(\ul u)$ of the quilted maps $\ul u=(u_P)_{P\in\cP_\cQ}$ are proportional $I(\ul u)=\tau A(\ul u)$ via a constant $\tau\ge 0$.
This helps with excluding bubbling because it relates Fredholm indices (i.e.\ expected dimension of moduli spaces) to the energy of the solutions, so that bubbling (i.e.\ loss of energy) forces loss of Fredholm index -- which in the relevant moduli spaces would yield solutions of negative expected dimension. Once these are ruled out by appropriate regularization, the bubbling can be excluded without actually constructing the compactified moduli space.
The same argument is used to exclude bubbling in the strip and annulus shrinking of \cite{ww:isom} to prove the isomorphism $L_{12}\# L_{23} \sim L_{12} \circ L_{23}$ in $\Symp^\tau$ when the latter geometric composition is embedded. 
\end{remark}

\begin{remark}[Generalized notions of symplectic 2-categories]  \label{rmk:ainfty2symp}\rm
In order to extend the construction of a symplectic 2-category to non-monotone settings, 
and more generally study the relationship between the algebraic and geometric compositions 
$L_{12}\# L_{23}$ and $L_{12} \circ L_{23}$, 
a Gromov compactification for strip and annulus shrinking -- involving multi-level trees of pseudoholomorphic disks, figure eights, and spheres -- is constructed in \cite{BW:squiggly} with the help of removable singularity results for the figure eight bubble in \cite{B:remsing}. 
By analyzing the boundary strata of the resulting compactified moduli spaces, and supported by the upcoming Fredholm theory \cite{B:Fred} for moduli spaces of figure eights, we then predict a 2-categorical structure that comprises all (compact) symplectic manifolds and Lagrangians, and in which composition of 1-morphisms is given by geometric composition of Lagrangians. It takes the form of a curved $A_\infty$ 2-category, 
and in fact motivates the definition of this new algebraic notion in \cite{BT}.
\end{remark}

We end this section by disclosing the categorical ignorance in the first publications on the symplectic 2-category in \cite{ww:cat}. While that paper painstakingly constructs a 2-functor $\Symp^\tau\to\Cat$, this directly coincides with the Yoneda construction.

\begin{lemma}\label{le:sympcat}
The functor $\Symp^\tau\to\Cat$ constructed in \cite{ww:cat} is identical to the functor $\cF_{\pt}$ given by Lemma~\ref{lem:Yoneda} with the distinguished object $x_0=\pt$.
\end{lemma}

Similarly, \cite{ww:qhf} proves isomorphisms between quilted Floer homology groups for cyclic 1-morphisms in $\Symp$ which are related by a geometric composition by arguing that the adiabatic limit analysis in \cite{ww:isom} transfers directly. In the 2-categorical setup with adjoints, this can now be proven more directly by the categorical ``local to global'' argument of Remark~\ref{rmk:localtoglobal}.

\begin{remark}[Isomorphisms of Floer homology under geometric composition] \rm\label{rmk:localtoglobal-Lag}
The ``local to global'' principle discussed in \S\ref{ss:af} and Remark~\ref{rmk:localtoglobal} translates to the quilted Floer homology groups via identifications
$$
HF\bigl( \uL=(L_{i(i+1)})_{i\in\Z_N} \bigr) = \Mor^2_{\Symp} \bigl( \uL=(L_{i(i+1)})_{i\in\Z_N} \bigr) . 
$$
So for purely algebraic reasons (which are interpreted geometrically in Remark~\ref{rmk:LLBS}), we obtain the implications 
\begin{align*}
L_{12} \# L_{23} \sim L_{12} \circ L_{23}  \quad&\Longrightarrow\quad 
( \ldots , L_{12} \# L_{23},  \ldots ) \sim ( \ldots L_{12} \circ L_{23}  \ldots )  \\
&\Longrightarrow\quad 
HF( \ldots , L_{12} , L_{23},  \ldots ) \simeq HF( \ldots L_{12} \circ L_{23}  \ldots )  .
\end{align*}
Thus to prove that quilted Floer homology is invariant (up to isomorphism) under embedded geometric composition, it suffices to prove that any embedded geometric composition $L_{12} \circ L_{23}$ as defined in \eqref{eq:embedded} gives rise to an isomorphism $L_{12} \# L_{23} \sim {L_{12} \circ L_{23}}=: L_{13}$ between algebraic and geometric compositions.
Such local isomorphisms require the construction of quilted Floer homology classes 
\begin{align*}
\alpha &\in \Mor^2_{\Symp}(L_{12} \# L_{23} , L_{12} \circ L_{23}) = HF(L_{12}, L_{23}, L_{13}^T),  \\
\alpha^{-1} &\in \Mor^2_{\Symp}(L_{12} \circ L_{23} , L_{12} \# L_{23}) = HF(L_{13}, L_{23}^T, L_{12}^T) 
\end{align*}
that satisfy 
$$
\alpha\circ_{\rm v} \alpha^{-1} = \id_{L_{12} \# L_{23}}, \qquad
\alpha^{-1}\circ_{\rm v} \alpha = \id_{L_{13}} .
$$ 
To find such classes suppose that we have isomorphisms of the ``local'' quilted Floer homologies under embedded geometric composition (with $L_{23}^T\circ L_{12}^T  = (L_{12}\circ L_{23})^T$), 
\begin{align*}
HF(L_{12}, L_{23}, L_{13}^T) &\overset{\sim}{\to} HF(L_{12}\circ L_{23}, L_{13}^T) = HF(L_{13}, L_{13})= \Mor^2_{\Symp}(L_{13},L_{13}) , \\ 
HF(L_{13}, L_{23}^T, L_{12}^T) &\overset{\sim}{\to}  HF(L_{13}, L_{23}^T\circ L_{12}^T ) = HF(L_{13}, L_{13})= \Mor^2_{\Symp}(L_{13},L_{13}), \\
HF(L_{12}, L_{23}, L_{13}^T) &\overset{\sim}{\to} HF(L_{12},  L_{23}, L_{23}^T, L_{12}^T) = \Mor^2_{\Symp}(L_{12}\#L_{23},L_{12}\#L_{23}) , \\ 
HF(L_{13}, L_{23}^T, L_{12}^T) &\overset{\sim}{\to}  HF(L_{12}, L_{23}, L_{23}^T , L_{12}^T )  = \Mor^2_{\Symp}(L_{12}\#L_{23},L_{12}\#L_{23}) ,
\end{align*}
and suppose that these isomorphisms are compatible with identities and products.
Then we may pull back the identities  $\id_{L_{13}}\in \Mor^2_{\Symp}(L_{13},L_{13})$ and 
$\id_{L_{12}\#L_{23}}\in Mor^2_{\Symp}(L_{12}\#L_{23},L_{12}\#L_{23})$ to obtain two well defined 
classes $\alpha,\alpha^{-1}$ as required,
$$
\alpha\circ_{\rm v} \alpha^{-1} =  \id_{L_{12} \# L_{23}} \circ_{\rm v}  \id_{L_{12} \# L_{23}} =  \id_{L_{12} \# L_{23}}, \qquad
\alpha^{-1}\circ_{\rm v} \alpha = \id_{L_{13}} \circ_{\rm v}  \id_{L_{13}} =  \id_{L_{13}}.
$$
\end{remark}

Finally, we will clarify some confusions regarding the generality and possible obstructions to the isomorphism of Floer homology under geometric composition. 

\begin{remark}[Genearlized Floer isomorphisms under geometric composition]  \rm\label{rmk:LLBS}
The isomorphism $L_{12} \# L_{23} \sim {L_{12} \circ L_{23}}$ should generalize directly to exact noncompact settings as long as the Lagrangians have a conical structure near infinity that allows one to use maximum principles to guarantee compactness. An application to the construction of a Floer field theory that extends the link invariants \cite{SS} was proposed in \cite{reza} but unfortunately seems to be lacking this conical structure. 

On the other hand, extensions of this isomorphism to negative monotone settings announced in \cite{ll}  overlooked obstructions arising from Morse-Bott trajectories.\footnote{
The published version of \cite{ll} erroneously claims in Lemmas 11, 14 that ``bubbling at the Morse-Bott end'' is captured in topologically in terms of a pair of disks with boundary on $L_{01}$ and $L_{12}$. 
This is generally false, with the first mistake being that a resolution of the $L_{02}=L_{01}\circ L_{12}$ seam yields a quilt with seam conditions in the order $(L_{01},L_{12},L_{12}^T,L_{01}^T)$ instead of $(L_{01},L_{12},L_{01},L_{12})$ which in fact is nonsensical unless $M_0=M_2$. Second, one should note that folding of the quilt indeed yields a strip in ``$\underline M = M_0\times M_1\times M_1\times M_2$'' with both boundary conditions given by Lagrangian embeddings of $L_{01}\times L_{12}$, but the correct symplectic structure on $\underline M$ is $(\omega_0,-\omega_1,-\omega_1,\omega_2)$, and the two embeddings differ by a permutation of the two $M_1$ factors. 
Finally, even if torsion assumptions would allow to deform a Morse-Bott trajectory of positive symplectic area into a sum of disk classes, these are generally no longer pseudoholomorphic. Now one should note that $a+b>0\not\Rightarrow a,b>0$, so that even in the torsion case it does not suffice to exclude disks of positive area. 

Thus the published arguments in \cite{ll} are insufficient to exclude breaking at the Morse-Bott end in any case other than exactness, and its Theorem 3 -- the isomorphism in the new case of negative monotonicity -- is in the corrigendum only claimed under the additional assumption of ``absence of quantum contributions to the Morse-Bott differential'' as discussed below.
}
In fact, these obstructions are homotopically identical to the figure eight bubbles conjectured in \cite{ww:isom} and established in \cite{BW:squiggly}, so that true generalizations of this isomorphism are expected only from the compactification and Fredholm theory for figure eight bubbles in \cite{B:remsing,B:Fred,BW:squiggly}, towards capturing the obstructions algebraically. 
It is however worthwhile to discuss the approach by Matthias Schwarz which \cite{ll} attempted to implement: It is a geometric version of the ``local to global'' approach in Remark~\ref{rmk:localtoglobal-Lag}, which aims for an explicit construction of a direct homomorphism 
\begin{equation}\label{eq:schwarz}
HF( \ldots , L_{12} , L_{23},  \ldots ) \to HF( \ldots , L_{13} ,  \ldots )
\end{equation}
from a relative quilt invariant with canonical asymptotics at a Morse-Bott end for the cyclic 1-morphism $(L_{12} , L_{23}, L_{13}^T)$. 
This corresponds to a map $\beta \mapsto \Phi_{\cQ\cD}(\beta,\alpha)$ given by plugging a canonical element $\alpha$ into a quilted composition map 
$$
\Phi_{\cQ\cD}: HF( \ldots , L_{12} , L_{23},  \ldots ) \otimes HF( L_{12} , L_{23}, L_{13}^T) 
\to HF( \ldots , L_{13} , \ldots ) 
$$
that arises from the quilt diagram in Figure~\ref{fig:quiltdiagramforHFiso}.
\begin{figure}[!h]
\centering
\includegraphics[width=5.5in]{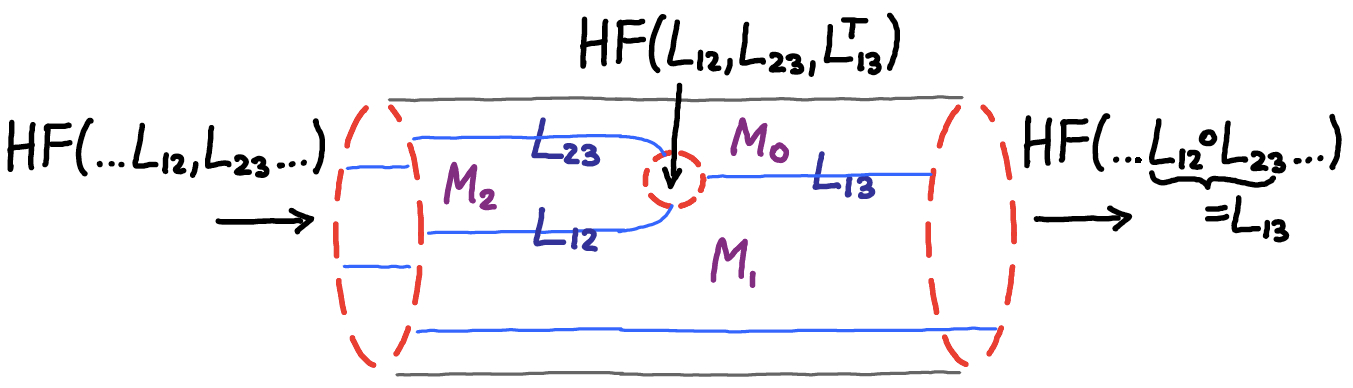}
\caption{
The isomorphism $HF( \ldots , L_{12} , L_{23},  \ldots ) \simeq HF( \ldots L_{12} \circ L_{23}  \ldots )$ in Remark~\ref{rmk:localtoglobal-Lag} arises from a quilted Floer homology class $\alpha \in HF(L_{12}, L_{23} , L_{13}^T)$ via the quilted composition map induced by the above quilt diagram.
}
\label{fig:quiltdiagramforHFiso}
\end{figure}

Using the classes $\alpha,\alpha^{-1}$ from Remark~\ref{rmk:localtoglobal-Lag}, we obtain an inverse $\gamma \mapsto \Phi_{\cQ\cD'}(\gamma,\alpha^{-1})$ to \eqref{eq:schwarz} from 
$$
\Phi_{\cQ\cD'}: HF( \ldots , L_{13},  \ldots ) \otimes HF( L_{13} , L_{23}^T, L_{12}^T) 
\to HF( \ldots ,L_{12} , L_{23} ,   \ldots ) ,
$$
a quilted composition map arising from another quilt diagram $\cQ\cD'$ that is obtained by reflecting $\cQ\cD$.
Indeed, $\cQ\cD,\cQ\cD'$ glue -- in two orders, one of which is shown in Figure~\ref{fig:quiltdiagram-calculation} -- to diagrams 
which also correspond to the gluing of quilt diagrams $\cQ\cD'', \cQ\cD'''$ with the string diagram for $\circ_{\rm v}$. Moreover, if in the latter gluings we replace the $\circ_{\rm v}$ diagram with the string diagram for the corresponding identity, then the glued diagram is a quilted cylinder.
\begin{figure}[!h]
\centering
\includegraphics[width=5.5in]{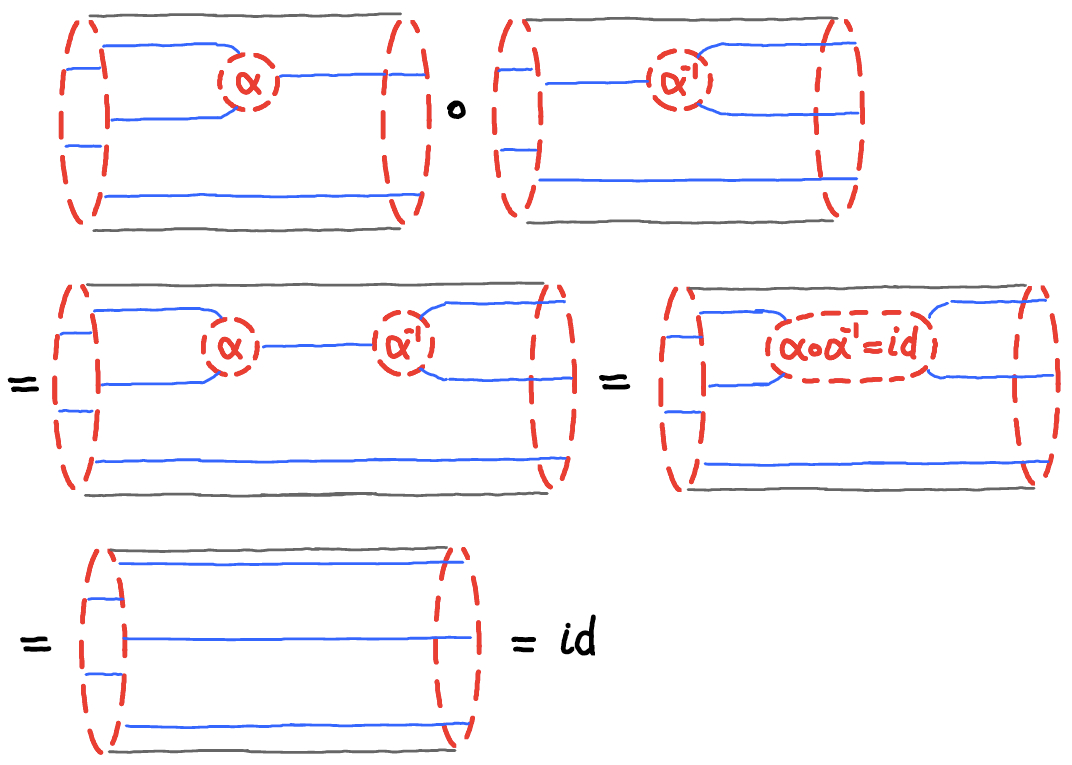}
\caption{
The identity 
$\Phi_{\cQ\cD'}( \Phi_{\cQ\cD}(\beta,\alpha),\alpha^{-1})=\beta$ follows from applying the gluing, deformation, and cylinder axioms for the quilted composition maps and the identity 
$\alpha\circ_{\rm v} \alpha^{-1} =  \id_{L_{12} \# L_{23}} $.
}
\label{fig:quiltdiagram-calculation}
\end{figure}
Now the gluing and cylinder axioms imply
\begin{align*}
\Phi_{\cQ\cD'}\bigl( \Phi_{\cQ\cD}(\beta,\alpha) , \alpha^{-1} \bigr) &= 
 \Phi_{\cQ\cD''}(\beta,\alpha \circ_{\rm v} \alpha^{-1} ) =  \Phi_{\cQ\cD''}(\beta, \id_{L_{12} \# L_{23}} ) = \beta , \\
 \Phi_{\cQ\cD}\bigl( \Phi_{\cQ\cD'}(\gamma,\alpha^{-1}) , \alpha \bigr) &= 
 \Phi_{\cQ\cD'''}(\gamma,\alpha^{-1} \circ_{\rm v} \alpha ) =  \Phi_{\cQ\cD'''}(\beta, \id_{L_{13}} ) = \gamma ,
\end{align*}
which proves that \eqref{eq:schwarz} is an isomorphism.
The Morse-Bott end amounts to an implicit construction of $\alpha:=[L_{13}]\in HF( L_{12} , L_{23}, L_{13}^T)$ from the fundamental class of $L_{13}$ and an identification of the chain groups (but not the differentials) which yield the Floer homology resp.\ the Morse homology of the Lagrangian intersection,
$$
CF( L_{12} , L_{23}, L_{13}^T) \simeq CM(\cap ( L_{12} , L_{23}, L_{13}^T) )  \simeq  CM (L_{13}) .
$$
The resulting Morse homology $HM (L_{13})$ is isomorphic to singular homology of $L_{13}$ since the Lagrangian intersection is diffeomorphic to $L_{13}= L_{12}\circ L_{23}$,
\begin{align*}
\cap ( L_{12} , L_{23}, L_{13}^T) &\;=\; ( L_{12}\times L_{23} \times L_{13}^T ) \cap (\Delta_{M_1}\times \Delta_{M_2}\times \Delta_{M_3})^T  \\
&\;\cong\; ( ( L_{12}\circ L_{23}) \times L_{13}^T ) \cap (\Delta_{M_1}\times \Delta_{M_3})^T  \;\cong\; L_{13} .
\end{align*}
Schwarz proposed to prove that \eqref{eq:schwarz} is an isomorphism by arguing that it has ``upper triangular form'', but also observed that the crucial step is to construct a chain map in the first place that induces \eqref{eq:schwarz}, which amounts to showing that $\alpha=[L_{13}]$ lies in the kernel of the Floer differential.
The only cases beyond the monotone case covered in \cite{ww:isom} in which this is claimed to be known at this point (after correction of \cite{ll}) are those in which not just the generators but also the differentials of the chain complexes 
$CF(L_{13}, L_{23}^T, L_{12}^T)$,
$CF(L_{12},  L_{23}, L_{23}^T, L_{12}^T)$,
$CF(L_{12},  L_{23}, L_{13}^T)$, 
$CF(L_{13}, L_{13})$
all agree with the Morse chain complex $CM(L_{13})$. 
In other words, we assume absence of quantum contributions to the Floer-Bott differential.\footnote{
In the corrigendum to \cite{ll}, this assumption is misleadingly labeled ``additional monotonicity''. 
While (exact / positive / negative) monotonicity assumptions for the relevant quilted Floer cylinders also had to be added, the crucial extra assumption in the negative monotone case is that solutions of positive energy (i.e.\ possible quantum differentials) have sufficiently negative Fredholm index -- exactly such that their occurrence in the Floer-Bott differential can be excluded by transversality. 
}
On the one hand, this allows one to define $\alpha:=[L_{13}]\in HF( L_{12} , L_{23}, L_{13}^T)$ and thus obtain a homomorphism \eqref{eq:schwarz}. On the other hand, this also completes the two previous algebraic arguments for the isomorphism in simple ways that require neither \cite{ww:isom} nor \cite{ll}.
In Remark~\ref{rmk:localtoglobal-Lag}, an absence of quantum differentials yields the required identification (compatible with composition and identities) of Floer homologies (in a Morse-Bott setup in which only Morse trajectories contribute)
$HF(L_{13}, L_{23}^T, L_{12}^T)\simeq 
HF(L_{12},  L_{23}, L_{23}^T, L_{12}^T)
\simeq HF(L_{12},  L_{23}, L_{13}^T)
\simeq HF(L_{13}, L_{13})$.
In the above construction of a direct homomorphism \eqref{eq:schwarz}, the absence of quantum differentials yields $\alpha$ (as above) and $\alpha^{-1}:=[L_{13}]\in HF(L_{13}, L_{23}^T, L_{12}^T)$ so that 
$\alpha\circ_{\rm v} \alpha^{-1} = [L_{13}] \cap [L_{13}] = [L_{13}]= \id_{L_{12} \# L_{23}}$ and
$\alpha^{-1}\circ_{\rm v} \alpha = [L_{13}] \cap [L_{13}] = [L_{13}]=  \id_{L_{13}}$.

The bottom line is that we do not get around proving -- implicitly or explicitly -- the isomorphism $L_{12} \# L_{23} \sim L_{12} \circ L_{23}$ as 1-morphisms in the symplectic 2-category.
\end{remark}

\subsection{Gauge theoretic 2-categories and quilted Atiyah Floer conjectures} \label{ss:gauge2}

This section takes the quilt diagram approach in the construction of the symplectic 2-category and applies it to the gauge theoretic ASD Yang-Mills PDE to obtain proposals for various 2-categories which mix gauge theoretic, symplectic, and topological data.
On the one hand, this categorical framework allows us to rigorously apply the abstract "local to global" approach of Remark~\ref{rmk:localtoglobal} to Atiyah-Floer type conjectures, as already sketched in Remark~\ref{rmk:SO3}; also see \eqref{eq:AFinfdim}. On the other hand, it yields various approaches to constructing 2+1+1 Floer-type field theories $\Bor^{\rm conn}_{2+1+1}\to\Cat$, which in turn leads us to formulate quilted Atiyah-Floer conjectures relating them.

Throughout, we fix a compact Lie group G and should also fix bundle types via characteristic classes. Ideally, this would avoid reducible connections as in the case of nontrivial $SO(3)$-bundles over 3-manifolds. However, this cannot generally be achieved in a coherent fashion when manifolds are decomposed to yield a field theory.
Moreover, the Donaldson invariants of 4-manifolds \cite{Don:inv,DK} -- defined for $G=SU(2)$ or $G=SO(3)$ -- successfully deal with reducibles by encoding them as ends of the ASD moduli spaces, which yields a polynomial structure.
On the other hand, instanton Floer homology for 3-manifolds \cite{Floer:inst} is currently only constructed in the absence of reducibles (using trivial $SU(2)$-bundles over homology 3-spheres, or nontrivial $SO(3)$-bundles), and thought to require an equivariant theory to deal with reducibles. For the following we will assume that such theories can be constructed from the same ASD moduli spaces.
Then the 3+1 field theory outlined by Donaldson \cite{Don:book} for 4-cobordisms between appropriate 3-manifolds should have a refinement to 2+1+1 dimensions which can be cast as the following 2-category.

\begin{example} \label{ex:Donaldson} \rm 
The {\bf Donaldson 2+1 bordism 2-category ${\rm\mathbf{DBor}}$} should consist of:
\begin{itemize}
\item
Objects in $\Obj_{{\rm DBor}}:=\Obj_{\Bor_{2+1+1}}$ are closed oriented surfaces $\Sigma$.
\item
1-morphisms in $\Mor^1_{{\rm DBor}}(\Sigma_+,\Sigma_-):=\Mor^1_{\Bor_{2+1+1}}(\Sigma_+,\Sigma_-)$ are 3-cobordism $Y$ with boundary collars $\iota^\pm:[0,1]\times\Sigma_\pm\to Y$ as in Example~\ref{ex:1epsbor}.
\item
Horizontal 1-composition is gluing
$Y_{01} \circ_{\rm h} Y_{12} :=\bigl(Y_{01}\sqcup Y_{12}\bigr)/ {\iota_{01}^+(s,x)\sim \iota^-_{12}(s,x)}$
as in Example~\ref{ex:1epsbor}.
\item
2-morphisms in $\Mor_{\rm Dbor}^2(Y,Y'):= HF_{\rm inst}(\#(Y,Y'))$ for $Y,Y'\in\Mor^1_{{\rm DBor}}(\Sigma_+,\Sigma_-)$ are instanton Floer homology classes constructed analogous to \cite{Floer:inst,Don:book} on the closed 3-manifold $\#(Y,Y'):= \bigl(Y^-\sqcup Y'\bigr)/ {\iota_Y^\pm\sim \iota^\pm_{Y'}}$ obtained by reversing the orientation of $Y$ and gluing at both incoming and outgoing boundaries.
\item
Vertical and horizontal composition of 2-morphisms and $\id_{Y}\in\Mor_{\Symp_G}^2(Y,Y)$
arise from the ASD moduli spaces representing the associated string diagrams; see below.
\end{itemize}
Here we associate to every quilt diagram 
$\cQ\cD=\bigl(\cQ,(\Sigma_P)_{P\in\cP_\cQ}, (Y_S)_{S\in\cS_\cQ}\bigr)$ an elliptic PDE as follows:
\begin{itemize}
\item
A patch $P$ labeled by a surface $\Sigma$ is represented by a connection $\Theta_P\in\cA(\widehat P\times\Sigma)$ satisfying the ASD equation $F_{\Theta_P}+*F_{\Theta_P}=0$.
Here $\widehat P$ is an oriented 2-manifold that covers the closure $\overline P\subset Q_2\less Q_0$ as in Remark~\ref{rmk:quilt}.
\item
A seam $S$ labeled by a 3-cobordism $Y_S\in\Bor_{2+1}(\Sigma_{P_S^-},\Sigma_{P_S^+})$ is represented 
by a connection $\Theta_S\in\cA(S\times Y)$ satisfying the ASD equation $F_{\Theta_S}+*F_{\Theta_S}=0$
and diagonal seam condition: The restrictions of the connections $\Theta_{P_S^-}, \Theta_{P_S^+}$ for the adjacent patches $P_S^\pm \in\cP_{\cQ'}$ to the boundary slices $\{s\}\times \Sigma_{P_S^\pm}$ for $s\in S$ are required to coincide with $\Theta_S|_{\{s\}\times \partial Y}$ over $\partial Y = \Sigma_{P_S^-}^- \sqcup \Sigma_{P_S^+}$ .
\end{itemize}
Up to challenges with reducibles, this should yield a quilted 2-category $\rm DBor$ with adjunction given by orientation reversal as in Lemma~\ref{lem:borquilt}.
Note in particular that the matching conditions for the connections at the seams, after applying an appropriate gauge transformation, simply become a smooth extension on the glued 4-manifolds. Thus the moduli space constructed here is the moduli space of ASD connections on the 4-manifold constructed in Lemma~\ref{lem:borquilt} from the quilt diagram $\bigl(\cQ,(\Sigma_P)_{P\in\cP_\cQ}, (Y_S)_{S\in\cS_\cQ}\bigr)$ in $\Bor_{2+1+1}$.
We may replace the boundary and corners of this 4-manifold at each end $e\in\cE^\pm_\cQ$ by a cylindrical end over the 3-manifold $\#(\ul Y_e):= \bigl(\bigsqcup_{i\in\Z_{N_e}} Y_{S_i}\bigr)/{\io_{Y_{S_i}}^+\sim \io_{Y_{S_{i+1}}}^-}$ that is obtained by gluing the components of the cyclic 1-morphism $\underline{Y}_e=( Y_{S_i} )_{i\in\Z_{N_e}}$ associated to this end. 
Then the quilted composition map should be given by the
relative Donaldson invariant for the resulting 4-manifold with cylindrical ends, 
$$
\Phi_{\cQ\cD}:\otimes_{e\in\cE_\cQ^-} \Mor^2_\cC(\ul Y_e)=HF(\#(\ul Y_e)) \to HF(\#(\ul Y_{e^+})) = \Mor^2_\cC(\ul Y_{e^+}) .
$$
\end{example}

Next, we build an analogous target 2-category for the infinite dimensional Floer field theory outlined in Example~\ref{ex:infdim}. In order to obtain differentiable structures this requires the choice of an integrability constant $p>2$.

\begin{example} \label{ex:inst2cat} \rm 
The {\bf symplectic instanton 2-category ${\rm\mathbf{SIn}}$} should consist of:
\begin{itemize}
\item
Objects are closed, oriented surfaces $\Sigma$ -- thought to represent the symplectic Banach space 
$\cA(\Sigma)=L^p(\Sigma,{\rm T}^*\Sigma\otimes\cg)$ of trivial G-connections on $\Sigma$.
\item
1-morphisms $\underline{\cL}=(\cL_{01},\ldots,\cL_{(k-1)k})\in\Mor^1_{\rm SIn}(\Sigma,\Sigma')$ are chains of Lagrangian submanifolds $\cL_{ij}\subset \cA(\Sigma_i)^-\times\cA(\Sigma_j) = \cA(\Sigma_i^-\sqcup\Sigma_j)$ in the symplectic spaces of connections over a chain of surfaces $\Sigma=\Sigma_0,\Sigma_1,\ldots,\Sigma_k=\Sigma'$, which are gauge invariant, $\cG(\Sigma_i^-\sqcup\Sigma_j)^*\cL_{ij} = \cL_{ij}$.
\item
Horizontal composition of 1-morphisms
$\underline{\cL}\circ_{\rm h}\underline{\cL}' := \underline{\cL} \# \underline{\cL}'$
is defined by concatenation
$(\ldots,\cL_{(k-1)k}) \# (\cL'_{01},\ldots) := (\ldots,\cL_{(k-1)k},\cL'_{01},\ldots)$, 
with identities $1_\Sigma= (\;) \in\Mor_{\rm SIn}^1(\Sigma,\Sigma)$ for $\circ_{\rm h}$ given by 
trivial chains.
\item
2-morphisms in $\Mor_{\rm SIn}^2(\underline{\cL},\underline{\cL}'):= HF_{\rm inst}(\underline{\cL},\underline{\cL}')$  are the elements of quilted instanton Floer homology groups \cite{SW} outlined below.
\item
Vertical and horizontal composition of 2-morphisms and $\id_{\underline{\cL}}\in\Mor_{\rm SIn}^2(\underline{\cL},\underline{\cL})$
arise from ASD quilts representing the associated string diagrams; see below.
\end{itemize}
Up to challenges with reducibles, this approach should yield a quilted 2-category with adjunction given by transposition analogous to Remark~\ref{rmk:sympquilt}.
As in Remark~\ref{rmk:quilts} the key step is to associate elliptic PDEs to quilt diagrams 
$\cQ\cD=\bigl(\cQ,(\Sigma_P)_{P\in\cP_\cQ}, (\underline{\cL}_S)_{S\in\cS_\cQ}\bigr)$.
For that purpose we first replace any seam labeled by a sequence of Lagrangians ${\cL_{i(i+1)}\subset \cA(\Sigma_i)^- \times \cA(\Sigma_{i+1})}$ with strips resp.\ annuli labeled by the $\Sigma_i$ and seams between them labeled by the simple Lagrangians $\cL_{i(i+1)}$. 
Then the new quilt diagram $\cQ\cD'=(\cQ',\ldots )$ determines a moduli space of {\bf ASD quilts} $(\Theta_P)_{P\in\cP_{\cQ'}}$, which satisfy the following PDE:
\begin{itemize}
\item
A patch $P$ labeled by a surface $\Sigma$ is represented by a connection $\Theta_P\in\cA(\widehat P\times\Sigma)$ satisfying the ASD equation $F_{\Theta_P}+*F_{\Theta_P}=0$.
\item
A seam $S$ labeled by a Lagrangian submanifold $\cL_S\subset \cA(\Sigma_{P_S^-})^- \times \cA(\Sigma_{P_S^+})$ is represented by a Lagrangian seam condition: The restrictions of the connections $\Theta_{P_S^-}, \Theta_{P_S^+}$ for the adjacent patches $P_S^\pm \in\cP_{\cQ'}$ to the boundary slices $\{s\}\times \Sigma_{P_S^\pm}$ for $s\in S$ induce connections $(\Theta_{P_S^-}, \Theta_{P_S^+})_s \in \cA(\Sigma_{P_S^-})^- \times \cA(\Sigma_{P_S^+})$, which are required to lie in $\cL_S$.
\end{itemize}
These moduli spaces can be given compactifications and Fredholm descriptions by the nonlinear elliptic analysis for ASD connections with Lagrangian boundary conditions that is developed in \cite{W:ell1, W:ell2} for 4-manifolds with boundary space-time splitting such as $\bigsqcup_{P\in\cP_{\cQ'}} P\times \Sigma_P$. They should hence induce quilted composition maps
$$
\Phi_{\cQ\cD}:\otimes_{e\in\cE_\cQ^-} \Mor^2_{\rm SIn}(\ul \cL_e) \to \Mor^2_{\rm SIn}(\ul \cL_{e^+}) , 
$$
where the cyclic 2-morphism spaces $\Mor^2_{\rm SIn}(\ul\cL)= HF_{\rm inst}(\ul\cL)$ for cyclic 1-morphisms $\ul \cL =(\cL_{i(i+1)})_{i\in\Z_N}$ are defined to be the quilted instanton Floer homology. The latter is the homology of a Floer complex whose differential arises from moduli spaces of ASD quilts on a quilted cylinder with seam conditions in the $\cL_{i(i+1)}$ (modulo an overall $\R$-shift).
This also defines the usual 2-morphisms of pairs of 1-morphisms $\ul\cL=( \cL_{(i-1)i})_{i=1\ldots k}, \ul\cL'=( \cL'_{(i-1)i})_{i=1\ldots k'}\in \Mor^1_{\rm SIn}(\Sigma,\Sigma')$,
$$
\Mor^2_{\rm SIn}(\ul\cL,\ul\cL') := HF_{\rm inst}(\underline{\cL},\underline{\cL}')
:= HF_{\rm inst}( \#(\underline{\cL},\underline{\cL}')),
$$ 
by concatenation to a cyclic 1-morphism indexed by $\Z_{k+k'}$
$$
\#(\underline{\cL},\underline{\cL}'):= \bigl( \cL_{(k-1)k}^T, \ldots,\cL_{01}^T , \cL'_{01},\ldots, \cL'_{(k'-1)k'} \bigr) .
$$ 
A first case of instanton Floer theory with Lagrangian boundary conditions is developed in \cite{SW} to construct a Floer homology $HF_{\rm inst}(Y,\cL)$ for pairs of a 3-manifold $Y$ with boundary and a Lagrangian $\cL\subset\cA(\partial Y)$, using ASD connections $\Theta\in\cA(\R\times Y)$ with boundary conditions $\Theta|_{\{s\}\times\partial Y}\in\cL\;\forall s\in\R$.
It requires an exclusion of nontrivial reducible connections, as is guaranteed for pairs $(Y,\cL_H)$ when $Y\cup_\Sigma H$ is a homology 3-sphere (i.e.\ has the same homology with integer coefficients). For more general pairs, an equivariant Floer homology would be required to deal with the reducible flat connections on $Y$ with boundary restriction in $\cL$. 
Apart from dealing with the reducibles, the quilted setup above with $\cL_{i(i+1)}\subset\cA(\Sigma_i)^-\times\cA(\Sigma_{i+1})$ can be reformulated as the instanton Floer theory with Lagrangian boundary conditions 
$HF_{\rm inst}(\ul\cL) = HF_{\rm inst}\bigl( \sqcup_{i=0}^{N-1} [0,1]\times \Sigma_i , (\cL_{i(i+1)})_{i\in\Z_N} \bigr)$.

For the quilted Atiyah-Floer Conjectures \ref{con:qAF} we will restrict to the topologically generated part of this 2-category: 
The {\bf symplectic instanton 2+1 bordism 2-category ${\rm\mathbf{SIn}}_{\mathbf{2+1}}$} is given as above with the restriction that the gauge invariant Lagrangian submanifolds $\cL_{ij}=\cL(Y_{ij})\subset \cA(\Sigma_i)^-\times\cA(\Sigma_j) = \cA(\Sigma_i^-\sqcup\Sigma_j)$ must be those associated in Example~\ref{ex:infdim} to handle attachments $Y_{ij}\in\Mor^1_{\Bor_{2+1+1}}(\Sigma_i,\Sigma_j)$.
\end{example}

The analogous 2-category associated to the finite dimensional Floer field theory arising from G-representation spaces as outlined in Example~\ref{ex:rep} is a restriction of the symplectic 2-category as follows.

\begin{example} \label{ex:21bord} \rm
The {\bf symplectic 2+1 bordism 2-category of G-representations}, ${\rm\mathbf{Symp}}^{\mathbf G}_{\mathbf{2+1}}$, is given by the topologically generated part of $\Symp$ in Example~\ref{ex:2symp}, as follows.
\begin{itemize}
\item
Objects are the symplectic representation spaces $M_\Sigma$ of Example~\ref{ex:rep}.
\item
1-morphisms $\uL=(L_{01},\ldots,L_{(k-1)k})\in\Mor^1_{\Symp^G_{2+1}}(M_\Sigma,M_{\Sigma'})$ are chains of Lagrangian submanifolds $L_{ij}=L_{Y_{ij}}\subset M_{\Sigma_i}^-\times M_{\Sigma_j}$ which were associated in Example~\ref{ex:rep} to handle attachments $Y_{ij}\in\Mor^1_{\Bor_{2+1+1}}(\Sigma_i,\Sigma_j)$.
\item
Horizontal composition of 1-morphisms $\uL\circ_{\rm h}\uL' := \uL \# \uL'$ is concatenation as in $\Symp$, 
with identities $1_{M_\Sigma}= (\;) \in\Mor_{\Symp^G_{2+1}}^1(M_\Sigma,M_\Sigma)$ given by 
trivial chains.
\item
2-morphism spaces $\Mor_{\Symp^G_{2+1}}^2(\uL,\uL'):= HF(\uL,\uL')$ are quilted Floer homology.
\item
Vertical and horizontal composition of 2-morphisms and $\id_{\uL}\in\Mor_{\Symp^G_{2+1}}^2(\uL,\uL)$
arise from pseudoholomorphic quilts representing the associated string diagrams as in Remark~\ref{rmk:quilts}.
\end{itemize}
Here the challenge of reducibles is more severe since it leads to singular symplectic spaces $M_\Sigma$. However, working for example with nontrivial bundles as in Remark~\ref{rmk:rigrep} yields a quilted 2-category with adjunction given by transposition as in Remark~\ref{rmk:sympquilt} -- as a subcategory of the monotone symplectic 2-category $\Symp^\tau$.
\end{example}

Along with these three gauge theoretic 2-categories $\cC= \rm DBor, \rm SIn_{2+1}, \Symp^G_{2+1}$ we have
three proposals for Floer field theories via functors $\Bor_{2+1}\to \cC$, to the 1-category level of one of these 2-categories:
\begin{itemize}
\item $\Bor_{2+1}\to\Symp^G_{2+1}$ is determined by the G-representation spaces $\Sigma \mapsto M_\Sigma$ and $Y\mapsto L_Y$ in Example~\ref{ex:rep}.
\item
$\Bor_{2+1}\to{\rm SIn}_{2+1}$ is determined by the G-connection spaces $\Sigma \mapsto \cA(\Sigma)$ and $Y\mapsto \cL_Y$  in Example~\ref{ex:infdim} .
\item 
$\Bor_{2+1}\to{\rm DBor}_{2+1}$ is determined by $\Sigma \mapsto \Sigma$ and $Y\mapsto Y$.
\end{itemize}
The Floer field theory extension principle outlined in Conjecture~\ref{con:fftext} should yield natural extensions $\Bor_{2+1+1}\to\cC\to\Cat$ for each of these, after making coherent choices of bundles as in Remark~\ref{rmk:rigrep} or otherwise resolving the challenge of reducibles. (This may require a restriction to the connected bordism category.)
Now the natural extension of the Atiyah-Floer Conjectures~\ref{con:af} and \ref{con:afc} for Heegaard splittings and cyclic Cerf decompositions is the following for Donaldson theory.

\begin{conjecture}[Quilted Atiyah-Floer conjecture] \label{con:qAF}
The three extended Floer field theories 
$\Bor_{2+1+1}\to\Symp^G_{2+1}$,
$\Bor_{2+1+1}\to{\rm SIn}_{2+1}$, 
$\Bor_{2+1+1}\to{\rm DBor}_{2+1}$
arising from appropriate G-bundles
induce isomorphic 2-functors $\Bor_{2+1+1}\to\Cat$.
\end{conjecture}

Note here that the last extended Floer field theory $\Bor_{2+1+1}\to{\rm DBor}_{2+1}\to\Cat$ should comprise the Donaldson invariants of 4-manifolds, as is visible not from the essentially trivial generating Floer field theory $\Bor_{2+1}\to{\rm DBor}_{2+1}$, but from the 2-morphism level of the target 2-category ${\rm DBor}_{2+1}$.
Analogous conjectures can be made for Seiberg-Witten theory and have been partially proven in \cite{KLT}.

\begin{remark}[Quilted Atiyah-Floer conjecture for Seiberg-Witten-Heegaard-Floer theory] \label{rmk:qAF} \rm
Example \ref{ex:sym} should induce a Floer field theory
$\Bor_{2+1}\to\Symp^{\rm symm}_{2+1}$ 
to a subcategory of $\Symp^\tau$; 
both being generated by symmetric products $\Sigma \mapsto {\rm Sym}^{g+n}\Sigma$ and $Y_\alpha\mapsto L_\alpha$.
Another Floer feld theory
$\Bor_{2+1}\to{\rm SWBor}_{2+1}$ should arise from the partial functor $\Sigma \mapsto \Sigma$ and $Y\mapsto Y$ to a 2-category defined as in Example~\ref{ex:Donaldson}, with instanton Floer theory resp.\ Donaldson invariants replaced by monopole Floer theory resp.\ Seiberg-Witten invariants.

Now the two resulting extended Floer field theories 
$\Bor_{2+1+1}\to\Symp^{\rm symm}_{2+1}$ and
$\Bor_{2+1+1}\to{\rm SWBor}_{2+1}$ should
induce isomorphic 2-functors $\Bor_{2+1+1}\to\Cat$.
\end{remark}

Finally, we will outline two further 2-categories which will serve to compare the above 2-categories and related Floer field theories, by embedding both into a ``convex span''. The first will serve to compare $\Don$ with $\rm SIn$.

\begin{example} \label{ex:enhanced} \rm 
The {\bf instanton Atiyah-Floer 2-category ${\rm\mathbf{InAF}}$} should consist of:
\begin{itemize}
\item
Objects are closed, oriented surfaces $\Sigma$.
\item
$\Mor_{\rm InAF}(\Sigma,\Sigma')$ combines $\Mor_{\rm SIn}(\Sigma,\Sigma')$ and 
$\Mor_{\Don}(\Sigma,\Sigma')$ as follows:
\begin{itemize}
\item
1-morphisms are chains $\ul f = (f_{i(i+1)})_{i=0,\ldots k-1}$ of morphisms between surfaces $\Sigma=\Sigma_0,\Sigma_1,\ldots,\Sigma_k=\Sigma'$, where for each $i=0,\ldots {k-1}$ we either have $f_{i(i+1)}=\cL_{i(i+1)}\in \Mor^1_{\rm SIn}(\Sigma_i,\Sigma_{i+1})$ a gauge invariant Lagrangian submanifold of $\cA(\Sigma_i)^-\times\cA(\Sigma_{i+1})$, or $f_{i(i+1)}=Y_{i(i+1)}\in \Mor^1_{\Don}(\Sigma_i,\Sigma_{i+1})$
a  3-cobordism.
\item
$\Mor_{\rm InAF}^2(\ul f,\ul g):= HF_{\rm inst}(\ul f, \ul g)$ is the quilted instanton Floer homology group
arising from quilted cylinders with seams labeled by the entries of $\ul f, \ul g$.
\item
Vertical composition $\circ_{\rm v}$ and its identities $\id_{\ul f}\in\Mor_{\rm InAF}^2(\ul f, \ul f)$
arise from moduli spaces of ASD quilts representing the associated string diagrams.
\end{itemize}
\item
The composition functor $\Mor_{\rm InAF}(\Sigma,\Sigma')\times \Mor_{\rm InAF}(\Sigma',\Sigma'')\to \Mor_{\rm InAF}(\Sigma,\Sigma'')$ is defined by concatenation $\ul f \circ_{\rm h}\ul f' := \ul f\# \ul f'$ on 1-morphisms, 
with identities given by trivial chains $1_\Sigma= (\;) \in\Mor_{\rm InAF}^1(\Sigma,\Sigma)$, and horizontal 2-composition arises from moduli spaces of ASD quilts representing the associated string diagram.
\end{itemize}
The quilt diagrams here are represented by the same moduli spaces of ASD quilts as in Example~\ref{ex:inst2cat}, where as in Example~\ref{ex:Donaldson} a seam $S$ labeled by a 3-cobordism $Y_S$ represents an ASD connection $\Theta_S\in\cA(S\times Y)$ that matches (slice-wise, or completely after gauge) with the restrictions of the connections $\Theta_{P_S^-}, \Theta_{P_S^+}$ for the adjacent patches. 
Up to challenges with reducibles, this should yield a quilted 2-category with adjunction given by transposition as in Remark~\ref{rmk:sympquilt}.
\end{example}

The final outline of a 2-category is the ``convex span'' of $\Don$ with $\Symp^G_{2+1}$, after which one can easily imagine a combination of $\rm SIn$ with $\Symp^G_{2+1}$, or a 2-category comprising all three of the basic gauge theoretic 2-categories $\Don$, $\rm SIn$, $\Symp^G_{2+1}$.

\begin{example} \label{ex:superenhanced} \rm 
The {\bf Atiyah-Floer 2-category ${\rm\mathbf{AtFl}}$} roughly consists of the following.
\begin{itemize}
\item
Objects in $\Obj_{\rm AtFl} = \Obj_{\Don} \cup \Obj_{\Symp^G_{2+1}}$ are either closed, oriented surfaces $\Sigma$ or symplectic representation spaces $M_\Sigma$ associated to a surface as in Example~\ref{ex:rep}.
\item
$\Mor_{\rm AtFl}(\Sigma,\Sigma')$ extends $\Mor_{\Symp^G_{2+1}}(M_{\Sigma},M_{\Sigma'})$ and 
$\Mor_{\Don}(\Sigma,\Sigma')$ as follows:
\begin{itemize}
\item
Simple morphisms all arise from 3-cobordisms $Y\in \Mor^1_{\Bor_{2+1+1}}(\Sigma,\Sigma')$, but depending on the type of objects they relate, they appear as
\begin{itemize}
\item 3-cobordisms $Y \in \SMor_{\rm AtFl}(\Sigma,\Sigma') := \Mor^1_{\Don}(\Sigma,\Sigma')$, 
\item
Lagrangians 
$L_Y \in \SMor_{\rm AtFl}(M_{\Sigma},M_{\Sigma'}) := \Mor^1_{\Symp^G_{2+1}}(M_{\Sigma},M_{\Sigma'})$, 
\item
Lagrangians
$\cL_Y/\cG(\Sigma') \subset \cA(\Sigma)^-\times M_{\Sigma'}$
in $\SMor_{\rm AtFl}(\Sigma,M_{\Sigma'})$,
\item
Lagrangians
$\cL_Y/\cG(\Sigma) \subset M_{\Sigma}^-\times \cA(\Sigma')$
in $\SMor_{\rm AtFl}(M_\Sigma,\Sigma')$.
\end{itemize}
Here the three types of Lagrangians are only associated to handle attachments $Y$, and the last two types are projections of the gauge invariant Lagrangian $\cL_Y\subset \cA(\Sigma)\times \cA(\Sigma')$ from Example~\ref{ex:infdim}, which by construction lies in the flat connections, $\cL\subset \cA_{\rm flat}(\Sigma)\times \cA_{\rm flat}(\Sigma')$, so that quotienting by the gauge group on the first factor yields a projection to the representation space $M_\Sigma= \cA_{\rm flat}(\Sigma)/\cG(\Sigma)$. The same goes for projection in the second factor, and projecting in both factors yields $L_Y= \cL_Y/(\cG(\Sigma)\times \cG(\Sigma'))$.
\item
1-morphisms are chains $\ul f = (f_{i(i+1)})_{i=0,\ldots k-1}$ which consist of simple morphisms 
$f_{i(i+1)}\in \SMor_{\rm AtFl}(x_i,x_{i+1})$ for a sequence $x_0,\ldots, x_{k-1}\in \Obj_{\rm AtFl}$ of objects, 
i.e.\ a sequence in which each entry is of one of the types $Y_{i(i+1)}$, $L_{Y_{i(i+1)}}$, $\cL_{Y_{i(i+1)}}/\cG(\Sigma_i)$, $\cL_{Y_{i(i+1)}}/\cG(\Sigma_{i+1})$ for a chain of 1-morphisms $Y_{i(i+1)}\in\Mor^1_{\Bor_{2+1+1}}(\Sigma_i,\Sigma_{i+1})$. 
\item
$\Mor_{\rm AtFl}^2(\ul f,\ul g):= HF_{\rm inst}(\ul f, \ul g)$ is the quilted instanton Floer homology group
arising from quilted cylinders with seams labeled by the entries of $\ul f, \ul g$.
\item
Vertical composition $\circ_{\rm v}$ and its identities $\id_{\ul f}\in\Mor_{\rm AtFl}^2(\ul f, \ul f)$
arise from moduli spaces of ASD quilts representing the associated string diagrams.
\end{itemize}
\item
The composition functor $\Mor_{\rm AtFl}(x,x')\times \Mor_{\rm AtFl}(x',x'')\to \Mor_{\rm AtFl}(x,x'')$ is defined by concatenation $\ul f \circ_{\rm h}\ul f' := \ul f\# \ul f'$ on 1-morphisms, 
with identities given by trivial chains $1_\Sigma= (\;) \in\Mor_{\rm AtFl}^1(x,x)$, and horizontal 2-composition arises from moduli spaces of ASD quilts representing the associated string diagram.
\end{itemize}
The quilt diagrams here are represented by a coupling of the pseudoholomorphic and ASD moduli spaces in Examples~\ref{ex:2symp} and \ref{ex:Donaldson} via Lagrangian seam conditions similar to those in Example~\ref{ex:inst2cat}.
Combining the PDE representations from those constructions, it remains to give PDE meaning to seams labeled with simple morphisms in $\SMor_{\rm AtFl}(M_\Sigma,\Sigma')$ or $\SMor_{\rm AtFl}(\Sigma,M_{\Sigma'})$.
Since these are related by transposition, it suffices to consider the first:
\begin{itemize}
\item
A seam $S$ with adjacent patches $P_S^\pm \in\cP_{\cQ'}$ labeled by a Lagrangian submanifold $\cL_{Y_S}/\cG(\Sigma_{P_S^-}) \subset M_{\Sigma_{P_S^-}}^-\times \cA(\Sigma_{P_S^+})$ is represented by a seam condition between the
pseudoholomorphic map $u_{P_S^-}:\widehat{P_S^-} \to M_{\Sigma_{P_S^-}}$ and the ASD connection $\Theta_{P_S^+}\in\cA(\widehat{P_S^+}\times\Sigma_{P_S^+})$.
Their restrictions to the seam induce a map \\
$\displaystyle \phantom{yes} \qquad\qquad
S\to M_{\Sigma_{P_S^-}}^- \times \cA(\Sigma_{P_S^+}), \qquad
s\mapsto (u_{P_S^-}(s), \Theta_{P_S^+}|_{\{s\}\times \Sigma_{P_S^+}}),
$\\
which is required to take values in $\cL_{Y_S}/\cG(\Sigma_{P_S^-})$.
\end{itemize}
For this to rigorously define a quilted 2-category with adjunction given by transposition as in Remark~\ref{rmk:sympquilt}, one again has to resolve the challenge of reducibles by e.g.\ working with nontrivial bundles as in Remark~\ref{rmk:rigrep}.
Once the symplectic spaces $M_\Sigma$ are all smooth, the analytic setup for ASD connections with Lagrangian boundary conditions in Example~\ref{ex:inst2cat} directly transfers to prove the basic Fredholm and compactness properties for these moduli spaces.\footnote{
The only missing pieces are regularity and estimates for the pair $(u,A)$ near a seam $S$ labeled by a Lagrangian $\cL/\cG(\Sigma) \subset M_{\Sigma}^-\times \cA(\Sigma')$. 
This could be achieved by the analytic setup developed in \cite{W:Banach,W:ell1,W:ell2}, which proceeds by splitting $\Theta_P=\Phi {\rm d}s + \Psi {\rm d}t + A$ into functions $\Phi,\Psi:P\to \cg$ in a neighbourhood $U=\{(s,t)|t\ge 0\}$ of the seam $S=\{t=0\}$, and a map $A:U\to\cA(\Sigma')$. Coulomb gauge fixing conditions include a Dirichlet boundary condition $\Psi|_{t=0}=0$ and -- via the flatness part of the Lagrangian boundary condition $F_A|_{t=0}=0$ -- induce a Neumann condition $\partial_t\Phi|_{t=0}=0$. 
Thus estimates for $\Phi,\Psi$ are obtained from Dirichlet resp.\ Neumann problems with lower order contributions from $A$. Once these are established, the map $A:U\to\cA(\Sigma')$ satisfies a  Cauchy-Riemann equation with respect to the complex structure given by the Hodge $*$ on $\Sigma$ and lower order or controlled inhomogeneous terms. Moreover, it is coupled with the pseudoholomorphic map $u^-:U\to M_{\Sigma}^-$ (obtained from $u$ by reflection on the seam) by a Lagrangian seam condition. Thus estimates on $(u,A)$ in a neighbourhood of the seam would follow from the general theory for the Cauchy-Riemann equation in Banach spaces \cite{W:Banach}.
}
An exposition of the compactness results in an explicitly quilted setting can be found in \cite{lipyanski}.
\end{example}

\section*{Acknowledgements}
I would like to credit and thank Denis Auroux, Chris Douglas, Dan Freed, David Gay, Robert Lipshitz, Tim Perutz, Dietmar Salamon, Chris Schommer-Priess, Peter Teichner, and Chris Woodward for illuminations of various aspects of the ideas presented here.
Moreover, thanks to the organizers and participants of the 2015 AWM Symposium for sharing an amazing breadth of quality math and inspiring me to try and make the Floer field theory ideas rigorous.
Finally, thanks to the diligent and speedy referees for help with cleaning up various details and encouragement to make known issues known.

\bibliographystyle{alpha}

\end{document}